\documentclass[reqno,11pt]{amsart}

\usepackage{amsmath}
\usepackage{amsthm}
\usepackage{amssymb}
\usepackage{amscd}
\usepackage{xypic}
\usepackage{verbatim}
\usepackage[plainpages=false,colorlinks,hyperindex,pdfpagemode=None,bookmarksopen,linkcolor=red,citecolor=blue,urlcolor=blue]{hyperref}
\usepackage{pdflscape}
\usepackage{stmaryrd}
\usepackage{mathabx}

\usepackage{yfonts} 
\usepackage{palatino}

\usepackage{bbm}

\def\beq{\begin{equation}}
\def\eeq{\end{equation}}

\swapnumbers
\theoremstyle{plain}
\newtheorem{theorem}[subsection]{Theorem}

\newtheorem*{theorem*}{Theorem}
\newtheorem{proposition}[subsection]{Proposition}
\newtheorem*{proposition*}{Proposition}
\newtheorem{lemma}[subsection]{Lemma}
\newtheorem*{lemma*}{Lemma}
\newtheorem*{fact*}{Fact}
\newtheorem{corollary}[subsection]{Corollary}
\newtheorem*{corollary*}{Corollary}

\theoremstyle{definition}

\theoremstyle{remark}
\newtheorem{remark}[subsection]{Remark}
\newtheorem*{remarks}{Remarks}

\newtheorem{example}[subsection]{Example}

\renewcommand{\comment}[1] {  }

\DeclareFontFamily{OT1}{rsfs}{}
\DeclareFontShape{OT1}{rsfs}{n}{it}{<-> rsfs10}{}
\DeclareMathAlphabet{\mathscr}{OT1}{rsfs}{n}{it}

\newcommand{\from}{\leftarrow}

\newcommand{\inv}{{\operatorname{inv}}}

\newcommand{\Res}{\mathrm{Res}}

\newcommand{\Z}{\mathbb{Z}}

\newcommand{\CC}{\mathbb{C}}

\newcommand{\RR}{\mathbb{R}}

\newcommand{\QQ}{\mathbb{Q}}

\newcommand{\unr}{{\operatorname{unr}}}
\newcommand{\ab}{{\operatorname{ab}}}

\newcommand{\Hom}{\operatorname{Hom}}
\newcommand{\End}{\operatorname{End}}
\newcommand{\Aut}{{\operatorname{Aut}}}

\newcommand{\varchi}{\mathcal{X}}
\newcommand{\Gm}{\mathbb{G}_m}
\newcommand{\Ga}{\mathbb{G}_a}

\newcommand{\GL}{\operatorname{GL}}
\newcommand{\Mat}{\operatorname{Mat}}

\newcommand{\PGL}{\operatorname{PGL}}

\newcommand{\SL}{\operatorname{SL}}

\newcommand{\SO}{{\operatorname{SO}}}

\newcommand{\spec}{\operatorname{spec}}
\newcommand{\Vol}{\operatorname{Vol}}
\newcommand{\diag}{{\operatorname{diag}}}

\newcommand{\disc}{{\operatorname{disc}}}
\newcommand{\temp}{{\operatorname{temp}}}
\newcommand{\sm}{{\operatorname{sm}}}
\newcommand{\cont}{{\operatorname{cont}}}

\newcommand{\Id}{\operatorname{Id}}

\newcommand{\cusp}{{\operatorname{cusp}}}
\newcommand{\noncusp}{{\operatorname{noncusp}}}

\newcommand{\rest}{{\operatorname{rest}}}

\newcommand{\Subunit}{{\operatorname{Subunit}}}
\newcommand{\sing}{{\operatorname{sing}}}
\newcommand{\Cent}{{\operatorname{Cent}}}

\newcommand{\St}{{\operatorname{St}}}

\begin{document}
\numberwithin{equation}{section}
\setcounter{tocdepth}{1}
\title{Paley--Wiener theorems for a $p$-adic spherical variety}

\author{Patrick Delorme}
\address[Patrick Delorme]{Aix Marseille Universit\'e\\
CNRS, Centrale Marseille, I2M, UMR 7373,\\
13453 Marseille, \\ France.}
\email{patrick.delorme@univ-amu.fr}

\author{Pascale Harinck}
\address[Pascale Harinck]{
CMLS, \'Ecole polytechnique,\\ CNRS, Universit\'e Paris-Saclay,\\ 91128 Palaiseau Cedex,\\ France.}
\email{pascale.harinck@polytechnique.edu}

\author{Yiannis Sakellaridis}
\address[Yiannis Sakellaridis]{Rutgers University - Newark \\
101 Warren Street \\
Smith Hall 216 \\
Newark, NJ 07102\\
USA.\\
 and  Department of Mathematics,
School of Applied Mathematical and Physical Sciences,
National Technical University of Athens,
Heroon Polytechneiou 9,
Zografou 15780, Greece.
}
\email{sakellar@rutgers.edu}

\begin{abstract}
Let $\mathcal S(X)$ be the Schwartz space of compactly supported smooth functions on the $p$-adic points of a spherical variety $X$, and let $\mathscr C(X)$ be the space of Harish-Chandra Schwartz functions. Under assumptions on the spherical variety, which are satisfied when it is symmetric, we prove Paley--Wiener theorems for the two spaces, characterizing them in terms of their spectral transforms. As a corollary, we get relative analogs of the smooth and tempered Bernstein centers --- rings of multipliers for $\mathcal S(X)$ and  $\mathscr C(X)$.  When $X=$ a reductive group, our theorem for $\mathscr C(X)$ specializes to the well-known theorem of Harish-Chandra, and our theorem for $\mathcal S(X)$ corresponds to a first step --- enough to recover the structure of the Bernstein center --- towards the well-known theorems of Bernstein \cite{BePadic} and Heiermann \cite{Heiermann}. 
\end{abstract}



\maketitle
\tableofcontents

\section{Introduction}

The goal of this paper is to characterize the spectral transform of the spaces of Schwartz (i.e.\ smooth, compactly supported) and Harish-Chandra Schwartz functions on the points of a homogeneous spherical variety over a $p$-adic field, { and produce rings of multipliers, that is, $G$-endomorphisms, which generalize the (tempered and smooth) Bernstein centers}. We do it under some assumptions on the variety, the main one being that the variety and its associated ``Levi varieties'' are ``factorizable'' --- this is a condition that allows one to continuously vary the central character of a representation appearing in the space of functions on the variety by multiplying by characters of the group.
This condition restricts us to a slightly larger setting than that of symmetric spaces. (In the non-symmetric case, there are some other conditions for which we have no general proof, and have to be checked ``by hand'' in each case; but we expect them to hold in general.) Our assumptions are explained in \S \ref{ssassumptions}, and the range of their validity is discussed in detail, including some examples, in Appendix \ref{app:factorizable}.

Let $X$ be a spherical variety for a group $G$ over a non-Archimedean local field $F$, satisfying those assumptions. We will be denoting $X(F)$ simply by $X$ (and similarly for other varieties), when this causes no confusion. We assume that $X=X(F)$ is endowed with a $G$-eigenmeasure, and normalize the action of $G$ on $L^2(X)$ (and other spaces of functions on $X$) so that it is unitary. The maximal split torus $\mathcal Z(X)$ of $G$-automorphisms of $X$ is the (split) \emph{center} of $X$. To $X$ one associates some ``simpler'' spherical $G$-spaces $X_\Theta$ with more symmetries, called the \emph{boundary degenerations}, parametrized by standard Levi subgroups in the ``dual group'' of $X$, whose Weyl group we denote by $W_X$. When $G$ is not split then we demand that $X$ is symmetric, and these symbols refer to their ``relative'' versions, cf.\ \S \ref{generalsymmetric}.

\subsection{Paley--Wiener for the Harish-Chandra Schwartz space}

The definition of the Harish-Chandra Schwartz space $\mathscr C(X_\Theta)$ (including the case $X_\Theta=X$) is recalled in \S \ref{sseigenmeasures}. It is a topological vector space (more precisely: an LF-space, i.e.\ countable strict direct limit of Fr\'echet spaces) of functions which plays a central role in the derivation of the Plancherel formula for the group by Harish-Chandra, cf.\ \cite{Waldspurger-Plancherel}. On the other hand, the method of proof of the Plancherel formula introduced in \cite{SV} and adopted in \cite{Delorme-Plancherel} directly leads to the $L^2$-Plancherel formula, without having to characterize the spectral transform of Harish-Chandra Schwartz functions; thus, this problem remained open.

In \S \ref{sec:discreteHC} it is shown that $\mathscr C(X_\Theta)$ has a direct summand, the intersection  with the ``discrete-modulo-center'' part $L^2(X_\Theta)_\disc$ of $L^2(X_\Theta)$, which we will denote by $\mathscr C(X_\Theta)_\disc$. That carries an action of a ring of multipliers $\mathfrak z^\disc(X_\Theta^L)$, the ``discrete center'' of $X_\Theta$, which is isomorphic to the ring of $C^\infty$ functions on the ``discrete spectrum'' $\widehat{X_\Theta^L}^\disc$ of $X_\Theta$ (to be explained below):
$$\mathfrak z^\disc(X_\Theta^L):= C^\infty(\widehat{X_\Theta^L}^\disc).$$

By \cite{SV, Delorme-Plancherel}, for each $\Theta$ one has a canonical ``Bernstein map'': 
$$\iota_\Theta: L^2(X_\Theta)\to L^2(X).$$
Moreover, for each $w\in W_X(\Omega,\Theta)$, i.e.\ each element of $W_X$ which takes a standard Levi $\Theta$ of the dual group to a standard Levi $\Omega$, there is a canonical ``scattering map'' 
$$S_w: L^2(X_\Theta)\xrightarrow\sim L^2(X_\Omega),$$
which is $w$-equivariant with respect to the ``centers'' (i.e.\ $G$-automorphism groups) of these spaces and such that we have a decomposition:
\beq\label{ii}\iota_\Omega^*\iota_\Theta = \sum_{w\in W_X(\Omega,\Theta)} S_w.\eeq
Notice that, despite the notation, the scattering operators are not parametrized by elements of $W_X$, but by triples $(\Theta,\Omega, w\in W_X(\Omega,\Theta))$.

The main theorem \cite[Theorem 7.3.1]{SV}, \cite[Theorem 6]{Delorme-Plancherel} on the Plancherel decomposition of $L^2(X)$ states:

\begin{theorem}\label{L2theorem}
 Let $\iota_{\Theta,\disc}^*$ denote the map $\iota_\Theta^*$ composed with projection to the discrete spectrum. The sum:
 \begin{equation}\label{L2} \iota^*:= \sum_\Theta \frac{\iota_{\Theta,\disc}^*}{\sqrt{c(\Theta)}}: L^2(X) \to \bigoplus_\Theta L^2(X_\Theta)_\disc,
 \end{equation}
 where $c(\Theta)$ is the number of ``Weyl chambers'' associated to $\Theta$ ($= \#\{w\in W_X| w\Theta\subset\Delta_X\}$), is an isometric isomorphism of $L^2(X)$ onto 
 $$\left(\bigoplus_\Theta L^2(X_\Theta)_\disc\right)^\inv,$$ the subspace consisting of collections $(f_\Theta)_\Theta$ such that for all triples $(\Theta,\Omega, w\in W_X(\Omega,\Theta))$ we have: $S_w f_\Theta = f_\Omega$. 
\end{theorem}

Our first version of the Paley--Wiener theorem for the Harish-Chandra Schwartz space reads:

\begin{theorem}[cf.\ Theorem \ref{HCtheorem}] \label{HCPW1}
The scattering maps $S_w$ restrict to $\mathfrak z^\disc(X_\Theta^L)$-equivariant isomorphisms (of LF-spaces) on the discrete part of the Harish-Chandra Schwartz spaces:\footnote{In Theorem \ref{unitaryextension} we extend this statement to the whole Harish-Chandra Schwartz space, but this is not necessary for formulating the Paley--Wiener theorem and only comes as a corollary of it.}
\begin{equation}\label{HCPW1eq} S_w: \mathscr C(X_\Theta)_\disc\xrightarrow{\sim} \mathscr C(X_\Omega)_\disc,\end{equation}
where $\mathfrak z^\disc(X_\Theta^L)$ acts on $\mathscr C(X_\Omega)_\disc$ via the isomorphism: 
\begin{equation}\label{wisomdisc}
\mathfrak z^\disc(X_\Theta^L)\xrightarrow\sim \mathfrak z^\disc(X_\Omega^L)
\end{equation} induced by $w$.

The sum \eqref{L2} restricts to an isomorphism of topological $G$-modules:
\begin{equation} \iota^*: \mathscr C(X) \xrightarrow\sim \left(\bigoplus_\Theta \mathscr C(X_\Theta)_\disc\right)^\inv.
\end{equation}
\end{theorem}

There is also a more explicit version of this theorem, in terms of ``normalized Eisenstein integrals'' and ``normalized constant terms''. Let $\pi$ be an irreducible smooth representation of $G$. 
 One defines the space of $\pi$-coinvariants $\mathcal S(X)_\pi$, the largest $\pi$-isotypic quotient of $\mathcal S(X)$; equivalently:
 $$\mathcal S(X)_\pi = \Hom_G( S(X),\pi)^*\otimes \pi.$$
 Its smooth dual can be identified with a canonical submodule $C^\infty(X)^\pi$ of  $C^\infty (X)$. One defines various subspaces $C^{\infty}_{\disc}(X)^\pi$,  $C^{\infty}_{\cusp}(X)^\pi $ corresponding to the condition of square integrability, resp.\ compact support modulo center, and denotes the sets of unitary irreducible representations which appear discretely, resp.\ cuspidally, by $\hat X^\disc$, resp.\ $\hat X^\cusp$. 
  Dually, we have the corresponding quotients: 
$$\mathcal S(X)_\pi \twoheadrightarrow \mathcal S(X)_{\pi, \disc} \twoheadrightarrow \mathcal S(X)_{\pi, \cusp}.$$

The same definitions can be given for any boundary degeneration $X_\Theta$, but taking into account that this space is ``parabolically induced'' from a ``Levi spherical variety'' $X_\Theta^L$ for a Levi subgroup $L_\Theta$, i.e.:
$$ X_\Theta \simeq X_\Theta^L\times^{P_\Theta^-} G,$$
The corresponding coinvariants are also parabolically induced, and indexed by representations of $L_\Theta$. 

As $\sigma$ varies over the set $\widehat {X_\Theta^L}^\disc$ of those representations which appear discretely-mod-center, the spaces $\mathscr L_{\Theta,\sigma}:=\mathcal S(X_\Theta)_{\sigma, \disc}$ are the fibers of a \emph{complex algebraic vector bundle} (actually, a countable direct limit of such) $\mathscr L_\Theta$ over the complexification of $\widehat {X_\Theta^L}^\disc$, and the canonical quotient maps give rise to a surjective morphism:
\begin{equation}\label{surjective} \mathcal S(X_\Theta) \twoheadrightarrow \CC[\widehat{X_\Theta^L}^\disc, \mathscr L_\Theta],\end{equation}
where $\CC[\bullet,\bullet]$ denotes regular (polynomial) sections of the given vector bundle.
In Theorem \ref{thmdiscrete} we show that this extends continuously to an isomorphism (via the orthogonal quotient map $\mathscr C(X)\to \mathscr C(X)_\disc$) of LF spaces:
\beq\label{isomdiscrete}\mathscr{C}(X_\Theta)_{\disc} \xrightarrow\sim C^\infty(\widehat{X_\Theta^L}^{\disc}, \mathscr L_\Theta)\eeq
(smooth sections).
The aforementioned action of the discrete center $\mathfrak z^\disc(X_\Theta^L)$ on the left is, by definition, the natural action of $C^\infty(\widehat{X_\Theta^L}^\disc)$ on the right.

It follows from their $\mathfrak z^\disc(X_\Theta^L)$-equivariance that the operators $S_w$ act fiberwise on these vector bundles; more precisely, it turns out that there are elements:
$$\mathscr S_w \in \Gamma\left(\widehat{X_\Theta^L}^\disc, \Hom_G(\mathscr L_\Theta,w^*\mathscr L_\Omega)\right),$$
where $\Gamma$ denotes \emph{rational} sections whose poles do not meet the unitary set (cf.\ \S \ref{sssections}), such that the following diagram of isomorphisms commutes:
$$\xymatrix{\mathscr C(X_\Theta)_\disc \ar[r]^\sim\ar[d]_{\mathcal S_w} & C^\infty(\widehat{X_\Theta^L}^\disc, \mathscr L_\Theta) \ar[d]^{\mathscr S_w}\\
\mathscr C(X_\Omega)_\disc \ar[r]^\sim  & C^\infty(\widehat{X_\Omega^L}^\disc, \mathscr L_\Omega).
}$$

Similarly, the Bernstein maps $\iota_\Theta$ are explicitly given by \emph{normalized Eisenstein integrals} associated to discrete data, which are explicitly defined maps: 
$$ E_{\Theta, \sigma, \disc}: \widetilde{\mathscr{L}_{\Theta, \sigma}}= C^\infty (X_\Theta)^{\tilde{\sigma}}_{\disc}\to C^\infty(X)$$ 
(where $\tilde ~$ denotes smooth dual), varying rationally with $\sigma$. If $f \in L^2 (X_\Theta)_{\disc}^\infty $ admits the decomposition: 
$$f = \int_{ \widehat{X_{\Theta}^{L}}^{\disc} }  f^{\tilde{\sigma}} (x) d\sigma$$ with $f^{\tilde{\sigma}}\in  C^\infty(X_\Theta)^{\tilde{\sigma}}_{\disc}$, 
then its image under the Bernstein map is the \emph{wave packet}:
\begin{equation}\label{explicitiota}\iota_\Theta f= \int_{ \widehat{X_\Theta^L}^{\disc}} E_{\Theta, \sigma, \disc} f ^{\tilde{\sigma}}d \sigma,
\end{equation}
cf.\ \cite[Theorem 15.6.1]{SV}, \cite[Theorem 7]{Delorme-Plancherel}. We use the $L^2$-continuity of $\iota_\Theta$ to prove that the normalized Eisenstein integrals (which are a priori rational in $\sigma$) have no poles on the imaginary axis, thus dually we get \emph{normalized constant terms} (often called \emph{normalized Fourier transforms} in the literature on symmeric spaces): 
 
\begin{equation}\label{Fourier} E_{\Theta,\disc}^*: \mathcal S(X)\to \Gamma(\widehat{X_\Theta^L}^\disc, \mathscr L_\Theta),\end{equation}
representing $\iota^*_{\Theta,\disc}$, where by $\Gamma(\bullet,\bullet)$ we denote again rational sections whose poles do not meet the unitary set. Combining all of this with Theorem \ref{HCPW1} we get the following explicit Paley--Wiener theorem for the Harish-Chandra Schwartz space:

\begin{theorem}[cf.\ Theorem \ref{HCtheorem2}] \label{HCPW2}
 The normalized constant terms \eqref{Fourier} extend to an isomorphism of LF-spaces:
\begin{equation} \label{HCPW2eq}\mathscr C(X) \xrightarrow\sim \left(\bigoplus_\Theta C^\infty(\widehat{X_\Theta^L}^\disc, \mathscr L_\Theta)\right)^\inv,\end{equation}
where $~^\inv$ here denotes $\mathscr S_w$-invariants, i.e.\ collections of sections $(f_\Theta)_\Theta$ such that for all triples $(\Theta,\Omega, w\in W_X(\Omega,\Theta))$ we have: $\mathscr S_w f_\Theta = f_\Omega$. 
\end{theorem}
In the group case, this theorem is part of the Plancherel formula of Harish-Chandra, appearing in Waldspurger \cite{Waldspurger-Plancherel}. { However, our proof is new, starting from \emph{a priori} knowledge of the $L^2$-maps \eqref{L2} and their properties.}

We remark that \eqref{explicitiota}, in combination with the fact that $\iota^* \iota$ is the identity on $\mathcal S_w$-invariants, provide an explicit way to invert this map by means of normalized Eisenstein integrals. Notice that we do not explicitly identify the scattering maps; this can be the object of further research, with a number-theoretic flavor since their poles are often related to $L$-functions. We only describe their relation to normalized Eisenstein integrals in \eqref{Bmatrices}, and give a few examples of those scattering operators in \S \ref{sec:examples}.

A corollary of this theorem (or its previous version \ref{HCPW1}) is the existence of a ring of \emph{multipliers} on $\mathscr C(X)$. Notice that each $w\in W_X(\Omega,\Theta)$ induces the isomorphism \eqref{wisomdisc} between discrete centers.
Let:
\begin{equation}\mathfrak z^\temp(X) = \left(\bigoplus_\Theta \mathfrak z^\disc(X_\Theta^L)\right)^\inv\end{equation}
denote the invariants of these isomorphisms, for all triples $(\Theta,\Omega, w\in W_X(\Omega,\Theta))$. One can call this ring the \emph{tempered center} of $X$ --- it is the relative analog of the tempered center of Schneider and Zink \cite{SZ2} (whose structure can also be inferred directly from the Plancherel theorem of Harish-Chandra \cite{Waldspurger-Plancherel}). 

\begin{corollary}[s.\ Corollary \ref{mult-HC2}]\label{mult-HC}
 There is a canonical action of $\mathfrak z^\temp(X)$ by continuous $G$-endomorphisms on $\mathscr C(X)$. 
\end{corollary}

This action, by definition, corresponds to the obvious action of $\mathfrak z^\temp(X)$ on the right hand sides of \eqref{HCPW1eq}, \eqref{HCPW2eq}.

\subsection{Paley--Wiener for the Schwartz space}

We now come to a Paley--Wiener theorem for the Schwartz space $\mathcal S(X)$ of compactly supported smooth functions on $X$. In analogy with the previous case, this has a distinguished direct summand $\mathcal S(X)_\cusp$, its ``cuspidal part'', consisting of those functions $f\in \mathcal S(X)$ such that for any open compact subgroup $J$, the Hecke module $\mathcal H(G,J)\cdot f$ is a finitely generated module under $\mathcal Z(X)$ (s.\ section \ref{sec:cuspidal}). The (orthogonal) complement of $\mathcal S(X)_\cusp$ in $\mathcal S(X)$ consists of those functions which are orthogonal to any of the spaces $C^{\infty}_\cusp(X)^\pi$ introduced before. 

The same definitions hold for the boundary degenerations $X_\Theta$, and the space $\mathcal S(X_\Theta)_\cusp$ comes equipped with the action of a ``cuspidal center'' $\mathfrak z^\cusp(X_\Theta^L)$, identified with the ring of polynomial functions on the subset $\widehat{X_\Theta^L}^\cusp \subset \widehat{X_\Theta^L}^\disc$:
\begin{equation}
\mathfrak z^\cusp(X):= \CC[\widehat{X_\Theta^L}^\cusp].
\end{equation}

Here we have the ``equivariant exponential maps'':
\begin{equation}
 e_\Theta: \mathcal S(X_\Theta) \to \mathcal S(X),
\end{equation}
whose transposes:
$$ e_\Theta^* : C^\infty(X)\to C^\infty(X_\Theta)$$
are a convenient way to generalize the classical theory of asymptotics of matrix coefficients (see \cite[\S 5]{SV}). The name ``exponential maps'' is due to the fact that the space $X_\Theta$ can be identified with the open $G$-orbit in a normal bundle to some orbit $\infty_\Theta$ in a compactification of $X$, and on characteristic functions of sets close to $\infty_\Theta$ the map $e_\Theta$ coincides with a physical ``exponential'' map, that is, a $p$-adic analytic map whose differential is the identity, cf.\ \cite[\S 5]{SV}. For an explicit formula for the maps $e_\Theta$, cf.\ \eqref{explicite} below. The space $\mathcal S(X)$ is the sum of all $e_\Theta \mathcal S(X_\Theta)_\cusp$:

\begin{theorem}[s.\ Theorem \ref{propbasicdecomp}] \label{pbd}
We have:
$$ \mathcal S(X) = \sum_{\Theta\subset\Delta_X} e_\Theta \mathcal S(X_\Theta)_\cusp.$$
\end{theorem}
We note that this \emph{fails} to be true without the assumption that $X$ is strongly factorizable, cf.\ Remark \ref{falseifnotfactorizable}; interesting phenomena await the researcher who will work on the general case!

A basic element in our analysis is a similar to the unitary case decomposition into ``smooth scattering maps'' when $\Theta$ and $\Omega$ are conjugate:\footnote{Again, these maps are initially defined only on cuspidal summands, but a posteriori extended to the whole space, cf.\ Theorem \ref{smoothextension}.}
\begin{equation} \label{ee} 
e^*_\Omega  {e_\Theta}|_{\mathcal S(X_\Theta)_\cusp} = \sum_{w\in W_X(\Omega,\Theta) } \mathfrak{S}_w, 
\end{equation}
where the maps $\mathfrak{S}_w: \mathcal S(X_\Theta)_\cusp\to C^\infty(X_\Omega)_\cusp$ are $\mathfrak z^\cusp(X_\Theta^L)$-equivariant when this ring acts on $C^\infty(X_\Omega)_\cusp$ via the isomorphism:
\begin{equation}\label{wisomcusp}
\mathfrak z^\cusp(X_\Theta^L)\xrightarrow\sim \mathfrak z^\cusp(X_\Omega^L)
\end{equation}
induced by $w$. Note that in this case neither the scattering maps nor the isomorphisms between cuspidal spectra of $X_\Theta^L$ and $X_\Omega^L$ are provided by the $L^2$-theory: all these are results that we need to establish.

The adjoint $e^*_\Omega$ (``smooth asymptotics map'') of $e_\Omega$ does not preserve compact support, therefore the maps $\textswab S_w$ have image in some subspace of  $C^\infty(X_\Omega)$. If we let $\mathcal S^+(X_\Omega)$ denote the space generated by the images of those $\textswab S_w$ (for all associates $\Theta$ of $\Omega$ and all $w\in W_X(\Omega,\Theta)$), then (cf.\ Theorem \ref{smoothscattering}) each scattering map $\textswab S_w$ extends canonically to an isomorphism:
$$\textswab S_w: \mathcal S^+(X_\Theta) \xrightarrow\sim \mathcal S^+(X_\Omega).$$

The first version of our Paley--Wiener theorem for the Schwartz space reads:

\begin{theorem}[cf.\ Theorem \ref{Schwartztheorem}]\label{PW1}
 Let $e_{\Theta,\cusp}^*$ denote the map $e_\Theta^*$ composed with projection to the cuspidal summand. The sum:
\begin{equation}\label{PW1eq}
e^*:= \sum_\Theta e_{\Theta,\cusp}^*: \mathcal S(X) \to \bigoplus_\Theta \mathcal S^+(X_\Theta)_\cusp
\end{equation} 
is an isomorphism into the $(\textswab S_w)_w$-invariants of the space on the right, i.e.\ the subspace consisting of collections $(f_\Theta)_\Theta$ such that for all triples $(\Theta,\Omega, w\in W_X(\Omega,\Theta))$ we have: $\textswab S_w f_\Theta = f_\Omega$. 
\end{theorem}

Again there is a more explicit version of this theorem. 
Consider the bundle $\mathcal L_\Theta$ whose fibers are the cuspidal coinvariants $\mathcal L_{\Theta,\sigma}:=\mathcal S(X_\Theta)_{\sigma,\cusp}$; it is a (countable direct limit of) complex algebraic vector bundle(s) over the complexification of the subset $\widehat{X_\Theta^L}^\cusp \subset \widehat{X_\Theta^L}^\disc$ where these spaces are non-zero.

In analogy to \eqref{isomdiscrete}, the canonical quotient maps give rise to isomorphisms:
\begin{equation}\label{isomcuspidal} \mathcal S(X_\Theta)_\cusp \xrightarrow\sim \CC[\widehat{X_\Theta^L}^\cusp, \mathcal L_\Theta].\end{equation}
The action of the cuspidal center $\mathfrak z^\cusp(X_\Theta^L)$ is nothing but the action of $\CC[\widehat{X_\Theta^L}^\cusp]$ on the right hand side. 

For the space $\mathcal S^+(X_\Theta)_\cusp$ this extends to an identification with a ``fractional ideal'' (i.e.\ a subspace of the space of rational sections which, when multiplied by a suitable element of $\CC[\widehat{X_\Theta^L}^\cusp]$, becomes regular):
\begin{equation}\label{spectral-Splus}\mathcal S^+(X_\Theta)_\cusp \xrightarrow\sim \CC^+[\widehat{X_\Theta^L}^\cusp, \mathcal L_\Theta]\subset \CC(\widehat{X_\Theta^L}^\cusp,\mathcal L_\Theta),
\end{equation}
but this identification needs some explanation. The ``fractional ideal'' will not be identified (except in specific examples); this seems to be a number-theoretic question, as in all known examples it involves $L$-functions. We only know that it is obtained by inverting ``linear polynomials'' (see \S \ref{sslinearpoles} for the definition of ``linear''). Despite the notation, it does not only depend on the isomorphism class of $X_\Theta$, but it actually depends on $X$ itself. It can, in principle, be computed by \eqref{Bmatrices} whenever the normalized Eisenstein integrals can.

Using the isomorphism \eqref{spectral-Splus}, the smooth scattering maps $\textswab S_w$ can be expressed in terms of the \emph{same} fiberwise scattering maps $\mathscr S_w$ as before (but restricted, of course, to the subbundle $\mathcal L_\Theta$ of $\mathscr L_\Theta$ which they turn out to preserve). Namely, the isomorphism \eqref{spectral-Splus} fits into a commuting diagram:

$$\xymatrix{\mathcal S^+(X_\Theta)_\cusp \ar[r]^\sim\ar[d]_{\mathcal S_w} & \CC^+[\widehat{X_\Theta^L}^\cusp, \mathcal L_\Theta] \ar[d]^{\mathscr S_w}\\
\mathcal S^+(X_\Omega)_\cusp \ar[r]^\sim  & \CC^+[\widehat{X_\Omega^L}^\cusp, \mathcal L_\Omega].
}$$

Although the fiberwise scattering maps $\mathscr S_w$ are the same as before, the inversion of \eqref{spectral-Splus} is \emph{not} given by the same fiberwise formula as the inversion of \eqref{isomdiscrete}: the latter is inverted by an integral over the unitary spectrum, and the former as an integral over a translate of the unitary spectrum, cf.\ \eqref{smoothSw} and \eqref{unitarySw}. Thus, the smooth scattering maps $\textswab S_w$ do not coincide, as maps between spaces of functions, with the unitary scattering maps $S_w$, despite the fact that their spectral transforms are expressed in terms of the same operators $\mathscr S_w$.

Similarly,  
the explicit version of the equivariant exponential map $e_\Theta$ is given by normalized Eisenstein integrals (as was the case for the Bernstein map $\iota_\Theta$), but using \emph{shifted} wave packets this time. More precisely, if we fix a Haar-Plancherel measure $d\sigma$ on $\widehat{X_\Theta^L}^\cusp$ and use it to write $f\in\mathcal S(X_\Theta)_\cusp$ as:
$$ f = \int_{ \widehat{X_{\Theta}^{L}}^{\cusp} }  f^{\tilde{\sigma}} (x) d\sigma$$ with $f^{\tilde{\sigma}}\in  C^\infty(X_\Theta)^{\tilde{\sigma}}_{\cusp}= \widetilde{\mathcal L_{\Theta,\sigma}}$, then by \eqref{isomcuspidal}  $f^{\tilde{\sigma}}$ extends polynomially to non-unitary $\sigma$'s and we have:
\begin{equation}\label{explicite}
  e_\Theta f (x) = \int_{\omega^{-1}\widehat{X_\Theta^L}^\cusp} E_{\Theta,\sigma,\cusp} f^{\tilde\sigma}(x) d\sigma.
\end{equation}
for every ``sufficiently positive'' character $\omega$, cf.\ Theorem \ref{explicitsmooth}. (For symmetric spaces, the fact that  shifted wave packets  are compactly supported can also be   proved using the results of \cite{CD} and a technique due to Heiermann in the group case \cite{Heiermann}.)

Dually, this gives an expression of $e^*_{\Theta,\cusp}$ as a \emph{normalized constant term}:
\begin{equation}\label{cuspFourier} \mathcal S(X)\to \Gamma(\widehat{X_\Theta^L}^\cusp, \mathcal L_\Theta),\end{equation}
which is the same map as \eqref{Fourier} composed with the natural quotient $\mathscr L_\Theta\to \mathcal L_\Theta$.

The explicit version of our Paley--Wiener theorem for the Schwartz space reads: 

\begin{theorem}[cf.\ Theorem \ref{Schwartztheorem2}] \label{PW2}
 The morphisms \eqref{cuspFourier} give rise to an isomorphism:
\begin{equation}\label{PW2eq}
\mathcal S(X) \xrightarrow\sim \left(\bigoplus_\Theta \CC^+[\widehat{X_\Theta^L}^\cusp, \mathcal L_\Theta]\right)^\inv,
\end{equation}
where $~^\inv$ here denotes $\mathscr S_w$-invariants.
\end{theorem}

A corollary of this theorem (or its previous version \eqref{PW1}) is the existence of a ring of \emph{multipliers} on $\mathcal S(X)$. Notice that each $w\in W_X(\Omega,\Theta)$ induces the isomorphism \eqref{wisomcusp} between cuspidal centers.
Let:
$$\mathfrak z^\sm(X) = \left(\bigoplus_\Theta \mathfrak z^\cusp(X_\Theta^L)\right)^\inv$$
denote the invariants of these isomorphisms, for all triples $(\Theta,\Omega, w\in W_X(\Omega,\Theta))$. One can call this ring the \emph{smooth center} of $X$ --- it is the relative analog of the Bernstein center (cf.\ \S \ref{ssBcenter}). Then:

\begin{corollary}[s.\ Corollary \ref{mult-sm2}]\label{mult-sm}
 There is a canonical action of $\mathfrak z^\sm(X)$ by $G$-endomorphisms on $\mathcal S(X)$. 
\end{corollary}

This action, by definition, corresponds to the obvious action of $\mathfrak z^\sm(X)$ on the right hand sides of \eqref{PW1eq}, \eqref{PW2eq}. { Note that this ring of multipliers is, in general, larger than the ring induced by the Bernstein center. Indeed, there are many known examples of relatively cuspidal representations which are not cuspidal for the group, cf.\ \cite{Murnaghan}; the simplest example is the Steinberg representation for the variety $X=T\backslash\PGL_2$, where $T$ is a split torus.}

In section \ref{sec:group} we discuss the example of $X=$ a reductive group $H$ under the $G=H\times H$-action by left and right multiplication. We show that the multiplier ring $\mathfrak z^\sm(X)$ that we described above provides an alternative proof for the structure of the Bernstein center as the algebra of polynomials on the ``space'' of cuspidal supports.
We also discuss the relationship of our Paley--Wiener theorem with those of Bernstein \cite{BePadic} and Heiermann \cite{Heiermann}: in this case, our work is analogous to part A of \cite{Heiermann}, and one needs to apply part B, which is the hardest part of that paper, to obtain the usual Paley--Wiener theorem. This is a good point to reflect on what our theorem really represents: It represents a reduction of the study of smooth functions on $X$ to (relatively) cuspidal spectra plus the study of scattering operators; it does not, however, reveal much about the nature of these operators, which can be the object of further research.

{ However, this reduction is not straightforward, as there are facts that are ``obvious'' in the case of a group, but not in the relative case. The most important of those is to show why a relatively supercuspidal representation for a boundary degeneration $X_\Theta$ ``scatters'' as a relatively supercuspidal in an associate direction $X_\Omega$. (In the group this is obvious by the description of cuspidality in terms of coinvariants under unipotent subgroups.) This is one of the goals of the scattering theorems described in Section \ref{sec:scatteringgoals}, and its proof is based on one of the main technical results of the paper, Proposition \ref{fibercuspidal}. If we may try to encode its proof in one sentence, we would say that ``a priori knowledge that the asymptotics maps $e_\Theta$ preserve compact support does not allow the scattering maps to break the cuspidality condition''. This proposition generalizes results of Carmona and Delorme \cite{CD} in the symmetric case, which used a completely different proof exploiting the structure of symmetric spaces.

We now come a more detailed description of the contents of this paper, and the main steps in our proofs.}

\subsection{Proofs}

After introducing the necessary structure theory of spherical varieties in section \ref{sec:definitions} and the bundles of discrete and cuspidal coinvariants in sections \ref{sec:generalities} and \ref{sec:coinvariants}, the first step is to show that the discrete, resp.\ cuspidal summand of $\mathscr C(X)$, resp.\ $\mathcal S(X)$, is a direct summand. This is relatively easy to do, and is done in sections \ref{sec:discreteHC} and \ref{sec:cuspidal}.

The spectral characterization of $\mathscr C(X)_\disc$ \eqref{isomdiscrete} and $\mathcal S(X)_\cusp$ \eqref{isomcuspidal} is the next step, and the basis for those is the surjection \eqref{surjective}; this follows from the definition of the bundle $\mathscr L_\Theta$, and an application of Nakayama's lemma (Proposition \ref{discreteSchwartz}). After this, \eqref{isomdiscrete} follows from the analogous statement for abelian groups (we use here the assumption that $X$, and later $X_\Theta^L$, are all factorizable, cf.\ \S \ref{sec:definitions}), and \eqref{isomcuspidal} is immediate by projection from discrete to cuspidal.

The unitary scattering operators $S_w$ were introduced in \cite{SV}, but here we need to prove that they preserve Harish-Chandra Schwartz spaces (at least their discrete summands). The explicit expression \eqref{explicitiota} for $\iota_\Theta$ allows us to relate the fiberwise versions $\mathscr S_w$ of the scattering maps to the asymptotics of normalized Eisenstein integrals and normalized constant terms, hence deducing their rationality in the parameter by a linear algebra argument, Proposition \ref{decompEisenstein}. Essentially, the operator $\mathscr S_w$, for $w\in W_X(\Omega,\Theta)$, 
is the ``$w$-equivariant part'' of the asymptotics $e_\Omega^*$ of the normalized Eisenstein integral $E_{\Theta, \disc}$. A priori knowledge of $L^2$-boundedness of the operators $S_w$, together with the ``linear'' form of the poles of Eisenstein integrals \ref{regularEisenstein}, cf.\ also \cite{BD}, allow us to show that their poles do not meet the unitary spectrum, Theorem \ref{fiberwisescattering}, and since they are unitary it follows from \eqref{isomdiscrete} that $S_w$, for $w\in W_X(\Omega,\Theta)$, maps $\mathscr C(X_\Theta)_\disc$ isomorphically onto $\mathscr C(X_\Omega)_\disc$. 

Using this fact, and a characterization of the Bernstein maps $\iota_\Theta$ from \cite{SV}, we are able to prove that $\iota_\Theta$ maps $\mathscr C(X_\Theta)_\disc$ into $\mathscr C(X)$ (Proposition \ref{Theta to X}). This is essentially the fact that some wave packets are in the Harish-Chandra Schwartz space 
(cf.\ \cite{DH} for a result of this type for symmetric spaces). Vice versa, the description of $\iota_\Theta$ in terms of normalized Eisenstein integrals \eqref{explicitiota}, together with the regularity of normalized Eisenstein integrals on the unitary spectrum, proves that $\iota_{\Theta,\disc}^*$ continuously sends the space $\mathscr C(X)$ into $\mathscr C(X_\Theta)_\disc$ (Proposition \ref{X to Theta}), and this is enough to prove the Paley--Wiener theorems \ref{HCPW1}, \ref{HCPW2} for the Harish-Chandra Schwartz space. 

To construct the smooth scattering operators $\textswab S_w$ one needs to study properties of the compositions $e_\Omega^* e_\Theta$ restricted to cuspidal summands, and more precisely that the restriction of this composition to $\mathcal S(X_\Theta)_\cusp$ is zero if $\Omega$ does not contain an associate of $\Theta$, has cuspidal image if $\Omega$ is associate to $\Theta$ and has image in the orthogonal complement of the cuspidal summand if $\Omega$ strictly contains an associate of $\Theta$, Theorem \ref{smoothscattering}. The proofs of these facts are accomplished in section \ref{sec:smooth}. 
The proof relies  in a crucial way on a theorem in \cite{SV} (which in turn was based on a theorem of Bezrukavnikov and Kazhdan \cite{BezK}) which says that the support of $e^*_\Omega f$ for $f \in \mathcal S(X)$, is \emph{bounded}, i.e.\ of compact closure in a (fixed) affine embedding of $X_\Theta$; this allows to prove the vanishing of certain ``exponents'' of the normalized Eisenstein integrals which by \eqref{explicite} spectrally decompose the maps $e_\Theta$. By totally different methods these results were obtained by Carmona-Delorme \cite{CD} for symmetric spaces, via an explicit description of the constant term of Eisenstein integrals, starting from cuspidal data, in terms of ``$C$-functions''.  

As was mentioned in Proposition \ref{pbd}, the space $\mathcal S(X)$ is the sum of all ``shifted cuspidal wave packets'', i.e.\ the sum of all $e_\Theta\mathcal S(X_\Theta)_\cusp$. Then (\ref{ee}) can be understood as a decomposition of the asymptotics of shifted wave packets. The proof of Theorem  \ref{PW2} rests  mainly on (\ref{ee}).  

\subsection{Acknowledgments} We are very grateful to the referee for numerous corrections and suggestions. The first author  has been supported by a grant of Agence Nationale de la Recherche with reference
ANR-13-BS01-0012 FERPLAY, and by the Institut Universitaire de France. The third author was supported by NSF grants
DMS-1101471 and DMS-1502270. He would like to thank David Kazhdan for a very motivating conversation.

\part{Structure, notation and preliminaries}

\section{Boundary degenerations, exponents, Schwartz and Harish-Chandra Schwartz spaces} \label{sec:definitions}

\subsection{Assumptions} \label{ssassumptions}

We let $X$ be a homogeneous, quasi-affine spherical variety for a reductive group $G$ over a non-Archimedean local field $F$ in characteristic zero. We will generally denote the points of a variety $Y$ over our fixed non-Archimedean field $F$ simply by $Y$, when this creates no confusion. 
 The assumption on characteristic is in order to use the results of \cite{SV} which freely applied the structure theory of spherical varieties in characteristic zero. With minor modifications, those results should work in positive characteristic, and then the results of the current paper will directly extend. We notice that for symmetric spaces, \cite{Delorme-Plancherel} only required that the characteristic of the field be different than $2$; thus, we can already relax the assumption on the characteristic in that case.

We will make the following assumptions on $X$:
\begin{description}

\item[] If $G$ is not split, $X$ is symmetric. The symmetric condition (whether $G$ is split or not) subsumes all the conditions that follow, but should be considered as a provisional assumption in order to use the non-split analogs of spherical root systems used in \cite{Delorme-Plancherel} (cf.\ \S \ref{generalsymmetric}). Our methods do not depend otherwise on the structure of symmetric spaces, and once the analogs of \cite{Delorme-Plancherel} are extended to the broader setting of spherical varieties satisfying the assumptions below, our results immediately extend.

If $G$ is split, we assume:
\item[(wf)] $X$ is \emph{wavefront};
\item [(sf)] $X$ is \emph{strongly factorizable}  (cf.\ below for both of these notions);
\item [(gi)] $X$ satisfies a strong version of the ``generic injectivity'' condition (cf.\ \S \ref{ssgenericinjectivity}).

Up to now, our assumptions guarantee the validity of the full Plancherel decomposition of \cite[Theorem 7.3.1]{SV}, \cite[Theorem 6]{Delorme-Plancherel}. Finally, we require the validity of the explicit Plancherel formula in terms of normalized Eisenstein integrals:

\item [(ep)] The explicit Plancherel formula of \cite[Theorem 15.6.2]{SV}, \cite[Theorem 8]{Delorme-Plancherel} holds; this is the case, for instance, if the ``small Mackey restriction'' of \cite[\S 15.5]{SV} is generically injective. 

\end{description}

We repeat that all these conditions are satisfied if $X$ is symmetric; for the strong version of the generic injectivity assumption in the symmetric case, which was not used in the aforementioned references, we prove this in \S \ref{ssgenericinjectivity}. We expect condition (sf) to be the only crucial condition for the methods of this paper to work. (Without it, results for the Schwartz space have to be modified, cf.\ Remark \ref{falseifnotfactorizable}.) Condition (wf) is used because we need the theory of asymptotics of \cite{SV} (which should hold without this condition), and we expect conditions (gi) and (ep) to hold in general (but for now they have to be checked ``by hand'' in any non-symmetric case that one is interested in). In Appendix \ref{app:factorizable} we check those assumptions for a couple of non-symmetric examples.

\subsection{Whittaker-induction}

Our results also hold for a ``variety'' that is ``Whittaker-induced'' from one as above, at least when $G$ is split, where the necessary results on which this paper is based have been proven in \cite{SV}. That is, in a certain setting one can consider, instead of the spaces of functions that we will encounter, also spaces of sections of a line bundle defined by a character of a unipotent group. The precise setting was explained in \emph{loc.cit}. \S 2.6, and we repeat it here: 

Let $P^-$ be a parabolic subgroup of $G$, with a Levi decomposition $P^-=L \ltimes U_{P^-}$. Suppose that $X^L$ is a spherical $L$-variety which, for the purposes of this paper, we will assume to satisfy the assumptions of the previous subsection. Let $V = \Hom(U_{P^-},\Ga)$, and assume that we are given an $L$-equivariant morphism $\Lambda:X^L\to V$ with Zariski \emph{open} image. Finally, let $\psi: F\to \CC^\times$ be a non-trivial unitary character. 

We let $X = X^L\times^{P^-} G$ be the corresponding ``parabolically induced'' variety, and denote by $C^\infty(X,\mathcal L_\Psi)$ the space of smooth functions on the $F$-points of $X^L\times^L G$ which satisfy $f(x, ug) = \Psi_x(u) f(x,g)$ for every $u\in U_{P^-}$, where $\Psi_x$ is the composition of $\psi$ with $\Lambda(x)$. 

The contents of the present paper apply to the space $C^\infty(X,\mathcal L_\Psi)$ (and the subspaces $\mathcal S(X,\mathcal L_\psi)$, $\mathscr C(X,\mathcal L_\psi)$) without modification, once one has the correct notion of ``Weyl group''. This Weyl group is explained in \cite[\S 2.6]{SV}, and is different from the Weyl group of $X$ viewed as a $G$-variety; for example, for the Whittaker model of a split group $G$ this is the \emph{full} Weyl group of $G$.  For notational simplicity, we will not be writing $C^\infty(X,\mathcal L_\Psi)$ anywhere --- the notation in the paper is referring to sections of the trivial line bundle, and the immediate reformulations necessary to cover this case are left to the reader.

\subsection{The split case} \label{sssplitdefs}

We start by giving definitions when the group $G$ is split. We will then modify them for non-split $G$, when the space $X$ is symmetric (following \cite{Delorme-Plancherel}). 

Given a spherical variety $X$ for a group $G$, we define the (connected) \emph{center} of $X$ as the connected component of its $G$-automorphism group:
$$ \mathcal Z(X):= \Aut_G(X)^0.$$
It is known to be a torus, and we assume throughout (as we may, without loss of generality, by enlarging $G$ if necessary), that the natural map is surjective:
\begin{equation}\label{wholecenter}\mathcal Z(G)^0 \twoheadrightarrow \mathcal Z(X),\end{equation} 
where $\mathcal Z(G)^0$ denotes the connected center of $G$.
For any fixed Borel subgroup, we denote by $\mathring X$ the open Borel orbit on $X$.

Our varieties will be homogeneous, $X=H\backslash G$, and we let $X^\ab$ be the homogeneous variety under the abelianization $G^\ab$ of $G$ which is obtained by dividing $X$ by the action of the commutator group $[G,G]$. If we choose a point $x\in X$ with stabilizer $H$ and let $\overline{H^\ab}$ be the image of $H$ in $G^\ab$, then \emph{as algebraic varieties}: $X^\ab= G^\ab/\overline{H^\ab}$. 

We call $X$ \emph{factorizable} if $\dim X^\ab=\dim \mathcal Z(X)$; all symmetric varieties have this property. 
 If $X$ is factorizable then \emph{as algebraic varieties} (but not necessarily in terms of their $F$-points):
$$ X \simeq \mathcal Z(X)\cdot X',$$
where $X' = H\cap [G,G]\backslash [G,G]$. This, of course, depends on the choice of base point definining the isomorphism $X\simeq H\backslash G$, and if we choose different such points $x_1,x_2,\dots$ we get different subvarieties $X_1',X_2',\dots$. 

Then, \emph{at the level of $F$-points}, there are a finite number of points $x_i$ such that $X$ is the disjoint union of open-closed subsets:
\begin{equation}\label{disjoint}
 X(F) = \bigsqcup_{i=1}^n \mathcal Z(X)(F)\cdot X_i'(F).
\end{equation}
The non-canonical subvarieties $X_i'$ will never appear in the statements, but will sometimes be used in the proofs.

The group of unitary complex characters of the $F$-points of the torus $X^\ab$ will be denoted by $\hat X^\ab$, and its complexification (which can be identified with the group of not necessarily unitary characters) by $\hat X^\ab_\CC$. The identity component of $\hat X^\ab$, i.e.\ the group of unramified unitary characters, will be denoted by $\hat X^\unr$, and this notation ($~^\unr$) will be used more generally to denote groups of unramified characters.

To every spherical variety $X$ one associates its set of (simple) spherical roots $\Delta_X$ and the ``little Weyl group'' $W_X$, cf.\ \cite[\S 2]{SV}. The spherical roots live in the lattice $\varchi (X)$
 of characters of a Borel subgroup which are trivial on stabilizers of generic points, and $W_X$ acts by automorphisms on $\varchi(X)$. 
There are actually various normalizations for the spherical roots, depending on the application that one has in mind; for a certain normalization, they are part of the root data of the ``dual group'' $\check G_X$ of $X$; for another (the standard one in the theory of spherical varieties), they determine the structure of certain compactifications. These two normalizations were referred to as ``normalized'' and ``unnormalized'' roots in \cite[\S 2.1, 2.2]{SV}, and both of them define a root system in the usual sense. The precise choice of normalization will not be of particular concern to us, in general, and when does matters we will clarify which definition we are referring to. In any case, the action of $W_X$ on $\varchi(X)$, together with the dominant chamber determined by those sets of simple spherical roots, is independent of the chosen normalization of their lengths.

What is important for us is that one has the following set of data:
\begin{itemize}
 \item \emph{Boundary degenerations:} For every subset $\Theta\subset\Delta_X$, a spherical $G$-variety $X_\Theta$ of the same dimension, with the property that $$\dim\left(\mathcal Z(X_\Theta)\right) = \dim\left(\mathcal Z(X)\right) + |\Delta_X\smallsetminus\Theta|.$$ We interchangeably denote: $$A_{X,\Theta}:= \mathcal Z(X_\Theta).$$
 
Under the convention that $\mathcal Z(G)^0\twoheadrightarrow \mathcal Z(X)$ that we are using, $X$ is called \emph{wavefront} if for every $\Theta$ the variety $X_\Theta$ is parabolically induced from a spherical variety $X_\Theta^L$ (called \emph{Levi variety}) for the Levi quotient $L_\Theta$ of a parabolic $P_\Theta^-$:
\begin{equation}\label{Leviinduced} X_\Theta \simeq X_\Theta^L \times^{P_\Theta^-} G
\end{equation}
such that the action of $\mathcal Z(X_\Theta)$ is induced from the action of the connected center of $L_\Theta$ on $X_\Theta^L$. Only the conjugacy class of $P_\Theta^-$ is canonically defined (and then $X_\Theta^L$ is the fixed point set of its unipotent radical on $X_\Theta$); thus, whenever we use those Levi varieties we will be careful that no non-canonical choice of a representative for $P_\Theta^-$ affects our statements. The isomorphism \eqref{Leviinduced} shows that $X_\Theta^L$ can also be identified as the quotient of the open $P_\Theta$-orbit on $X_\Theta$ (where $P_\Theta$ is in the class of parabolics opposite to $P_\Theta^-$) by the (free) action of the unipotent radical $U_\Theta$ of $P_\Theta$. Since $\mathring X P_\Theta/U_\Theta \simeq \mathring X_\Theta P_\Theta/U_\Theta$ canonically \cite[Lemma 2.8.1]{SV}, the Levi variety is also identified with the analogous quotient for $X$ (the quotient of its open $P_\Theta$-orbit by the $U_\Theta$-action). 

A wavefront spherical variety is called \emph{strongly factorizable} if all of its Levi varieties are factorizable. Symmetric varieties are strongly factorizable \cite[Proposition 9.4.2]{SV}, and these are the main source of examples. In Appendix \ref{app:factorizable} we characterize strongly factorizable varieties in terms of combinatorial data attached to the spherical variety, and give a few examples of non-symmetric, strongly factorizable spherical varieties.

\begin{remark}\label{sfvsf} 
Note that the Levi variety attached to the whole set of spherical roots is not equal to $X$, if $X$ is parabolically induced. For example, if $X=N_1\backslash G_1$, where $N_1$ is maximal unipotent in $G_1$, under the $G=A_1\times G_1$-action (where $A_1=B_1/N_1$, with $B_1$ the Borel subgroup normalizing $N_1$), then $X$ is wavefront, $\Delta_X = \emptyset$, but $X_{\Delta_X}^L=X_\emptyset^L = A_1$ under an $A_1\times A_1$-action. This creates the paradox that some varieties (such as this example) are ``strongly factorizable'' without being ``factorizable'', but this is only a minor nuisance, since for a parabolically induced variety all spaces of functions that we are interested in are parabolically induced, s.\ \S \ref{sseigenmeasures} --- thus, one can work directly with the Levi variety $X_{\Delta_X}^L$, which is factorizable.  To avoid extra notation, however, instead of writing $X_{\Delta_X}^L$ we will at several points in this paper assume, implicitly, that $X$ is factorizable.
\end{remark}

For $\Theta=\emptyset$ the variety $X_\Theta$ is \emph{horospherical}, i.e.\ stabilizers contain maximal unipotent subgroups of $G$. More precisely, stabilizers contain the commutator subgroup of a parabolic in the class of $P_\emptyset^-$, which in this case we denote by $P(X)^-$. Its opposite $P(X)$ is the parabolic which stabilizes the open Borel orbit on $X$. (Again, of course, only its class is defined.) We denote $A_{X,\emptyset}$ simply by $A_X$ --- it is the ``universal Cartan'' of $X$; its character group has a canonical identification with $\varchi(X)$. For every $\Theta$, $A_{X,\Theta}$ is canonically identified with the connected kernel of $\Theta$ in $A_X$, and we denote by $A_{X,\Theta}^+$ the monoid of elements $a\in A_{X,\Theta}(F)$ with the property that $|e^\gamma(a)|\le 1$ for all $\gamma\in \Delta_X$, and by $\mathring A_{X,\Theta}^+$ the subset of those elements with $|e^\gamma(a)|< 1$ for all $\gamma\in \Delta_X\smallsetminus\Theta$. (We use the exponential symbol in order to use additive notation for the group $\varchi(X)$). 
 
 \item \emph{Exponential map:} For every open compact subgroup $J$ of $G$, a system of $J$-stable subsets $N_\Theta$ of $X=X(F)$, with $N_\Theta\subset N_\Omega$ if $\Theta\subset\Omega$, and for each $\Theta$ a $J\times A_{X,\Theta}^+$-stable subset $N_\Theta^\Theta$ of $X_\Theta$, which generates all $X_\Theta$ under the action of $A_{X,\Theta}$, together with identifications:
 \begin{equation}\label{Nident} N_\Theta/J = N_\Theta^\Theta/J,\end{equation}
 characterized by their compatibility with certain $p$-adic analytic ``exponential maps'' (s.\ \cite[\S 5]{SV}) and by the fact that the induced map on characteristic functions extends to an equivariant map, that will be explained in \S \ref{sec:asymptotics}. Such a set $N_\Theta$ will be called a ``$J$-good neighborhood of $\Theta$-infinity'', and from now on we will not distinguish in notation between $N_\Theta$ and $N_\Theta^\Theta$, i.e.\ we will be denoting the latter also by $N_\Theta$. (This constitutes abuse of notation, since only $J$-orbits on these sets are identified, but it will only be used for statements that depend only on the identification of the $J$-orbits, not the sets themselves.) The above identifications clearly also identify $N_\Omega/J$, for all $\Omega\subset\Theta$, with subsets of $X_\Theta/J$, and the set: 
 $$N_\Theta'=N_\Theta\smallsetminus \bigcup_{\Omega\subsetneq\Theta} N_\Omega$$
 is stable under the action of $J\times A_{X,\Theta}^+$ and 
 has compact image in $X_\Theta/A_{X,\Theta}$. We note the decomposition:
\begin{equation}\label{Xdecomp}
X = \bigsqcup_\Theta N_\Theta'.
\end{equation} 
 We have $N_{\Delta_X} = X$, hence the complement of $\bigcup_{\Theta\subsetneq \Delta_X} N_\Theta$ is compact modulo the action of $\mathcal Z(X)$.
\end{itemize}

The modular character of $P_\Theta$, i.e.\ the inverse of the modular character of $P_\Theta^-$, will be denoted by $\delta_\Theta$. (Our convention is that a modular character is the quotient of right by left Haar measure.) The functor of \emph{normalized} induction from $P_\Theta$, resp.\ $P_\Theta^-$, will be denoted by $I_\Theta$, resp.\ $I_{\Theta^-}$: 
$$I_{\Theta}V:= \{ f: G\to V \mbox{ smooth}| f(pg)=\delta_\Theta^{\frac{1}{2}} f(g)\mbox{ for all }p\in P_\Theta\},$$
$$I_{\Theta^-}V:= \{ f: G\to V \mbox{ smooth}| f(pg)=\delta_\Theta^{-\frac{1}{2}} f(g)\mbox{ for all }p\in P_\Theta^-\}.$$

We similarly denote, for every representation $\pi$ of $G$, the \emph{normalized} Jacquet modules with respect to $P_\Theta$, resp.\ $P_\Theta^-$, by $\pi_\Theta$, resp. $\pi_{\Theta^-}$. These are, by definition, the coinvariants of the corresponding unipotent radicals, tensored by the inverse square root of the corresponding modular character, so that we have canonical $L_\Theta$-morphisms:
$$\left(I_{\Theta}V\right)_\Theta\twoheadrightarrow V, \,\,\,\, \left(I_{\Theta^-}V\right)_{\Theta^-} \twoheadrightarrow V.$$
 
Actions of Weyl groups will always be defined to be \emph{left} actions. We consider the Weyl group $W$ of $G$ as an automorphism group of its \emph{universal Cartan} $A = B/N$ (where $B$ is any Borel subgroup, with unipotent radical $N$, so that the universal Cartan is a unique torus up to unique isomorphism). For subset $S$ of the positive simple roots of $A$ in $G$, corresponding to a class of parabolics $P_S$, any element $w$ which maps $S$ into the positive simple roots gives rise to an isomorphism between the Levi quotients $L_S$ and $L_{wS}$ of the corresponding parabolics, unique up to inner conjugacy. In particular, this is true for the Levi quotients $L_\Theta$, $L_\Omega$ (where $\Theta,\Omega\subset \Delta_X$) and an element 
$$w\in W_X(\Omega,\Theta):= \{ w\in W_X| w\Theta =\Omega\} \subset W_X\subset W.$$ 
Finally, the notation $\Theta\sim\Omega$ will mean that $\Theta$ and $\Omega$ are associates, i.e.\ $W_X(\Omega,\Theta)\ne \emptyset$.

\subsection{The general symmetric case}\label{generalsymmetric}

In the general symmetric case (when $G$ is not necessarily split), the boundary degenerations $X_\Theta$ are defined in \cite[\S 3.1]{Delorme-Plancherel}. They are denoted there by $X_P$, while the Levi varieties $X_\Theta^L$ are denoted by $X_M$.

For consistency of notation with the split case, we will make a small modification to the definitions of \cite{Delorme-Plancherel}. Namely, in \S 2 of \emph{loc.cit.} the tori $A_P$ are defined as certain central split subtori of Levi subgroups; thus, they do not need to act faithfully on the boundary degenerations $X_P$. (More precisely, their action might have finite kernel.) Here, we will denote by $A_P$ (or, rather, $A_{X,\Theta}$) the quotient by which these tori act on $X_P$; equivalently, $A_P$ for us will be a \emph{quotient} of the identity component $\mathcal Z(M)^0$ of the maximal split torus in the center of the Levi quotient $M$.  (For a group $M$, we will use the notation $\mathcal Z(M)^0$ for the maximal split torus in its center; the notation $\mathcal Z(M)$, without the exponent $~^0$, will not be used, again for consistency with the split case.)

These tori correspond to the \emph{maximal split tori} of what, over the algebraic closure, is $\mathcal Z(X)$ or $\mathcal Z(X_\Theta)$ under the definitions of the previous subsection. While it is not very good to have notation which is not stable under base change, it is convenient here that the emphasis is not on geometry but on harmonic analysis, and we will adopt it. Similarly, for the definition of $W_X$ in the general symmetric case, cf.\ \cite[\S 7.5]{Delorme-Plancherel}, denoted there $W(A_\emptyset)$.

\subsection{Normalized action and the various Schwartz spaces} \label{sseigenmeasures}

We assume that $X(F)$ carries a $G(F)$-eigenmeasure\footnote{In fact, under our assumption of factorizability it is possible to twist such a measure and make it invariant; however, even if we do this for $X$ it will not be the natural choice for the Levi varieties $X_\Theta^L$, as we will see, so one ends up working with eigenmeasures anyway.} with eigencharacter $\eta$, and any choice of such measure endows all the spaces $X_\Theta(F)$ with $G(F)$-eigenmeasures with the same eigencharacter which make the identifications \eqref{Nident} of neighborhoods of the form $N_\Theta/J$ measure-preserving, cf.\ \cite[\S 4.1]{SV}, \cite[Theorem 2]{Delorme-Plancherel}. This measure on $X_\Theta(F)$ is also an $A_{X,\Theta}(F)$-eigenmeasure, and whenever a group acts on a space $Y$ endowed with an eigenmeasure with eigencharacter $\chi$, we \emph{normalize} the action of the group on functions on $Y$ so that it is an $L^2$-isometry:
\begin{equation} \label{normalizedaction} (g\cdot f)(y) = \sqrt{\chi(g)} f(yg).\end{equation}
This also identifies the space $C^\infty(Y)$ (\emph{uniformly} locally constant functions on $Y$) with the smooth dual of $\mathcal S(Y):=C_c^\infty(Y)$.

On the Levi varieties $X_\Theta^L = \mathring X P_\Theta/ U_\Theta$ the measure on $X$ gives rise to an $L_\Theta$-eigenmeasure for which the following is true:
$$\int_{\mathring X P_\Theta} f(x) dx = \int_{X_\Theta^L} \int_{U_\Theta} f(ux) du dx.$$
This depends on the choice of Haar measure on $U_\Theta$. The character by which $L_\Theta$ acts on this measure is $\delta_\Theta \eta$ (recall that $\eta$ is the eigencharacter of the measure on $X$). Thus, we need to twist the unnormalized action of $L_\Theta$ on functions by $(\eta\delta_\Theta)^{\frac{1}{2}}$ in order to obtain a unitary representation.

Another way to describe this twisting is the following: if we identify $X_\Theta^L$ as a subvariety of $X_\Theta$ fixed by the parabolic $P_\Theta^-$, and $g\in P_\Theta^-$ with image $l\in L_\Theta$, then for a function $f$ a function on $X_\Theta$ we have:
\begin{equation} \label{Ltwist}  l\cdot (f|_{X_\Theta^L}) := \delta_\Theta^{\frac{1}{2}}(l) (g\cdot f)|_{X_\Theta^L}.\end{equation} 
(The twist by $\sqrt\eta$ is already contained in the $G$-action on $X_\Theta$.) An important observation is that, by introducting this twisted action for $L_\Theta$, the action of the connected center of $L_\Theta$ on $f|_{X_\Theta^L}$ coincides with the action of $\mathcal Z(X_\Theta)$ on $C^\infty(X_\Theta)$, under the identification of $\mathcal Z(X_\Theta) = A_{X,\Theta}$ as a quotient of $\mathcal Z(L_\Theta)^0$. Indeed, the twist by $\delta_\Theta^\frac{1}{2}$ is contained in \eqref{normalizedaction}, by taking into account the eigencharacter of the measure under the action of $\mathcal Z(X_\Theta)$.

 \emph{We caution the reader that this may not be the most natural-looking action}; for instance, if $X$ has a $G$-invariant measure and we consider the Levi variety $X_\emptyset^L\simeq A_X$, the usual action of $A$ on $C^\infty(A_X)$ is twisted by the square root of the modular character of $P(X)$. However, this definition is such that the space of $L^2$-functions on $X_\Theta$ is \emph{unitarily induced} from the analogous space on $X_\Theta^L$:
\begin{equation}\label{L2induced} L^2(X_\Theta) = I_{\Theta^-} L^2(X_\Theta^L),
\end{equation}

The \emph{Schwartz space} $\mathcal S(X)$ is, by definition, the space $C_c^\infty(X)$ of compactly supported smooth functions on $X$ (and similarly for any homogeneous space). The twist \eqref{Ltwist} on functions on $X_\Theta^L$ allows us to write, using again the functor of normalized induction from $P_\Theta^-$:
\begin{equation}\label{Schwartzinduced} \mathcal S(X_\Theta) = I_{\Theta^-} \mathcal S(X_\Theta^L),
\end{equation}

Moreover, if $X$ is a direct product:
$$ X = \mathcal Z(X) \times X',$$
where $X'$ is a $[G,G]$-spherical variety, we clearly have a decomposition:
$$ \mathcal S(X) = \mathcal S(\mathcal Z(X))\otimes \mathcal S(X').$$
In the general factorizable case, using a decomposition such as \eqref{disjoint} and pulling back functions by the action map:
$$\mathcal Z(X)\times X_i'\to \mathcal Z(X)\cdot X_i',$$
it is immediate to identify $\mathcal S(X)$ with:
\begin{equation}\label{Sfactorizable}
 \bigoplus_i \left( \mathcal S(\mathcal Z(X))\otimes \mathcal S(X_i')\right)^{(\mathcal Z(G)^0\cap [G,G])^\diag},
\end{equation}
i.e.\ invariants under the simultaneous action of the finite subgroup $\mathcal Z(G)^0\cap [G,G]$ on both factors. (Recall that $\mathcal Z(G)^0\twoheadrightarrow \mathcal Z(X)$ under our conventions.)

For any function on $X_\Theta$ which is $A_{X,\Theta}$-finite (i.e.\ its translates under the normalized action \eqref{normalizedaction} of $A_{X,\Theta}$ span a finite-dimensional space) we call \emph{exponents} its generalized $A_{X,\Theta}$-characters, considered as a multiset (i.e.\ each character appears with a certain multiplicity).

We say that a function $f\in C^\infty(X)$ (invariant, say, by an open compact subgroup $J$) is \emph{tempered} if for every $\Theta\subset\Delta_X$ there is a $J$-good neighborhood of $\Theta$-infinity where $|f|$ is bounded by \emph{an $A_{X,\Theta}$-finite function with trivial exponents} (equivalently: by the absolute value of an $A_{X,\Theta}$-finite function with \emph{unitary} exponents).

The \emph{Harish-Chandra Schwartz space} $\mathscr C(X)$ is the space of those functions $f\in C^\infty(X)$ such that for every tempered function $F$ we have:
\begin{equation}\label{HCcondition}   \rho_F(f) := \int_X |f\cdot F| dx <\infty.
\end{equation}
    
For example, in the abelian case $X=\mathcal Z(X)$ (by choosing a base point), any smooth function descends to a function on a finitely generated abelian group $\simeq R \mbox{ (torsion) }\times \mathbb Z^r$, and it is in the Harish-Chandra Schwartz space iff its restriction to any $\mathbb Z^r$-orbit is bounded by the multiple of the inverse of any polynomial in the coordinates $n_1, n_2, \dots, n_r$.
 
 We similarly define this notion for the spaces $X_\Theta$. Again, the twisted action \eqref{Ltwist} allows us to write the Harish-Chandra Schwartz space of $X_\Theta$ as the normalized induction of the Harish-Chandra Schwartz space of $X_\Theta^L$:
\begin{equation}\label{HCinduced} \mathscr C(X_\Theta) = I_{\Theta^-} \mathscr C(X_\Theta^L),
\end{equation}
Indeed, the action of $G$ is clearly the correct one as was the case for $L^2(X_\Theta)$ and $\mathcal S(X_\Theta$); and the notion of ``unitary exponents'' used to define tempered functions and, by duality, the Harish-Chandra Schwartz space coincides for the action of $A_{X,\Theta}$ on functions on $X_\Theta$ and $X_\Theta^L$. 
 
The $J$-invariants of each of those Harish-Chandra Schwartz spaces have a natural   Fr\'echet space structure, defined by a system of seminorms $\rho_F$ as above for $F$ belonging in any sequence $(F_n)_n$ of tempered, $J$-invariant functions with the property: for every tempered function $ F'$ there is an $n$ and a positive scalar $c$ such that $|F'|\le c\cdot |F_n|$. (In fact, this is a \emph{nuclear} Fr\'echet space.) Thus, the space $\mathscr C(X)$ is an LF-space, i.e.\ a countable strict inductive limit of Fr\'echet spaces.

In case $X$ is a direct product:
$$ X = \mathcal Z(X) \times X',$$
where $X'$ is a $[G,G]$-spherical variety, we have a decomposition:
$$ \mathscr C(X) = \mathscr C(\mathcal Z(X))\hat\otimes \mathscr C(X'),$$
where the completed tensor product is defined as a strict inductive limit over the corresponding spaces of invariants under compact open subgroups, and for each such subgroup it is uniquely defined by nuclearity. In simple terms, this means the following: We may choose the sequence as above of tempered functions $F_n$ to consist of product functions: $F_{ij} = F_i^{(1)} \otimes F_j^{(2)}$, where $F_i^{(1)}$ and $F_j^{(2)}$ denote, respectively, similar sequences on $\mathcal Z(X)$ and $X'$. Then $\mathscr C(X)^J$ is the completion of $\mathcal S(X)^J = \mathcal S(\mathcal Z(X))^{J\cap \mathcal Z(G)^0}\otimes \mathcal S(X')^{J\cap [G,G]}$ with respect to the corresponding seminorms.

In the general factorizable case, using a decomposition such as \eqref{disjoint} and pulling back functions by the action map:
$$\mathcal Z(X)\times X_i'\to \mathcal Z(X)\cdot X_i',$$
it is immediate to identify: 
\begin{equation}\label{Cfactorizable}
\mathscr C(X) \simeq \bigoplus_i \left( \mathscr C(\mathcal Z(X))\hat\otimes \mathscr C(X_i')\right)^{(\mathcal Z(G)^0\cap [G,G])^\diag}.
\end{equation}

Notice that we also have:
\begin{equation}\label{Lfactorizable} L^2(X) \simeq \bigoplus_i \left( L^2(\mathcal Z(X))\hat\otimes L^2(X_i')\right)^{(\mathcal Z(G)^0\cap [G,G])^\diag},
\end{equation}
where the completed tensor product here is the Hilbert space tensor product.

\subsubsection{Comparison with alternative definitions} 

Since the definition of Harish-Chandra Schwartz space is sometimes phrased differently in the literature, we would like to verify that the one we gave coincides with other versions. We start from the general definition given in \cite[\S 3.5]{BePl}; according to it, the Harish-Chandra Schwartz subspace of $C^\infty(X)^J$ is the one defined by the norms of $L^2(X, (1+r)^d \mu_X)^J$, for all $d\ge 1$. Here $r$ is a \emph{radial function} on $X$, and the measure $\mu_X$ is a $G$-invariant measure. (We leave the case of an eigenmeasure to the reader --- cf.\ \S 3.7 of \emph{loc.cit.})

We remind that a radial function $r:X\to \mathbb R^+$ is a locally bounded proper function such that for every compact subset $B\subset G$ there is a constant $C>0$ with 
\begin{equation}\label{radialdef}|r(gx)-r(x)|<C
\end{equation}
 for all $g\in B, x\in X$. The definition of the Harish-Chandra Schwartz space using radial functions generally depends on the radial function chosen up to the equivalence relation:
$$ r\sim r' \iff \exists C>0 \mbox{ s.t. } C^{-1} (1+r) \le 1+r' \le C(1+r).$$

However, for a homogeneous space $X$ there is a ``natural'' class of radial functions on $X$, described in \cite[\S 4.2]{BePl}. It admits the following explicit description, whose verification we leave to the reader, using a (weak) Cartan decomposition for $X$.

By a (weak) Cartan decomposition for $X$ we mean that there exists a subvariety $Y\subset X$, which (over the algebraic closure) is an orbit of a Cartan subgroup $T$ of $G$, such that:
\begin{equation}\label{weakCartan} X(F)=Y^+ U
\end{equation}
 for some large enough compact subset $U$ of $G(F)$, where $Y^+$ denotes a certain notion of ``dominant'' elements of $Y(F)$, cf.\ \cite{BO,DS} for the symmetric case and \cite[Lemma 5.3.1]{SV} for the general split case.  
 If we fix a 
 natural radial function $R$ on $Y(F)$ of the form $R=\Vert \omega \Vert$, where we choose a base point to identify $Y$ as a quotient of the Cartan subgroup $T$ and $\omega: Y(F)\to V$ is a homomorphism with compact kernel to a finite-dimensional real normed space $V$, 
  the following is a radial function on $X$, representing the natural class of radial functions:
$$ r(x) := \min\{R(a)| a\in Y^+, x \in aU\}.$$

In fact, as the proof of \cite[Lemma 5.3.1]{SV} shows, the subvariety $Y$ of the above decomposition can be identified with the torus $A_X$ defined in \S \ref{sssplitdefs}, in a way compatible with the $A_{X,\Theta}^+$-actions on good neighborhoods of infinity, i.e.: Considering a decomposition $X = \bigsqcup_\Theta N_\Theta'$ as in \eqref{Xdecomp}, the action of $A_{X,\Theta}^+\subset \mathcal Z(X_\Theta)$ on $Y^+\cap N_\Theta'$ coincides with its action on $A_X\simeq Y$, restricted to $Y^+\cap N_\Theta'$.

Thus, the above radial function is equivalent to the following one: Fix, for every $\Theta$, a $J$-invariant compact subset $M_\Theta\subset N_\Theta$ such that $A_{X,\Theta}^+ M_\Theta = N_\Theta$, and let, for each $x\in N_\Theta$: 
$$r'(x) = \min\{ R(a)| a\in A_{X,\Theta}^+, x\in a M_\Theta\},$$
where $R(a)$ is the fixed radial function on $A_{X,\Theta}\subset A_X$.

Returning to functions, observe that the functions 
$$F_d(x) = (1+r'(x))^\frac{d}{2} \mu_X(xJ)^{-\frac{1}{2}}$$ form a basis of tempered functions such as the ones used in our present definition \eqref{HCcondition} of the Fr\'echet structure of the Harish-Chandra Schwartz space. Thus, the system of norms of the spaces $L^2(X, (1+r)^d \mu_X)^J$ is equivalent to the system of norms $\rho_{F_d}$ defined using those functions. This shows the equivalence of our definition with Bernstein's.

Finally, we also have \cite[Definition 3]{DH} in the case of a symmetric space, which defines the Harish-Chandra Schwartz space in terms of bounds of the form:
$$ |f(x)| \ll \Theta_G(x) (N_d(x))^{-1}$$
with $d>0$. The function $N_d$ is of the form $(1+r)^d$, for an \emph{algebraic} radial function $r$; namely, $X$ is realized as a closed subvariety of affine space, and the function $r$ is the maximum of the absolute values of the coordinates. Such a radial function can easily be seen to be equivalent to the ones used above, cf.\ also \cite[\S 4.5]{BePl}. The function $\Theta_G$ is a nonvanishing positive smooth function, which up to a power of $(1+r)$ and a constant coincides with the function $x\mapsto \mu_X(xJ)^{-\frac{1}{2}}$,  by a result of Lagier \cite[(2.27)]{DH} and an estimate of the volumes in \cite[Proposition 2.6]{KatoTakano}. Thus, this definition of $\mathscr C(X)$ also coincides with the above ones.

\section{Bundles over tori} \label{sec:generalities}

\subsection{Bundles with flat connections over complex tori} \label{ssgeneralities}

Let $T$ be a complex algebraic torus, and let $V$ be a finite-dimensional complex vector space. Let $\Gamma\subset T$ be a finite subgroup, and let $\rho:\Gamma\to \GL(V)$ be a representation. Thus, $\Gamma$ acts on the total space of the vector bundle $T\times V$, and the quotient $V_\rho$ is a vector bundle over the quotient torus $Y=T/\Gamma$. 

By the following argument, one can see that this vector bundle is trivializable; however, we will not fix such a trivialization. The representation $\rho$ always extends to a complex algebraic homomorphism $\tilde\rho: T \to \GL(V)$. Indeed, $\rho$ decomposes into a finite sum of characters of $\Gamma$; viewing $\Gamma$ as the points of a finite algebraic group, each character is algebraic. The coordinate ring $\CC[T]$ of $T$, which is spanned by its characters, surjects onto the coordinate ring of $\Gamma$, and hence for every character $\chi$ of $\Gamma$, the $(\Gamma,\chi)$-equivariant part of $\CC[T]$ is non-zero. Thus, $\chi$ extends to a complex algebraic character of $T$, and $\rho$ extends to $\tilde\rho$.
Then, once we choose a basis $(v_1,\dots,v_n)$ of $V$, the trivialization of the bundle $V_\rho$ is given by the sections: $(t, v_i(t))$, where $v_i(t) = \tilde\rho(t)v_i$. 

In this paper, we will apply this construction to $T=\hat X^\unr_\CC$, $T/\Gamma=$a connected component of $\hat X^\disc_\CC$ (and the corresponding tori for ``boundary degenerations'' --- see \S \ref{sec:coinvariants} for the definitions). The vector bundle will come from certain spaces of coinvariants of $\mathcal S(X)$.

We want to endow the vector bundle $V_\rho$ with a flat connection, hence an action (on its sections) of the ring $\mathcal D(Y)$ of differential operators on $Y=T/\Gamma$. There are two obvious choices for doing that: One is to choose a trivialization by sections $v_i(t)$ as before and require that the $v_i(t)$'s are flat sections;  
 
this is \emph{not} the action that we will use. Rather, we consider the natural connection on the trivial vector bundle $T\times V$:
$$ D(\sum_i c_i(t) v_i) = \sum_i (Dc_i(t)) v_i \, \, (D\in \mathcal D(T)).$$
This \emph{descends} to a connection on the quotient vector bundle $(T\times V)/\Gamma$ over $Y$. Indeed, we have 
$$\mathcal D(Y) = \mathcal D(T)^\Gamma,$$ 
and moreover the action of $\mathcal D(T)$ on sections of $T\times V$ commutes with the action of $\GL(V)$. Thus, the subring $\mathcal D(Y)$ preserves $\Gamma$-invariant sections over $T$, which are precisely the sections of $V_\rho$ over $Y$. \emph{This is the action that we will be using.}

For convenience we introduce a notion of \emph{flat functional} on the vector bundle with total space $E=(T\times V)/\Gamma$. A flat functional will be an element of the dual vector space $V^*$, thought of as a \emph{flat section of the dual of the pull-back of $E$ to $T$}. It is by abuse of language that we call it a ``flat functional on $E$'' since it is really a flat section of the dual vector bundle over the \'etale cover $T$ of $Y=T/\Gamma$, \emph{not} a section over $Y$. Any section $y\mapsto f_y$ of $V_\rho$, together with a flat functional $v^*$, give rise to a function $F(f_y,v^*): t\mapsto \left<f_t,v^*\right>$ on $T$ (\emph{not} on $Y=T/\Gamma$). The action of differential operators was defined in such a way that for every section $f_y$, every flat section $v^*$ and every differential operator $D\in \mathcal D(Y)$ we have:
$$ F(Df_y,v^*) = DF(f_y,v^*).$$

\subsection{Various spaces of sections} \label{sssections}

Let $T$ now denote the maximal compact subtorus (considered as a real form) of a complex torus $T_\CC$; or, more generally, let $T$ be a \emph{torsor} (principal homogeneous space) for a compact real torus, and $T_\CC$ its complexification. Let $L$ be a finite dimensional, complex algebraic vector bundle over $T_\CC$.
We introduce the following notation for sections of $L$:

\begin{itemize}
 \item We denote by $\CC[T,L]$ the regular sections of $L$ over $T$ --- that is, over $T_\CC$.
 \item We denote by $\CC(T,L)$ the \emph{rational} sections. 
 \item We denote by $\Gamma(T,L)$ the \emph{rational} sections which are \emph{regular on the real subset} $T$; by ``regular'' we mean that their polar divisors do not intersect $T$; however, with an extra restriction on the poles which we will introduce, this will turn out to be equivalent to the weaker condition that they extend to $C^\infty$, or even $L^2$, sections (see Lemma \ref{lemmalinearpoles}).
 \item Thinking of (the set of real points of) $T$ as a smooth manifold and of $L$ as a smooth vector bundle over $T$, we denote by $C^\infty(T,L)$ the \emph{smooth} sections over $T$; it carries a canonical structure of a Fr\'echet space.
 
 If $L$ is trivializable, we have a canonical isomorphism:
$$ C^\infty(T,L) = \CC[T,L]\otimes_{\CC[T]} C^\infty(T).$$

 \item We now come to hermitian forms. The bundle of sesquilinear forms on  $L$  is the (smooth) complex vector bundle $L^*\otimes \bar L^*$ over $T_\CC$. However, for the purposes of this paper, where $T$ will parametrize unitary representations, it is more meaningful to start with sesquilinear forms over $T$, view them as bilinear pairings between a representation $\pi$ and its dual $\tilde\pi$, by identifying $\tilde\pi$ with $\bar\pi$, and extend them as such to $T_\CC$. 
 
Therefore, we will not adopt the common notation where $\bar L$ denotes the complex conjugate of $L$, but $\bar L$ will denote the complex algebraic bundle which is obtained by $L$ via base change by complex conjugation with respect to the compact real form $T$: 
$$\Res_{\CC/\RR}(T_\CC) \to \Res_{\CC/\RR}(T_\CC).$$
In other words, the vector space of sections of $\bar L$ over an open $U\subset T_\CC$ will coincide with the conjugate vector space of sections of $L$ over the complex conjugate $\bar U$, and the coordinate ring of $\bar U$ will act on them via complex conjugation:
$$ \CC[\bar U]\to \CC[U].$$

Of course, over the real form $T$ this canonically induces the same smooth complex bundle as the complex conjugate of $L$. But, one of the benefits of our definition of $\bar L$ is that now $L^*\otimes \bar L^*$ is a complex algebraic vector bundle over $T_\CC$.
 
A \emph{hermitian metric} on $L$ (over $T$) is a smooth section of $L^*\otimes \bar L^*$  over $T$ which corresponds to a positive definite hermitian form on every fiber. A hermitian metric, together with a Haar measure on $T$, give rise to the Hilbert space of $L^2$ sections of $L$. Since $T$ is compact, all hermitian metrics and Haar measures give isomorphic topological vector spaces of $L^2$-sections, although of course the Hilbert norm will depend on the choices. 
\end{itemize}

The following easy lemma will be useful:
\begin{lemma}\label{regularsmooth}
 Let $T$ be a real torus and $\Gamma$ a finite subgroup. The natural map: 
$$ \CC[T]\otimes_{\CC[T/\Gamma]} C^\infty(T/\Gamma) \to C^\infty(T)$$
is an isomorphism.
\end{lemma}

\begin{proof}
As in the beginning of \S \ref{ssgeneralities}, we can extend each complex character $\chi$ of $\Gamma$ to a character $\tilde \chi$ of $T$ ($\tilde\chi\in \CC[T]$), then write each $f\in C^\infty(T)$ as a linear combination of its $\chi$-equivariant parts:
$$ f = \sum_\chi f_\chi, \mbox{ where } f_\chi(t) = \frac{1}{|\Gamma|} \sum_\gamma \chi^{-1}(\gamma) f(\gamma t),$$
and finally $f_\chi = \tilde\chi \cdot \frac{f_\chi}{\tilde \chi}$, with the last factor an element of $C^\infty(T/\Gamma)$. This shows surjectivity. Vice versa, for any sum $\sum_i P_i \otimes f_i \in \CC[T]\otimes_{\CC[T/\Gamma]} C^\infty(T/\Gamma)$, we can similarly decompose the $P_i$'s in terms of characters of $\Gamma$, and then the $\chi$-equivariant part of the sum can be written as $\tilde\chi \otimes f_\chi$, with $f_\chi \in C^\infty(T/\Gamma)$. For the sum $\sum_\chi \tilde\chi \otimes f_\chi$ to be zero, each $f_\chi$ has to be zero, proving injectivity.
\end{proof}

\subsection{Linear poles} \label{sslinearpoles}

We continue assuming that $L$ is a complex algebraic vector bundle over a complex torus $T_\CC$, whose compact real form we denote by $T$ or, more generally, over a torsor $T$ for a compact real torus. A \emph{linear divisor} on $T$ will be the scheme-theoretic zero set of a polynomial of the form: 
\begin{equation}\label{lineardivisor}\prod_i (\chi_i - r_i)
\end{equation}
 where:
\begin{itemize}
 \item the $r_i$'s are non-zero scalars;
 \item the $\chi_i$'s are ``characters'' of $T_\CC$ --- more precisely: non-zero eigenfunctions for the torus acting on $T_\CC$.
\end{itemize}
In particular, a linear divisor is always principal. The word ``linear'' stems from the fact that under an exponential map: $\mathfrak t \to T$ (and its complexification $\mathfrak t_\CC \to T_\CC$) their preimages are unions of affine hyperplanes --- in fact, affine hyperplanes associated to the \emph{real} functionals $\sqrt{-1}\cdot d\chi_i$. 

We say that a rational section $f\in \CC(T,L)$ has linear poles if: $$\prod_i (\chi_i - r_i) f \in \CC[T,L] \,\, \mbox{ (regular sections)}$$ for a finite set of characters $\chi_i$ and complex numbers $r_i$ as above. A very crucial lemma will be the following:

\begin{lemma} \label{lemmalinearpoles}
 If $f\in \CC(T)$ has linear poles and belongs to $L^1(T)$, then it belongs to $\Gamma(T)$, i.e.\ its poles do not meet the real locus $T$.
\end{lemma}

The notion of $L^1(T)$ is defined with respect to any Haar measure on $T$.

\begin{proof}
Using the exponential map, we can pull back the function to a holomorphic function $F$ on the complexification $\mathfrak t_\CC$ of the Lie algebra, with poles along complex hyperplanes and locally integrable on the real subspace $\mathfrak t$. Thus, locally around any point on $\mathfrak t$ which without loss of generality we may assume to be the point $0$, the pullback is equal to:
$$ h\cdot \prod_i l_i^{-n_i},$$
where the $l_i$'s are real linear functionals, the $n_i$'s are positive integers, and $h$ is a holomorphic function which does not identically vanish on the zero set of any of the $l_i$'s. If such a function is locally integrable, the same is true a fortiori when the denominator is replaced by a single linear functional $l_1$, thus we may assume that $F=h\cdot l_1^{-1}$, where $h$ is not divisible by $l_1$. 

There is a (real) point of $\mathfrak t$ in any neighborhood of zero which is in the kernel of $l_1$ but not on the zero set of $h$ (otherwise, $h$ would be divisible by $l_1$). Thus, in a neighborhood of that point the function is bounded by a constant times $l_1^{-1}$, and cannot be integrable.
\end{proof}

Finally, if $L_1, L_2$ are two vector bundles as above, then $\Hom(L_1,L_2)$ is also such a vector bundle, and we can talk about its rational sections, and linear poles for those sections. In particular, we have the following easy corollary of the previous lemma:

\begin{corollary}\label{linearpolesmap}
Suppose that $M \in \CC(T, \Hom(L_1,L_2))$ has linear poles and induces a bounded map:
$$L^2(T,L_1) \to L^2(T,L_2)$$
(with respect to hermitian metrics on $L_1$, $L_2$ and a Haar measure on $T$ --- as remarked, all choices give isomorphic spaces of $L^2$ sections).
Then $M\in \Gamma(T,\Hom(L_1,L_2))$, i.e. its poles do not meet the real locus.
\end{corollary}

\begin{proof}
Locally on $T$ we may trivialize the bundles and bound the hermitian metric from below by a constant hermitian metric (with respect to the trivialization). Thus, the square of the absolute value of the fiberwise Hilbert-Schmidt norms:
 $$T\ni t\mapsto \Vert M_t\Vert^2 $$
is bounded below, locally, by a rational function with linear poles, no fewer than those of $M$. The norm of $M$ as a bounded map: $L^2(T,L_1) \to L^2(T,L_2)$ is the $L^1$-norm of this function, and the previous lemma (or rather, its proof) shows that the poles cannot meet the real locus.
\end{proof}

\section{Coinvariants and the bundles of $X$-discrete and $X$-cuspidal representations}\label{sec:coinvariants}

\subsection{Coinvariants}\label{sscoinvariants} 

For an irreducible representation $\pi$ of $G$, the space of \emph{$\pi$-coinvariants} of $\mathcal S(X)$ is the quotient of $\mathcal S(X)$ by the common kernel of all morphisms: $\mathcal S(X)\to \pi$. They can be canonically identified with:
\begin{equation} \mathcal S(X)_\pi = \left(\Hom_G(\mathcal S(X),\pi)\right)^* \otimes \pi.\end{equation}
This is a finite direct sum of copies of $\pi$, by \cite[Theorem 9.2.1]{SV}, \cite[Theorem 4]{Delorme-Plancherel}. 

A subspace of $\Hom_G(\mathcal S(X),\pi)$ corresponds to a quotient of the space $\mathcal S(X)_\pi$ of $X$-coinvariants. Let $\pi$ have unitary central character $\chi_\pi$; recall here that by \eqref{wholecenter} (cf.\ also \S \ref{generalsymmetric} for the meaning of $\mathcal Z(G)^0$ in the general case) we assume that the maximal split torus in the center of $G$ surjects onto the ``center of $X$''.
We call an element of $\Hom_G(\mathcal S(X),\pi)$ ``cuspidal'' if $\pi$ has unitary central character and the dual: 
$$ \tilde \pi \to C^\infty(X)$$
has image in the space of compactly supported functions modulo the center. We call it ``discrete'' if the dual has image in $L^2(X/\mathcal Z(X), \chi_{\tilde\pi})$, where $\chi_{\tilde\pi}$ is the central character. We call it ``tempered'' if the dual has image in the space of tempered functions or, equivalently, if the morphism extends continuously to the Harish-Chandra Schwartz space $\mathscr C(X)$. The ``continuous'' assumption will be implicit whenever we write homomorphisms from $\mathscr C(X)$.

Thus, we have natural surjections:
\begin{equation}\label{coinvtsseq}
  \mathcal S(X)_\pi \twoheadrightarrow \mathcal S(X)_{\pi,\temp} \twoheadrightarrow \mathcal S(X)_{\pi,\disc} \twoheadrightarrow \mathcal S(X)_{\pi,\cusp},
\end{equation}
where the second corresponds to tempered morphisms. The second, third and fourth are also coinvariants for the Harish-Chandra Schwartz space, i.e.\ the canonical quotient map from $\mathcal S(X)$ extends continuously to: 
\begin{equation}
 \mathscr C(X) \twoheadrightarrow \mathcal S(X)_{\pi,\temp}.
\end{equation}

If $\pi$ does not have unitary central character, we will still be using the notation $\mathcal S(X)_\pi$, $\mathcal S(X)_{\pi,\temp}$, $\mathcal S(X)_{\pi,\disc}$ and $\mathcal S(X)_{\pi,\cusp}$ for quotients of $\mathcal S(X)$ such that the corresponding morphisms: 
$$\mathcal S(X)\to \pi \,\, \mbox{ or } \,\, \tilde \pi \to C^\infty(X)$$
have the aforementioned properties up to a twist by a character of the group. 

\subsection{$X$-discrete and $X$-cuspidal components} 

Assume $X$ to be factorizable. 
We let $\hat X^\cusp$ denote the set of irreducible representations $\pi$ with unitary central character such that $\mathcal S(X)_{\pi,\cusp}\ne 0$; we let $\hat X^\disc$ denote the set of irreducible representations $\pi$ with unitary central character such that $\mathcal S(X)_{\pi,\disc}\ne 0$. Thus, $\hat X^\cusp\subset \hat X^\disc$.

Both sets have a natural topology, and split into disjoint, possibly infinite, unions of compact components which can naturally be identified with the real points of real algebraic varieties of the same dimension, each of which is a principal homogeneous space for a torus. This structure arises as follows: 

Recall from \S \ref{sec:definitions} that $X^\ab$ is a quotient variety of $X$, which is a torsor for the torus $G^\ab/\overline{H^\ab}$, and that we denote by $\hat X^\unr$ the real torus of unitary \emph{unramified} characters of this torus.

By our assumption that $X$ is factorizable, the torus $\hat X^\unr$ acts (with finite stabilizers) on $\hat X^\disc$ and $\hat X^\cusp$. Indeed, any morphism $M:\mathcal S(X)\to \pi$ can be ``twisted'' by any element  $\omega\in \hat X^\unr$ by fixing a base point $x_0\in X^\ab$ to identify this space with the abelian quotient of $G$ of which it is a torsor, and considering $\omega$ as a function on $X$. We then define $M_\omega \in \Hom_G(\mathcal S(X),\pi\otimes \omega)$ by:
\begin{equation}\label{Momega} M_\omega (\Phi) = M(\Phi \cdot \omega).
\end{equation}
It is clear that $M_\omega$ is discrete (resp.\ cuspidal) iff $M$ is.

By \cite[Theorem 9.2.1]{SV}, \cite[Theorem 4]{Delorme-Plancherel}, we have: 
\begin{proposition}\label{finiteds}
 For each open compact subgroup $J$ of $G$, the set of $\hat X^\unr$-orbits on elements of $\hat X^\disc$ with non-zero $J$-fixed vectors is finite.
\end{proposition}

In particular, the set of $\hat X^\unr$-orbits on $\hat X^\disc$ is countable, and the action endows the latter with a real algebraic structure. 
We denote by $\hat X^\cusp_\CC$, $\hat X^\disc_\CC$, $\hat X^\unr_\CC$ the complex points of these varieties.

\subsection{The bundles $\mathcal L_\pi$, $\mathscr L_\pi$} \label{ssthebundles}

Consider the associations: 
$$\hat X^\cusp_\CC \ni \pi\mapsto  \mathcal L_\pi:=\mathcal S(X)_{\pi,\cusp},$$
$$\hat X^\disc_\CC \ni \pi\mapsto  \mathscr L_\pi:=\mathcal S(X)_{\pi,\disc}.$$

The twisting \eqref{Momega} allows us to consider these spaces as fibers of complex algebraic vector bundles $\mathcal L$, $\mathscr L$ over $\hat X^\cusp_\CC$, resp.\ $\hat X^\disc_\CC$ endowed with (slightly noncanonical) flat connections, following the formalism of \S \ref{ssgeneralities}. 
More precisely, we will use this formalism for the $J$-coinvariants $\mathcal L^J$, $\mathscr L^J$ (where $J$ is any compact open subgroup), which will be vector bundles supported over a finite number of connected components by Proposition \ref{finiteds} and with finite-dimensional fibers, and then we will define the space of sections of $\mathcal L$, $\mathscr L$ to be direct limits over all $J$ of $J$-invariant sections.

For notational simplicity, we only discuss the case of $\mathscr L^J$ ($X$-discrete series); the other is identical and, after all, it is just a quotient of $\mathscr L^J$ (and, as we shall see later, also a direct summand). 

We will exhibit the vector bundle $\mathscr L^J$ over $\hat X^\disc$ as a vector bundle of the form $V_\rho$ in the notation of \S \ref{ssgeneralities}, with $T = \hat X^\unr_\CC$ and $Y = $ a connected component of $\hat X^\disc_\CC$. To identify $Y$ with a quotient of $T$, we need to fix a base point $\pi$ in this connected component, which we take in the unitary set $\hat X^\disc$. 

The vector space $V$ will be the coinvariant space $\mathcal S(X)_{\pi,\disc}^J$, identified as a linear space (not equivariantly) with $V_\omega:=\mathcal S(X)_{\pi\otimes\omega,\disc}^J$, for \emph{any} $\omega\in \hat X^\unr_\CC$.

The way to perform this linear identification: 
\begin{equation}\label{beta}
\beta_\omega: V = \mathcal S(X)_{\pi,\disc}^J \xrightarrow\sim \mathcal S(X)_{\pi\otimes\omega,\disc}^J = V_\omega
\end{equation}
is by fixing a point $x_0\in X^\ab$, as before, giving rise to the twisting $M \mapsto M_\omega$ of \eqref{Momega}, viewed here as a linear isomorphism:
$$\Hom_G(\mathcal S(X),\pi)_\disc \to \Hom_G(\mathcal S(X),\pi\otimes \omega)_\disc.$$ 
Moreover, the underlying vector space of the representation $\pi\otimes\omega$ is naturally identified with the vector space of $\pi$, thus we get linear isomorphisms between the spaces of discrete coinvariants:
\begin{equation}\label{coinvts} \mathcal S(X)_{\pi,\disc} = \left(\Hom_G(\mathcal S(X),\pi)_\disc\right)^* \otimes \pi.
 \end{equation}

This shows that the association:
\begin{equation}\label{sheafonchars}\hat X^\unr_\CC \ni \omega\mapsto \mathcal S(X)_{\pi\otimes\omega,\disc}^J,
\end{equation}
has the structure of an almost canonically (depending on the choice of $x_0$) trivial vector bundle, thus also an algebraic vector bundle (independently of $x_0$) with total space $T\times V$ over $T=\hat X^\unr_\CC$. (We remind that $V=\mathcal S(X)_{\pi,\disc}^J$.)

Now notice that
$\mathcal S(X)_{\pi,\disc}$ is a \emph{canonical} quotient of $\mathcal S(X)$ which depends only on the isomorphism class of $\pi$ and not on its realization. In particular, if $\pi\simeq\pi\otimes\omega$ for some character $\omega$ of $G$ then we have a \emph{canonical} isomorphism: 
\begin{equation}\label{coinvisom}\alpha_\omega: V=\mathcal S(X)^J_{\pi,\disc} \simeq \mathcal S(X)_{\pi\otimes \omega,\disc}^J=V_\omega.
\end{equation}
(This is obvious if we think of $\mathcal S(X)$ as the quotient by the common kernel of all morphisms $\mathcal S(X)\to \pi$; to explain it in terms of the isomorphism \eqref{coinvts},
we notice that the difference between any two choices of isomorphisms: $\pi\xrightarrow{\sim} \pi\otimes \omega$ is a scalar which gets cancelled when we tensor $\pi$ with the linear dual of $\Hom_G(\mathcal S(X),\pi)_\disc$.) This isomorphism, in general, is not the same as the linear isomorphism $\beta_\omega$ defined above for every $\omega$. The composition $\beta_\omega^{-1}\circ \alpha_\omega$ defines a representation $\rho$ of the stabilizer in $\hat X^\unr$ of $\pi$ on $V$.

Now we endow $\mathscr L^J$ with the structure of a complex algebraic vector bundle over $\hat X_\CC^\disc$, by declaring that the bundle \eqref{sheafonchars} is simply its pull-back under the orbit map $\omega\mapsto \pi\otimes\omega$. In the notation of \S \ref{ssgeneralities} we have $\mathscr L^J = V_\rho$, where $V=\mathcal S(X)_{\pi,\disc}^J$, $T=\hat X_\CC^\unr$, $\Gamma=$ is the subgroup of those $\omega$ such that $\pi\otimes\omega \simeq \pi$,
and $\rho(\omega)=\beta_\omega^{-1}\circ \alpha_\omega$.

Hence, $\mathscr L^J$ is the vector bundle over $Y=T/\Gamma$ with total space $(T\times V)/\Gamma$, and, by repeating the same process for each connected component, a vector bundle over $\hat X^\disc_\CC$. 
The algebraic structure does not depend on the choice of basepoint $\pi$ for our representations, or base point $x_0$ on $X$. However, the corresponding flat connection, and hence the notion of \emph{flat functionals} of \S \ref{ssgeneralities} depends on the choice of base point $x_0$ up to a character of the torus $\hat X^\unr_\CC$. More precisely, the ``flat functionals'' are the functionals $f_\omega\mapsto \left<(M_\omega) (f_\omega), w^*\right> $, with $w^*\in V^*$. This dependence will not play any role in our statements.

It is clear from the definitions that the natural map:
$$ \mathcal S(X)^J \to \mathcal S(X)_{\pi,\disc}^J$$
gives rise to regular sections of $\mathscr L^J$, as $\pi$ varies, i.e.\ it gives a canonical map:
\begin{equation}\label{canonicalmap-disc}
 \mathcal S(X)\to \CC[\hat X^\disc,\mathscr L].
\end{equation}
(Similarly, by composing with the quotient map $\mathscr L\to \mathcal L$ we get a canonical map:
\begin{equation}\label{canonicalmap-smooth}
 \mathcal S(X)\to \CC[\hat X^\cusp,\mathcal L].)
\end{equation}

We have:
\begin{proposition} \label{discreteSchwartz}
 The vector bundles $\mathscr L^J$, $\mathcal L^J$ over $\hat X^\disc_\CC$, resp.\ $\hat X^\cusp_\CC$, are trivializable (over each connected component).
 
The global (regular) sections of $\mathscr L^J$ over $\hat X^\disc_\CC$ (resp.\ $\mathcal L^J$ over $\hat X^\cusp_\CC$) are \emph{precisely} the images of elements of $\mathcal S(X)^J$ under \eqref{canonicalmap-disc}.
\end{proposition}
 
\begin{proof}
The two cases are identical, so we work with $\mathscr L^J$.
The fact that it is trivializable follows from the generalities discussed in \S \ref{sec:generalities}, but it will also be seen explicitly by the argument that follows. 

The map $\mathcal S(X)^J\to \mathcal S(X)^J_{\pi,\disc}$ is surjective for any irreducible $\pi$, hence \eqref{canonicalmap-disc}, composed with evaluation at each fiber, is surjective. Choose a finite number $f_i$ of characteristic functions on $J$-orbits $x_iJ$ on $X$ such that their images form a basis of $\mathcal S(X)^J_{\pi}$, for some fixed $\pi$. If $\hat f_i(\omega)$ denotes the image of $f_i$ in $\mathcal S(X)_{\pi\otimes\omega}$, with all those vector spaces identified with the same vector space $V$ as above, it is immediate from the above definitions that, in this common vector space, $\hat f_i(\omega) = \omega(x_i)\cdot \hat f_i(1)$. Hence, the images of the $f_i$'s form a basis for every fiber of \eqref{sheafonchars}. Since $\hat X_\CC^\disc$ is affine, these global sections trivialize the bundle.

To see that all global sections come from $\mathcal S(X)^J$, we can use Nakayama's lemma. Let $Z$ be a finitely generated subgroup of $\mathcal Z(X)$ that surjects onto $\mathcal Z(X)/(\mathcal Z(X)\cap J)$. Its group ring can be identified with $\CC[\hat Z]$, the ring of regular functions on the character group of $Z$. We have a restriction map $\hat X^\disc\to \hat Z$, hence both sides of \eqref{canonicalmap-disc} (restricted to $J$-invariants) are $\CC[\hat Z]$-modules. For every point $\chi$ of $\hat Z_\CC$ (i.e.\ for every maximal ideal of $\CC[\hat Z]$) the fiber of $\mathscr L^J$ over $\chi$ is:
\begin{equation}\label{pushforwardfiber}\oplus_{\pi \mapsto \chi} \mathcal S(X)_{\pi,\disc}^J,
\end{equation}
where the map $\pi\mapsto \chi$ is the restriction of the central character. This sum is finite (there are finitely many $X$-discrete series with given central character and non-zero $J$-fixed vectors), and $\mathcal S(X)^J$ surjects on it, because it surjects on every summand and the representations indexing the sum are irreducible and non-isomorphic. 

By Nakayama's lemma, the $\CC[\hat Z]$-modules $\mathcal S(X)^J$ and $\CC[\hat X_\CC^\disc,\mathscr L]$ coincide.

\end{proof}

We let:
$$\mathscr L = \lim_{\underset{J}\to}\mathscr L^J,$$
$$\mathcal L = \lim_{\underset{J}\to}\mathcal L^J$$
as direct limits of sheaves, i.e.\ the corresponding sections will be, by definition, sections of the finite-dimensional vector bundle of $J$-invariants for some open compact subgroup $J$. 

Now we will endow the bundles $\mathscr L$, $\mathcal L$ with hermitian structures, coming from the Plancherel formula for $X$. 

The Hilbert space $L^2(X)$ has an orthogonal direct sum decomposition $L^2(X)=L^2(X)_\disc \oplus L^2(X)_\cont$, where $L^2(X)_\disc$ has a Plancherel decomposition in terms of discrete morphisms from irreducible representations: $\pi\to C^\infty(X)$, in the sense of \S \ref{sscoinvariants}, i.e.\ with unitary central characters and in $L^2$ modulo the center.

The hermitian structure on $\mathscr L$ 
is a \emph{canonical measure on $\hat X^\disc$ valued in the space of hermitian forms on $\mathscr C(X)$}, that will be denoted 
\begin{equation}\label{Plmeasure}\left<\,\, , \,\, \right>_\pi d\pi,
 \end{equation}
characterized by the following properties:

\begin{enumerate}
\item for almost every $\pi$, the hermitian form $\left<\,\, , \,\, \right>_\pi$ is $G$-invariant, positive semi-definite, and factors through $\mathscr C(X)_{\pi,\disc} = \mathcal S(X)_{\pi,\disc} = \mathscr L_\pi$;
\item for $\Phi_1,\Phi_2\in\mathscr C(X)$, 
\begin{equation}\label{Pllocal} \left<\Phi_1,\Phi_2\right>_{L^2(X)_\disc} = \int_{\hat X^\disc} \left<\Phi_1 , \Phi_2 \right>_\pi d\pi.\end{equation}
\end{enumerate}

Of course, this measure is absolutely continuous with respect to Haar measure on $\hat X^\disc$. Choosing $d\pi$ to be a Haar measure we obtain $G$-invariant hermitian forms $\left<\,\, , \,\, \right>_\pi$ on the fibers of $\mathscr L$ over $\hat X^\disc$. These forms are actually positive definite, and ``flat'' in the following sense: 

Recall the conventions of \S \ref{sssections} for the vector bundle $\overline{\mathscr L}$; it is a complex vector bundle over $\hat X^\disc_\CC$, that only over $\hat X^\disc$ is equal to the complex dual of $\mathscr L$. Since $\bar\pi\simeq\tilde\pi$ (the smooth dual) over $\hat X^\disc$, the fiber of $\overline{\mathscr L}$ over an arbitrary $\pi \in \hat X^\disc_\CC$ can be identified with $\mathscr L_{\tilde\pi}$, the ``discrete'' $\tilde\pi$-coinvariants. The hermitian forms $\left<\,\, , \,\, \right>_\pi$ can be seen as linear functionals:
\begin{equation}\label{forms} \mathscr L_\pi \otimes \overline{\mathscr L}_\pi \to \CC,\end{equation}
for $\pi\in \hat X^\disc$. We claim that these are restrictions to $\hat X^\disc$ of \emph{flat functionals} in the sense of \S \ref{ssgeneralities}. 

Indeed, this is just another way to say the following: Fix a base point $\pi\in \hat X^\disc$, and consider the non-equivariant isomorphisms $\beta_\omega$ of \eqref{beta}.
The claim is that, with respect to these isomorphisms, the hermitian form $\left<\,\, , \,\,\right>_{\pi\otimes\omega}$ pulls back to the hermitian form $\left<\,\, , \,\, \right>_\pi$, for any fixed $\pi \in \hat X^\disc$. This is easy to see, since $X$ is factorizable; indeed, for every representation $\pi$ appearing discreetly mod center, the contribution of the family $\pi\otimes\omega$, $\omega\in \hat X^\unr$, to $L^2(X)$ is given by:
\begin{equation} \left<\Phi_1,\Phi_2\right>_{\{\pi\otimes\omega\}_\omega} = \int_{\hat X^\unr} \left< \Phi_1\cdot \omega^{-1},\Phi_2 \cdot \omega^{-1}\right>_\pi d\omega,
\end{equation}
for the Haar measure $d\omega$ whose push-forward to (the given connected component of) $\hat X^\disc$ is the measure $d\pi$. Notice that we have identified $\omega$ with a function on $X$, depending on the choice of a base point, as in the construction of $\beta_\omega$.

The forms $ (\Phi_1,\Phi_2)\mapsto \left< \Phi_1\cdot \omega^{-1},\Phi_2 \cdot \omega^{-1}\right>_\pi$ are the forms $\left<\Phi_1,\Phi_2\right>_{\pi\otimes\omega}$ appearing in \eqref{Pllocal}. To view them as restrictions to $\hat X^\disc$ of flat functionals of the form \eqref{forms} (defined for arbitrary $\pi\in \hat X^\disc_\CC$), we need to view them as bilinear forms on $\mathcal S(X)$:
$$ \Phi_1\otimes\Phi_2 \mapsto \left< \Phi_1\cdot \omega^{-1}, \overline{\Phi_2 \cdot \omega}\right>_\pi,$$
a formula which makes sense (for $\Phi_1,\Phi_2\in\mathcal S(X)$) even when $\omega$ is not unitary. It is clear from the definitions that these are indeed flat functionals in the sense of \S \ref{ssgeneralities}. Moreover, our ability to extend them off the tempered spectrum means that we can view the product 
\begin{equation}\label{volumeform}
\left<\,\, , \, \, \right> d\pi
\end{equation}
as a \emph{volume form on $\hat X^\disc_\CC$ valued in the dual of the complex vector bundle $\mathscr L_\pi \otimes \overline{\mathscr L}_\pi$}. It is completely canonical once the measure on $X$ is fixed. This will be a useful point of view in order to shift integrals such as \eqref{Pllocal} off the unitary locus.

Finally, the hermitian forms for $\pi$ unitary induce a splitting of the canonical quotient from discrete to cuspidal: $\mathscr L_\pi \to \mathcal L_\pi$. This is, of course, just the orthogonal projection to the cuspidal subspace of $L^2(X/\mathcal Z(X),\chi)_\disc$, for every unitary character $\chi$. The flatness of the hermitian forms with respect to the vector space identifications $\beta_\omega$ shows that is induced by an injection of \emph{algebraic vector bundles} over $\hat X^\disc_\CC$:
\begin{equation}\label{splitting}
\mathcal L\hookrightarrow \mathscr L
\end{equation}
splitting the canonical quotient map $\mathscr L\to \mathcal L$.

\subsection{The case of boundary degenerations}

We have a similar decomposition for the analogous spaces of $X_\Theta$, not in terms of representations of $G$ but in terms of representations of a Levi subgroup. Recall the isomorphisms \eqref{Schwartzinduced}, \eqref{HCinduced}:
$$\mathcal S(X_\Theta) = I_{\Theta^-} \mathcal S(X_\Theta^L),$$
$$\mathscr C(X_\Theta) = I_{\Theta^-} \mathscr C(X_\Theta^L).$$

For each irreducible representation $\sigma$ of $L_\Theta$, by inducing the quotient $\mathcal S(X_\Theta^L)_{\sigma,\disc}$ of $\mathcal S(X_\Theta^L)$ we get a representation: 
$$\mathcal S(X_\Theta)_{\sigma,\disc}:= \left(\Hom_{L_\Theta} (\tilde\sigma,C^\infty(X_\Theta^L))_\disc\right)^*\otimes I_{\Theta^-} \sigma,$$
together with a canonical map:
$$\mathcal S(X_\Theta) \to \mathcal S(X_\Theta)_{\sigma,\disc}.$$
(This is the ``discrete'' quotient of the space $\mathcal S(X_\Theta)_\sigma$ defined in \S 15.2.6 of \cite{SV}.) Similarly, we define the cuspidal $\sigma$-coinvariants: $\mathcal S(X_\Theta)_{\sigma,\cusp}$. In other words, the spaces of discrete and cuspidal $\sigma$-coinvariants are the quotients corresponding to morphisms:
$$I_{\Theta^-}\tilde\sigma\to C^\infty(X_\Theta^L)$$ 
which are \emph{induced by Frobenius reciprocity} from morphisms  (of the respective type):
$$\tilde \sigma \to C^\infty(X_\Theta^L).$$

The spaces $\mathcal S(X_\Theta)_{\sigma,\disc}$ form a trivializable complex vector bundle over $\widehat{X_\Theta^L}^\disc$, which we will denote by $\mathscr L_\Theta$ (again as a direct limit over $J$-invariants, to be precise). The spaces $\mathcal S(X_\Theta)_{\sigma,\cusp}$ form a trivializable complex vector bundle over $\widehat{X_\Theta^L}^\cusp$, which we will denote by $\mathcal L_\Theta$.
Again, the definition of the algebraic structure of these vector bundles is obtained by pulling back to $\widehat{X_\Theta^L}^\unr$, and they are endowed with the flat connections described in \S \ref{sec:generalities} (depending on the choice of a base point on $X_\Theta^L$).

Although the isomorphism \eqref{Leviinduced}, and the subsequent isomorphisms \eqref{HCinduced}, \eqref{Schwartzinduced}, depend on the choice of parabolic $P_\Theta^-$ in its class, it is clear that the spaces $\mathcal S(X_\Theta)_{\sigma,\disc}$, $\mathcal S(X_\Theta)_{\sigma,\cusp}$ can be considered as \emph{canonical quotients} of $\mathcal S(X_\Theta)$, and hence the vector bundles $\mathcal L_\Theta, \mathscr L_\Theta$ do not depend on choices. Indeed, the kernels of the maps $\mathcal S(X_\Theta)\to \mathcal S(X_\Theta)_{\sigma,\disc}$, $\mathcal S(X_\Theta)\to \mathcal S(X_\Theta)_{\sigma,\cusp}$ do not depend on the choice of parabolic. As in \eqref{canonicalmap-disc}, \eqref{canonicalmap-smooth}, these quotient maps give rise to canonical surjections:

\begin{equation}\label{canonicalmap-theta}
 \mathcal S(X_\Theta)\twoheadrightarrow \CC[\widehat{X_\Theta^L}^\disc,\mathscr L] \twoheadrightarrow \CC[\widehat{X_\Theta^L}^\cusp,\mathcal L_\Theta].
\end{equation}

As in the previous subsection, the Plancherel decomposition for $L^2(X_\Theta)_\disc$ gives a canonical volume form  on $\widehat{X_\Theta^L}^\disc_\CC$ valued in ($G^\diag$-invariant) linear functionals on $\mathscr L_\Theta \otimes \overline{\mathscr L_\Theta}$:
\begin{equation}\label{Plmeasure-Theta}
 \left<\,\, , \,\, \right>_\sigma d\sigma,
\end{equation}
and a splitting $\mathcal L_\Theta\hookrightarrow \mathscr L_\Theta$ of the canonical quotient map of vector bundles.

\part{Discrete and cuspidal summands}

\section{Discrete summand of the Harish-Chandra Schwartz space}\label{sec:discreteHC}

The Hilbert space $L^2(X)$ has an orthogonal direct sum decomposition $L^2(X)=L^2(X)_\disc \oplus L^2(X)_\cont$, where $L^2(X)_\disc$ has a Plancherel decomposition in terms of discrete morphisms from irreducible representations: $\pi\to C^\infty(X)$, in the sense of \S \ref{sscoinvariants}, i.e.\ with unitary central characters and in $L^2$ modulo the center. We let $\mathscr C(X)_\disc= \mathscr C(X)\cap L^2(X)_\disc$, and similarly for the spaces $X_\Theta$. 

\begin{proposition}\label{propdecomp}
Let $Y$ be a connected component of $\hat X^\disc$; it corresponds to a direct summand $L^2(X)_Y$ of $L^2(X)_\disc$ by restriction of the Plancherel measure to $Y$. The orthogonal projection of an element of $\mathscr C(X)$ to $L^2(X)_Y$ lies in $\mathscr C(X)$. In particular, the orthogonal projection of an element of $\mathscr C(X)$ to $L^2(X)_\disc$ lies in $\mathscr C(X)$, and we have a direct sum decomposition:
$$\mathscr C(X) = \mathscr C(X)_\disc \oplus \mathscr C(X)_\cont,$$
where $\mathscr C(X)_\disc = \mathscr C(X) \cap L^2(X)_\disc$ and $\mathscr C(X)_\cont = \mathscr C(X) \cap L^2(X)_\cont$.
\end{proposition}

Since for every open compact subgroup $J$ there is only a finite number of connected components $Y$ with $L^2(X)_Y^J\ne 0$, the proposition actually gives a finer decomposition of $\mathscr C(X)_\disc$:
\begin{equation}\label{finerdiscrete}
\mathscr C(X)_\disc = \oplus_Y \mathscr C(X)_Y, 
\end{equation}
where $Y$ ranges over all connected components of $\hat X^\disc$.

\begin{proof}
Since $X$ is assumed to be factorizable (cf.\ Remark \ref{sfvsf}), we may represent $\mathscr C(X)$ as in \eqref{Cfactorizable}. Clearly, ``projection to discrete'' can be defined only with respect to the action of $[G,G]$, which reduces the statement to the spaces $\mathscr C(X_i')$ in the notation of \eqref{Cfactorizable}, i.e.\ reduces the problem to the case: $\mathcal Z(X)=1$.

In this case, we recall that the space $L^2(X)_\disc^J$ is finite dimensional \cite[Theorem 9.2.1]{SV}, and its elements are $A_{X,\Theta}$-finite in a $J$-good neighborhood of $\Theta$-infinity, with \emph{strictly subunitary} exponents (i.e.\ $A_{X,\Theta}$-eigencharacters which are $<1$ in absolute value in $\mathring A_{X,\Theta}^+$). Thus, the elements of $L^2(X)_\disc^J$ belong to $\mathscr C(X)$, and the projection map: $\mathscr C(X)\to \mathscr C(X)_\disc$ is continuous. 

Notice that this argument is a generalization of the usual criterion of Casselman characterizing discrete series as those representations which appear with subunitary exponents in all directions, s.\ Kato-Takano \cite{KatoTakano} for the symmetric case.
\end{proof}

Similar decompositions hold for all the boundary degenerations $X_\Theta$; this is seen simply by inducing from the Levi varieties $X_\Theta^L$, i.e.\ it follows from \eqref{HCinduced}.

Now recall the vector bundle of $X_\Theta$-discrete series $\mathscr L_\Theta$. We have seen (\eqref{canonicalmap-theta} and  Proposition \ref{discreteSchwartz}) that, through the canonical quotient maps to coinvariants, elements of $\mathcal S(X_\Theta)$ give all regular sections of $\mathscr L_\Theta$, i.e.\ all elements of $\CC[\widehat{X_\Theta^L}^\disc,\mathscr L_\Theta]$.

Moreover, the Plancherel decomposition for $L^2(X_\Theta)_\disc$ endows the complex vector bundle $\mathscr L_\Theta$ over $\widehat{X_\Theta^L}^\disc$ with the hermitian structure that was discussed in \S \ref{ssthebundles}, extending the above map to a canonical isomorphism:
\begin{equation}\label{L2discrete}
 L^2(X_\Theta)_\disc\xrightarrow{\sim} L^2(\widehat{X_\Theta^L}^\disc,\mathscr L_\Theta).
\end{equation}

The spectral description of $\mathscr C(X_\Theta)_\disc$ is as follows: 

\begin{theorem} \label{thmdiscrete}
For every $\Theta$, the canonical quotient maps: $\mathscr C(X_\Theta)\to \mathscr C(X_\Theta)_{\sigma,\disc}$ give rise to a canonical isomorphism:
\begin{equation}\label{mapdiscreteisom} \mathscr C(X_\Theta)_\disc \simeq C^\infty(\widehat{X_\Theta^L}^\disc, \mathscr L_\Theta).
\end{equation}
\end{theorem}

The ``isomorphism'', here and throughout the paper, is in the category of $G$-representations on LF-spaces (countable strict inductive limits of Fr\'echet spaces).

\begin{proof}[Proof of Theorem \ref{thmdiscrete}]
The proof relies on the Payley-Wiener theorem for the Harish-Chandra Schwartz space of finitely generated abelian groups.

By the isomorphism \eqref{HCinduced}, it is enough to prove the theorem for $\mathscr C(X_\Theta^L)_\disc$, hence we are reduced to the case of $X_\Theta = X$, assumed factorizable. We need to show that the image of
$$\mathscr C(X)\to L^2(\hat X^\disc, \mathscr L)$$
lies in $C^\infty$, and that the resulting map
\begin{equation}\label{mapCdisc}\mathscr C(X)_\disc\to C^\infty(\hat X^\disc,\mathscr L)
\end{equation}
is an isomorphism of topological $G$-modules. 

We will first explain that it is enough to show this when $G$ is replaced by the group $G'=\mathcal Z(X)\times [G,G]$. Since the $F$-points of the latter map to a subgroup of finite index in $G(F)$, it is immediate from the definitions that the restriction map of representations is a finite covering $\hat X^\disc\to \hat X'^\disc$, the latter being the space of discrete coinvariants for $X$ under the action of $G'$. Moreover, for any $\pi'\in\hat X'^\disc$ the fiber of the corresponding bundle $\mathscr L'$ of discrete coinvariants over $\pi'$ is just the direct sum of the spaces $\mathscr L_\pi$, with $\pi$ ranging over the fiber of $\pi'$, in such a way that the inclusions $\mathscr L_\pi \to \mathscr L'_{\pi'}$ and the projections in the opposite direction are $C^\infty$, as $\pi$ varies in any small neighborhood in $\hat X^\disc$ locally isomorphic to its image in $\hat X'^\disc$ . Thus, there is an isomorphism of topological $G'$-modules
$$C^\infty(\hat X^\disc,\mathscr L) \simeq C^\infty(\hat X'^\disc,\mathscr L'),$$
and it is enough to consider the action of $G'$. 

But then, using the decomposition \eqref{Cfactorizable}, and the Paley--Wiener theorem for finitely generated abelian groups:
$$\mathscr C(\mathcal Z(X)) \simeq C^\infty(\widehat{\mathcal Z(X)})$$
(depending on a choice of Haar measure on $\mathcal Z(X)$), we get that $\mathscr C(X)_\disc$ is equal to the $\mathcal Z(G)^0\cap [G,G]$-invariant subspace of 
$$\bigoplus_i C^\infty(\widehat{\mathcal Z(X)}) \hat\otimes \mathscr C(X_i')_\disc.$$
The second factor is the direct limit over all open compact subgroups $J$ of its $J$-invariants, which are finite-dimensional, hence the completion of the tensor product here is immaterial. The space $\mathscr C(X_i')_\disc$ decomposes into a direct sum of isotypic components for the $[G,G]$-action, thus identifying $\mathscr C(\bigsqcup_i\mathcal Z(X) \times X_i)_\disc$, as a topological $G'$-module, with the analogous space 
$$ \bigoplus_i C^\infty(\widehat{\mathcal Z(X)} \times \hat X_i^\disc, \mathscr L_i),$$
of which $C^\infty(\hat X'^\disc,\mathscr L')$ is simply the space of invariants under the diagonal $\mathcal Z(G)^0\cap [G,G]$-action.

This completes the proof, but we would like to mention another way of showing that the map \eqref{mapCdisc} is onto, without appealing to the artificial decomposition \eqref{Cfactorizable}. Let $Z\subset \mathcal Z(X)$ be a free abelian subgroup such that $\mathcal Z(X)/Z$ is compact. Notice that the group ring of $Z$ is canonically isomorphic to the (complexification of the) coordinate ring of $\hat Z$, the torus of unitary characters of $Z$. Moreover, by elementary Fourier analysis, this extends to an isomorphism: 
\begin{equation}\label{PWabelian}\mathscr C(Z) \xrightarrow{\sim} C^\infty(\hat Z).
\end{equation}

By restriction of central characters we get embeddings:
$$\CC[\hat Z] \to \CC[\hat X^\unr].$$
and:
$$C^\infty(\hat Z) \to C^\infty(\hat X^\unr).$$

Recall the surjection of Proposition \ref{discreteSchwartz}:
\begin{equation}\label{dSeq} \mathcal S(X)\twoheadrightarrow \CC[\hat X^\disc,\mathscr L]
\end{equation}

The action of the Harish-Chandra Schwartz algebra of $Z$ on $\mathcal S(X)$:
$$ \mathscr C(Z)\otimes  \mathcal S(X) \to \mathscr C(X)$$
translates on the right hand side of \eqref{dSeq} as multiplication by $C^\infty(\hat Z)$. Finally, by Lemma \ref{regularsmooth} the multiplication map is surjective:

$$C^\infty(\hat Z) \otimes_{\CC[\hat Z]} \CC[\hat X^\disc,\mathscr L] \twoheadrightarrow C^\infty(\hat X^\disc, \mathscr L).$$

This shows surjectivity of \eqref{mapCdisc}. The kernel is, essentially by definition, the subspace $\mathscr C(X)_\cont$, and thus the map induces an isomorphism of $\mathscr C(X)_\disc$ with $C^\infty(\hat X^\disc,\mathscr L)$.
\end{proof}

\subsection{The discrete center of $X$}\label{ssdiscretecenter}

From Theorem \ref{thmdiscrete} we deduce that the ring $C^\infty(\widehat{X_\Theta^L}^\disc)$ of smooth functions on $\widehat{X_\Theta^L}^\disc$ acts $G$-equivariantly on the Harish-Chandra Schwartz space $\mathscr C(X_\Theta)_\disc$; we extend this action to the whole space $\mathscr C(X_\Theta)$ by demanding that it acts as zero on $\mathscr C(X_\Theta)_\cont$. We will call this ring \emph{the discrete center of $X_\Theta^L$}, and denote it by:
$$C^\infty(\widehat{X_\Theta^L}^\disc) =:\mathfrak z^\disc(X_\Theta^L).$$

In the case of $X_\Theta = X$, we can think of this as the relative analog of the discrete part of the center of the Harish-Chandra Schwartz algebra, i.e.\ the discrete part of the ``tempered Bernstein center'' of Schneider and Zink \cite{SZ2}. 

\begin{remark}\label{remarkLanglands} Maybe from the point of view of the ``relative Langlands program'' this is not quite the full ``center''. Notice that if, for $\pi\in \hat X^\disc$, the space $\mathcal S(X)_{\pi,\disc}$ has multiplicity $n>1$ as a $G$-representation, then there is a larger ring of $G$-automorphisms on the direct summand of $\mathscr C(X)$ corresponding to the connected component of $\pi$ in $\hat X^\disc$ (call $Y$ this connected component). However, this ring of $G$-automorphisms is non-commutative: it is, noncanonically, the ring $C^\infty(Y, \Mat_n)$, i.e.\ $\Mat_n$-valued smooth functions. 
The philosophy of the relative Langlands program proposed in \cite{SV} suggests that this multiplicity should be related to the number of lifts of the Langlands parameter of $\pi$ to a suitable ``$X$-distinguished parameter'' into the $L$-group ${^L G_X}$ of $X$; a more precise 
statement involves Arthur parameters and packets, and we won't get into that. That suggests that there might be a \emph{distinguished} decomposition of $\mathcal S(X)_{\pi,\disc}$ into a direct sum of multiplicity-free spaces, each corresponding to a lift of the Langlands parameter of $\pi$ to ${^L G_X}$. As $\pi$ varies in a family, this would give a decomposition of the corresponding direct summand of $\mathscr C(X)_\disc$, and the elements of the $G$-automorphism ring which preserve this decomposition would form a commutative ring, isomorphic to as many copies of $C^\infty(Y)$ as the multiplicity of $\pi$ in the discrete spectrum. This is not important for our analysis, but we mention it in order to relate the version of ``center'' that we are using here with that suggested by the Langlands picture.
\end{remark}

\section{Cuspidal part of the Schwartz space}\label{sec:cuspidal}

\subsection{Main result}

We have the following analog of Proposition \ref{propdecomp} and Theorem \ref{thmdiscrete}. We start by giving a definition for the cuspidal direct summand $\mathcal S(X)_\cusp$ of $\mathcal S(X)$.

Recall the canonical quotients: $\mathscr L_{\Theta,\sigma} = \mathcal S(X_\Theta)_{\sigma,\disc} \twoheadrightarrow  \mathcal L_{\Theta,\sigma}=\mathcal S(X_\Theta)_{\sigma,\cusp}$. We have seen that the Plancherel hermitian form-valued measure \eqref{Plmeasure-Theta} on $\widehat{X_\Theta^L}^\disc$ splits these quotients canonically; the resulting embedding of vector bundles $\mathcal L_\Theta\hookrightarrow \mathscr L_\Theta$ gives rise, by the Plancherel formula, to a subspace of $L^2(X_\Theta)_\cusp$ of $L^2(X_\Theta)_\disc$. 
 
 Let $\mathcal H(G,J)$ be the Hecke algebra of $J$-biinvariant measures on $G$.

\begin{proposition}\label{characterization1}
For a function $f\in \mathcal S(X)$, invariant under a compact open subgroup $J$, the following are equivalent:
\begin{enumerate}
\item $f\in L^2(X)_\cusp$;
\item the $\mathcal H(G,J)$-module generated by $f$ is finitely generated over $\mathcal Z(X)$;
\item the $\mathcal H(G,J)$-module generated by $f$ consists of functions that are zero on every $J$-good neighborhood of $\Theta$-infinity, for every $\Theta\ne \Delta_X$.
\end{enumerate}

\end{proposition} 

\begin{proof}
We first prove that the first statement implies the third.

An $f\in L^2(X)_\cusp^J$ has pointwise Plancherel decomposition:
\begin{equation}\label{ptw} f(x) = \int_{ \hat X^\cusp}  f^{\tilde\pi} (x) d\pi\end{equation} with $f^{\tilde\pi}\in  C^\infty(X)^{\tilde\pi}_{\cusp}$, the space spanned by the images of all those morphisms: $\tilde\pi\to C^\infty(X)$  with image in the space of functions that are compactly supported modulo $\mathcal Z(X)$. The theory of asymptotics, that we will recall in the next section, implies that all $f^{\tilde\pi}$ vanish in any $J$-good neighborhood of $\Theta$-infinity, for $\Theta\ne \Delta_X$, hence so does $f$.

The third statement implies the second, because the space of $J$-invariant, compactly supported functions that are supported in the complement of all those $J$-good neighborhoods is obviously finitely generated over $\mathcal Z(X)$, since this complement is compact modulo $\mathcal Z(X)$, and $\mathcal Z(X)$ is Noetherian.

To show that the second statement implies the first, we may without loss of generality assume that $J$ is ``good'', i.e.\ such that the functor of $J$-invariants is an equivalence of categories between representations with non-zero $J$-fixed vectors, and $\mathcal H(G,J)$-modules (cf.\ \cite[Corollaire 3.9]{BeCentre}). Indeed, if $f$ is invariant under some bigger subgroup $K$, and its $\mathcal H(G,K)$-module is finitely generated over $\mathcal Z(X)$, then the same holds for its $\mathcal H(G,J)$-module, which is of the form $\sum_i h_i \mathcal H(G,K)\cdot f$, for a finite number of elements $h_i\in\mathcal H(G,J)$.

Hence, we assume that $J$ is such, and that the space $S:=\mathcal H(G,J)\cdot f$ is finitely generated over $\mathcal Z(X)$. In the case where $\mathcal Z(X)=1$, this immediately implies that $S$ is of finite length as a $\mathcal H(G,J)$-module, hence that $f$ generates a $G$-module of finite length. Since this module belongs to $\mathcal S(X)\subset L^2(X)$, it is completely reducible, with its irreducible summands obviously in $L^2(X)_\cusp$.

In the general case, recalling that $X$ is factorizable, and writing it as in \eqref{disjoint}:
$$ X(F) = \bigsqcup_{i=1}^n \mathcal Z(X)(F)\cdot X_i'(F),$$
we notice first of all that the restriction of the space of functions $\mathcal  H(G,J) \cdot z \cdot f$ to (a certain) $X_i'$ is independent of the element $z\in \mathcal Z(X)$. Thus, this restriction is a finite-dimensional vector space which is an $\mathcal H([G,G], [G,G]\cap J)$-module, which implies by the above argument that it belongs to $L^2(X_i)_\cusp$ (the latter defined as before, replacing the group $G$ by $[G,G]$).  The restriction of the pointwise Plancherel decomposition \eqref{ptw} to $X_i$ is the pointwise Plancherel decomposition for $f|_{X_i}$. Hence, $z\cdot f^{\tilde\pi}|_{X_i}$ is supported on a fixed compact subset of $X_i$ for all $i$, $z$ and almost all $\pi$, which means that $f^{\tilde\pi}$ is compactly supported modulo $\mathcal Z(X)$, for almost all $\pi$. This proves that $f\in L^2(X)_\cusp$.

\end{proof}

The space $\mathcal S(X)\cap L^2(X)_\cusp$ of functions satisfying either of the above equivalent conditions is the cuspidal part of $\mathcal S(X)$ and will be denoted by $\mathcal S(X)_\cusp$. The same definitions hold for a Levi variety $X_\Theta^L$, and by \eqref{Schwartzinduced} this defines a subspace $\mathcal S(X_\Theta)_\cusp = I_{\Theta^-} \mathcal S(X_\Theta^L)_\cusp$ of $\mathcal S(X_\Theta)$.

\begin{theorem} \label{thmcuspidal}
For every connected component $Y$ of $\widehat{X_\Theta^L}^\cusp$, the orthogonal projection of an element of $\mathcal S(X_\Theta)$ to $L^2(X_\Theta)_{Y,\cusp}$ lies in $\mathcal S(X_\Theta)$. In particular, the orthogonal projection of an element of $\mathcal S(X_\Theta)$ to $L^2(X_\Theta)_\cusp$ lies in $\mathcal S(X_\Theta)$, and we have a direct sum decomposition:
$$\mathcal S(X_\Theta) = \mathcal S(X_\Theta)_\cusp \oplus \mathcal S(X_\Theta)_\noncusp,$$
where $\mathcal S(X_\Theta)_\cusp = \mathcal S(X_\Theta) \cap L^2(X_\Theta)_\cusp$ and $\mathcal S(X_\Theta)_\noncusp = \mathcal S(X_\Theta) \cap L^2(X_\Theta)_\cusp^\perp$.

Finally, the natural map \eqref{canonicalmap-smooth} from $\mathcal S(X_\Theta)$ to sections of $\mathcal L$ over $\widehat{X_\Theta^L}^\cusp$ is the composition of an isomorphism:
\begin{equation}\label{mapcuspidalisom}
\mathcal S(X_\Theta)_\cusp \xrightarrow{\sim} \CC[\widehat{X_\Theta^L}^\cusp,\mathcal L_\Theta],
\end{equation}
with the orthogonal projection from $\mathcal S(X_\Theta)$ to $\mathcal S(X_\Theta)_\cusp$.
\end{theorem}

\begin{proof}
First of all, by \eqref{Schwartzinduced} and the analogous isomorphism for the bundle $\mathcal L_\Theta$, the theorem is reduced to the case $X_\Theta=X$, assumed factorizable.

 Since our space is assumed to be factorizable, by \eqref{Sfactorizable} the problem is reduced to the case $\mathcal Z(X)=1$, in which case cuspidal morphisms: $\pi\to C^\infty(X)$ have image in $\mathcal S(X)$. The component $L^2(X)_{Y,\cusp}$, when $\mathcal Z(X)=1$, is spanned by the images of all those morphisms for a given $\pi$, and therefore orthogonal projection to (the finite-dimensional space) $L^2(X)_{Y,\cusp}^J$, for any fixed open compact subgroup $J$, preserves compact support. The direct sum decomposition follows.

By the fact that $\mathcal L$ is a direct summand of $\mathscr L$ (both trivializable vector bundles), and by Proposition \ref{discreteSchwartz}, the image of the map $\mathcal S(X)$ into sections of $\mathcal L$ over $\hat X^\cusp$ is equal to $\CC[\hat X^\cusp, \mathcal L]$; and the kernel is the space $\mathcal S(X)_\noncusp = \mathcal S(X) \cap L^2(X)_\cusp^\perp$. This proves the last claim.
\end{proof}

For future reference, we note that since $C^\infty(X_\Theta)$ is the smooth dual of $\mathcal S(X_\Theta)$ using the eigenmeasure that we have fixed (\S \ref{sseigenmeasures}), there is a corresponding direct sum decomposition:
$$ C^\infty(X_\Theta) = C^\infty(X_\Theta)_\cusp \oplus C^\infty(X_\Theta)_\noncusp,$$
where $C^\infty(X_\Theta)_\noncusp$ is defined as the orthogonal complement of $\mathcal S(X_\Theta)_\cusp$ and vice versa.
Of course, $\mathcal S(X_\Theta)_\cusp$ belongs to $C^\infty(X_\Theta)_\cusp$.

\subsection{The cuspidal center of $X$}\label{sscuspidalcenter}

The cuspidal center of $X_\Theta^L$ is the ring:
$$\mathfrak z^\cusp(X_\Theta^L):= \CC[\widehat{X_\Theta^L}^\cusp].$$

By Theorem \ref{thmcuspidal}, it acts naturally on $\mathcal S(X_\Theta)$, namely via the isomorphism \eqref{mapcuspidalisom} on $\mathcal S(X)_\cusp$ and as zero on $\mathcal S(X_\Theta)_\noncusp$.

Again, as in Remark \ref{remarkLanglands}, we could have a larger, noncommutative ring acting on $\mathcal S(X)_\cusp$ by $G$-automorphisms, if we wanted to take into account the multiplicity of the spaces $\mathcal S(X)_{\pi,\disc}$, but we will not consider that.

\part{Eisenstein integrals}

\section{Smooth and unitary asymptotics} \label{sec:asymptotics}

The theory of asymptotics of smooth representations \cite[\S 4]{SV} provides us with canonical morphisms (which in this paper we will call ``equivariant exponential maps''):
\begin{equation}
 e_\Theta: \mathcal S(X_\Theta)\to \mathcal S(X),
\end{equation}
characterized by the property that for a $J$-good neighborhood $N_\Theta\subset X$ of $\Theta$-infinity (s.\ \eqref{Nident}) the map $e_\Theta$ restricts to the identification of characteristic functions of $J$-orbits on $N_\Theta$ induced by \eqref{Nident}.

On the other hand, the theory of unitary asymptotics \cite[\S 11]{SV} provides us with canonical morphisms (the ``Bernstein maps''):
\begin{equation}
 \iota_\Theta: L^2(X_\Theta)\to L^2(X),
\end{equation}
characterized by the fact that they are ``asymptotically equal to $e_\Theta$'' close to $\Theta$-infinity (cf.\ \emph{loc.cit. }for details).

We can characterize the spaces $L^2(X)_\disc$, $\mathcal S(X)_\cusp$ using these maps:
\begin{proposition}\label{characterization}
 We have:
 $$\mathcal S(X)_\cusp = \bigcap_{\Theta\ne\Delta_X} \ker \left.e_\Theta^*\right|_{\mathcal S(X)},$$
 $$L^2(X)_\disc = \bigcap_{\Theta\ne\Delta_X} \ker \iota_\Theta^*.$$
\end{proposition}

\begin{proof}
 For $L^2(X)_\disc$ this is part of the Plancherel formula of \cite{SV}, \cite{Delorme-Plancherel}.
 
 For $\mathcal S(X)_\cusp$, if an element $f\in \mathcal S(X)$ is in the kernel of $e_\Theta^*$, for all $\Theta\ne \Delta_X$, then the third condition of Proposition \ref{characterization1} is satisfied. Vice versa, if that condition is satisfied, then $e_\Theta^* f = 0$ for all $\Theta\ne\Delta_X$, because, by \cite[Lemma 5.2.7]{SV}, there is no $\mathcal H(G,J)$-stable subspace of $C^\infty(X_\Theta)$ whose elements are zero on a $J$-good neighborhood of infinity. 
 \end{proof}

Under the assumptions of the present paper (in particular, in the case of symmetric varieties), and conjecturally always, these smooth and unitary asympotics have spectral expansions in terms of \emph{normalized Eisenstein integrals}.

Recall the spaces of discrete and cuspidal $\sigma$-coinvariants defined in section \ref{sec:coinvariants}; the \emph{normalized constant terms} (restricted, here, to discrete and cuspidal spectra), whose definition will be recalled in the next subsection, are certain explicitly defined morphisms:
$$E_{\Theta,\sigma, \disc}^*: \mathcal S(X)\to \mathscr L_{\Theta,\sigma} = \mathcal S(X_\Theta)_{\sigma,\disc},$$
$$E_{\Theta,\sigma, \cusp}^*: \mathcal S(X)\to \mathcal L_{\Theta,\sigma} = \mathcal S(X_\Theta)_{\sigma,\cusp},$$
the latter obtained from the former via the natural quotient maps: $\mathscr L_{\Theta,\sigma}\to \mathcal L_{\Theta,\sigma}$, which vary \emph{rationally} in $\sigma$, i.e.\ they are really the pointwise evaluations of elements:
$$ E_{\Theta, \disc}^* \in \CC\left(\widehat{X_\Theta^L}^\disc, \Hom_G(\mathcal S(X), \mathscr L_\Theta)\right),$$
$$ E_{\Theta, \cusp}^* \in \CC\left(\widehat{X_\Theta^L}^\cusp, \Hom_G(\mathcal S(X), \mathcal L_\Theta)\right),$$

Here $\Hom_G(\mathcal S(X), \mathscr L_\Theta)$ denotes the sheaf over $\widehat{X_\Theta^L}_\CC^\disc$ whose sections over an open subset $U$ is the space of $G$-morphisms: $\mathcal S(X)\to \CC[U, \mathscr L_\Theta]$ (and similarly for $\mathcal L_\Theta$ over $\widehat{X_\Theta^L}^\cusp$). 

\begin{remark}\label{remarkdef} 
The fiber of this sheaf over $\sigma\in \widehat{X_\Theta^L}^\disc$ is a priori \emph{not} identical to $\Hom_G(\mathcal S(X), \mathscr L_{\Theta,\sigma})$, since, in principle, there may be morphisms that don't extend locally to an algebraic family.
\end{remark}

A priori this sheaf could be infinite-dimensional, but we claim:
\begin{lemma}\label{coherent}
$\Hom_G(\mathcal S(X), \mathscr L_\Theta)$ is a coherent, torsion-free sheaf over $\widehat{X_\Theta^L}^\disc$. 
\end{lemma}
\begin{proof}
 Indeed, for every open compact subgroup $J$ the space $\mathcal S(X)^J$ is a \emph{finitely generated} module for the Hecke algebra $\mathcal H(G,J)$ \cite[Theorem A]{AAG}, \cite[Remark 5.1.7]{SV}. Therefore, 
$$ \Hom_G(\mathcal S(X), \mathscr L_\Theta) = \lim_{\underset{J}{\from}} \Hom_{\mathcal H(G,J)} (\mathcal S(X)^J, \mathscr L_\Theta^J),$$
and the individual $\Hom$-spaces on the right are coherent sheaves over $\widehat{X_\Theta^L}^\disc$. Moreover, for every connected component $Y$ of $\widehat{X_\Theta^L}^\disc$ there is a compact open subgroup $J$ such that
$$ \Hom_G(\mathcal S(X), \mathscr L_\Theta|_Y) = \Hom_{\mathcal H(G,J)} (\mathcal S(X)^J, \mathscr L_\Theta^J|_Y)$$
\cite[Corollaire 3.9]{BeCentre}. 
Therefore, $\Hom_G(\mathcal S(X), \mathscr L_\Theta)$ is a coherent sheaf over $\widehat{X_\Theta^L}^\disc$, and similarly for $\mathcal L_\Theta$ over $\widehat{X_\Theta^L}^\cusp$. Moreover, for every $Y$ and $J$ as above, it is a subsheaf of the locally free sheaf $(\mathscr L_\Theta^J|_Y)^S$, where $S$ is a finite set of generators of $\mathcal S(X)^J$, and hence it is torsion-free.
\end{proof}

The definition of normalized constant terms and normalized Eisenstein integrals will be recalled in the next subsection, where we will also prove the important property of regularity on the unitary set. We will also recall there the notion of a character $\omega\in \widehat{X_\Theta^L}^\unr_\CC$ being \emph{large}, denoted $\omega\gg 0$. 
Here we will take them for granted, in order to recall how they are used to express smooth and unitary asymptotics.

The normalized constant terms are adjoint to ``normalized Eisenstein integrals'', which can be described as morphisms:
$$E_{\Theta,\sigma, \disc}: \widetilde{\mathscr L_{\Theta,\sigma}} \to C^\infty(X),$$
$$E_{\Theta,\sigma, \cusp}: \widetilde{\mathcal L_{\Theta,\sigma}} \to C^\infty(X),$$
varying rationally in $\sigma$. (By $\tilde ~$ we denote smooth duals.)

Now we recall the way in which Eisenstein integrals can be used to explicate smooth and unitary asymptotics. To formulate it, start from the Plancherel formula for $X_\Theta$, which canonically attaches to every $f\in L^2(X_\Theta)^\infty_\disc$ a $C^\infty(X_\Theta)$-valued measure $f^{\tilde\sigma} d\sigma$ on $\widehat{X_\Theta^L}^\disc$. Explicitly, $f^{\tilde\sigma}$ belongs to the ``discrete $\tilde\sigma$-equivariant eigenspace of $C^\infty(X_\Theta)$'' (i.e.\ the dual of $\mathscr L_{\Theta,\sigma}$), and is characterized by the property that for every $\Phi\in \mathcal S(X_\Theta)$ we have:
\begin{equation}\label{onaxis}
\left<f, \bar\Phi\right>_{L^2(X_\Theta)} = \int_{\widehat{X_\Theta^L}^\disc} \int_{X_\Theta}  f^{\tilde\sigma}(x) \Phi(x) dx  d\sigma.
\end{equation}

When $f\in \mathcal S(X_\Theta)$ the measure $f^{\tilde\sigma} d\sigma$ extends to a $C^\infty(X_\Theta)$-valued \emph{differential form} on $\widehat{X_\Theta^L}^\disc_\CC$, and another way to describe it is as follows. Recall the canonical map \eqref{canonicalmap-disc} (adapted to $X_\Theta$):
$$ \mathcal S(X_\Theta)\to \CC[\widehat{X_\Theta^L}^\disc,\mathscr L_\Theta],$$
$$ f\mapsto (\sigma\mapsto f_{\sigma,\disc})$$
(where $f_{\sigma,\disc}$ denotes the image of $f$ in the discrete $\sigma$-coinvariants), and the canonical volume form: $\left< \, , \,\right>_\sigma d\sigma$, valued in hermitian forms on $\mathscr L_\Theta$ obtained from the discrete part of the Plancherel formula of $X_\Theta$, s.\  \eqref{Plmeasure-Theta}. Then: 
$$f^{\tilde\sigma} d\sigma = \left< \bullet , \bar f_{\sigma,\disc}\right> d\sigma$$
as differential forms valued in the smooth dual of $\mathcal S(X_\Theta)$.

We are particularly interested in the case when $f \in \mathcal S(X_\Theta)_\cusp\subset L^2(X_\Theta)_\disc$, in which the form $f^{\tilde\sigma}d\sigma$ is valued in the dual of $\mathcal L_{\Theta,\sigma}$ and supported on $\widehat{X_\Theta^L}^\cusp_\CC$.
 
Since the integrand in \eqref{onaxis} is entire and supported on $\widehat{X_\Theta^L}^\cusp$, we can shift the contour of integration and write:
\begin{equation}\label{offaxis}
\left<f,\bar\Phi\right>_{L^2(X_\Theta)} = \int_{\omega^{-1}\widehat{X_\Theta^L}^\cusp} \int_{X_\Theta} f^{\tilde\sigma}(x) \Phi(x) dx  d\sigma
\end{equation}
for \emph{any} character $\omega$ of $\widehat{X_\Theta^L}_\CC^\unr$.

\begin{theorem}[{\cite[Theorem 15.4.2]{SV}}] \label{explicitsmooth}
 For any $\omega\gg 0$, if $f\in \mathcal S(X_\Theta)_\cusp$ admits the decomposition (\ref{offaxis}) then:
\begin{equation}
  e_\Theta f (x) = \int_{\omega^{-1}\widehat{X_\Theta^L}^\cusp} E_{\Theta,\sigma,\cusp} f^{\tilde\sigma}(x) d\sigma.
\end{equation}
\end{theorem}

\begin{theorem}[{\cite[Theorem 15.6.1]{SV}, \cite[Theorem 7]{Delorme-Plancherel}}]\label{explicitunitary}
If $f\in L^2(X_\Theta)^\infty_\disc$ admits the decomposition (\ref{onaxis}), then:
\begin{equation} \iota_\Theta f(x) = \int_{\widehat{X_\Theta^L}^\disc} E_{\Theta,\sigma,\disc}f^{\tilde\sigma} (x) d\sigma.\end{equation}
\end{theorem}

 We need to extend the validity of Theorem \ref{explicitsmooth} to the cases considered in \cite{Delorme-Plancherel}. The proof of Theorem 15.4.2 in \cite{SV} carries over verbatim, up to Proposition 5.4.5 which we need to prove in the setting of \cite{Delorme-Plancherel}:

\begin{proposition}\label{support}
 There is an affine embedding $X_\Omega\hookrightarrow Y$ such that for every $\Phi\in \mathcal S(X)$, the support of $e_\Omega^*\Phi$ has compact closure in $Y$.
\end{proposition}

\begin{proof}

We choose a finite extension $E$ of our field over which $G$ splits, and we let $\bar X$, $\bar G$ etc.\ denote points over $E$. Then, in \cite[\S 5.5]{SV} there is a filtration of $\bar X$ defined by certain subsets $\bar X_{\succeq \mu}$ indexed by points $\mu$ in a rational vector space $\bar{\mathfrak a}_X$, and similarly for the spaces $\bar X_\Theta$. (We point the reader to \emph{loc.cit}.\ for definitions and the notation.) By taking intersections with $X$, $X_\Theta$, we have obtain filtrations for this space.
 
 Similarly, there is a filtration $\bar{\mathcal H}_{\ge \lambda}$ of the full Hecke algebra of $\bar G$ determined by the support of its elements, where $\lambda$ lies in a rational vector space $\bar{\mathfrak a}$ endowed with a surjective map: $\bar{\mathfrak a}\to \bar{\mathfrak a}_X$. We may analogously define a filtration of the full Hecke algebra of $G$, by imposing the same conditions on the support of its elements (considered as a subset of $\bar G$). 
 
The rest of the argument of \cite{SV} (namely, \cite[Lemma 5.5.5]{SV}, following \cite[Lemma 8.8]{BezK} and \cite[Proposition 5.4.5]{SV}) now follow verbatim, proving Proposition \ref{support}. 
\end{proof}

We complement this with a statement of moderate growth, that will be used later:

\begin{proposition}\label{moderategrowth}
 For any open compact subgroup $J$ the image of $\mathcal S(X)^J$ under $e_\Theta^*$ is a space of functions of \emph{uniformly moderate growth} on $X_\Theta$; i.e.\ there is a finite number of rational functions $F_i$, whose sets of definition cover $X_\Theta$, such that each $f \in e_\Theta^*\left(\mathcal S(X)^J\right)$ satisfies: 
 $$|f|\le C_{f} \cdot \min_i (1+ |F_i|)$$ on $X_\Theta$ (for some constant $C_{f}$ depending on $f$). 
\end{proposition}

This is \cite[Proposition 15.4.3]{SV}, whose proof holds in the general case.

\section{Definition and regularity of Eisenstein integrals}\label{sec:Eisenstein}

The normalized constant terms:
$$E_{\Theta,\sigma}^*: \mathcal S(X)\to \mathcal S(X_\Theta)_\sigma$$
are defined as the composition: $T_{\Theta,\sigma}^{-1} \circ R_{\Theta,\sigma}$, where $R_{\Theta,\sigma}$ and $T_{\Theta,\sigma}$ are operators --- essentially: spectral decompositions of Radon transforms --- fitting in a diagram:
\begin{equation}\label{Eisdiagram}
\xymatrix{
\mathcal S(X) \ar[dr]^{R_{\Theta,\sigma}} & \\
& \mathcal S(X_\Theta^h,\delta_\Theta)_\sigma \\
\mathcal S(X_\Theta)_\sigma \ar[ur]^{T_{\Theta,\sigma}}  &
}
\end{equation}

The space $X_\Theta^h$ is the space of (generic) \emph{$\Theta$-horocycles} of $X$, classifying pairs $(Q,\mathfrak O)$, where $Q$ is a parabolic in the conjugacy class opposite to that of $P_\Theta^-$ (defined in \S \ref{sec:generalities}) and $\mathfrak O$ is an orbit for its unipotent radical $U_Q$ on the open $Q$-orbit on $X$. Saying the same words about $X_\Theta$ would produce a \emph{canonically isomorphic} variety \cite[Lemma 2.8.1]{SV}, and the operators $R_\Theta$ and $T_\Theta$ are defined in completely analogous ways as operators from $\mathcal S(X)$, resp.\ $\mathcal S(X_\Theta)$, to $\mathcal S(X_\Theta^h,\delta_\Theta)_\sigma$. Thus, for notational simplicity, we only describe below the definition of the former. (From the definition it will be clear that $T_\Theta$ factors through the quotient $\mathcal S(X_\Theta)_\sigma$, which we noted in the diagram above in order to make sense of the inverse of $T_\Theta$. The operator $T_\Theta$ is essentially the \emph{standard intertwining 
operator} between induction from two opposite parabolics.)

Let $\Lambda\in \Hom_{L_\Theta} (\mathcal S(X_\Theta^L),\sigma)$, where $\sigma$ is an irreducible representation of $L_\Theta$. Recall from \S \ref{sec:generalities} that the Levi variety $X_\Theta^L$ can be identified with the quotient of the open $P_\Theta$-orbit on $X$ by its unipotent radical $U_\Theta$. We define a $P_\Theta$-morphism:
\begin{equation}\label{induced}\tilde\Lambda: \mathcal S(X)\to \sigma \otimes \delta_\Theta
\end{equation}
formally (at least) as:
$$ \tilde\Lambda(\Phi) = \Lambda\left(X_\Theta^L\ni x\mapsto \int_{U_\Theta} \Phi(xu) du\right),$$
i.e.\ by integrating over $U_\Theta$-orbits in the open $P_\Theta$-orbit and then applying the operator $\Lambda$. There are two difficulties here: First, integrating over $U_\Theta$-orbits requires fixing a measure on them; secondly the result of this integration will not be compactly supported on $X_\Theta^L$. 

Without fixing measures on $U_\Theta$-orbits, the operation (Radon transform) of integrating over them canonically takes values in a line bundle over $X_\Theta^L$ whose smooth sections we denote by $C^\infty(X_\Theta^L,\delta_\Theta)$, and admits a noncanonical isomorphism:
\begin{equation}\label{ncisom}C^\infty(X_\Theta^L,\delta_\Theta) \xrightarrow\sim C^\infty(X_\Theta^L)\otimes\delta_\Theta,\end{equation} cf.\ \cite[\S 5.4.1]{SV}. The image of ``integration over generic $U_\Theta$-orbits'' will be denoted by:
\begin{equation}\label{RTheta}\mathcal S(X) \overset{R_\Theta}{\twoheadrightarrow} C^\infty(X_\Theta^L,\delta_\Theta)_X\subset C^\infty(X_\Theta^L,\delta_\Theta).\end{equation}

Inducing this $P_\Theta$-functional to $G$, we get a $G$-morphism (denoted by the same symbol):
$$\mathcal S(X)\overset{R_\Theta}\twoheadrightarrow C^\infty(X_\Theta^h,\delta_\Theta)_X,$$
where this notation stands for the corresponding line bundle over $X_\Theta^h$.

For ease of presentation, let us now fix an isomorphism as in \eqref{ncisom}, and denote by $C^\infty(X_\Theta^L)_X$ the subspace of $C^\infty(X_\Theta^L)$ corresponding to $C^\infty(X_\Theta^L,\delta_\Theta)_X$:
\begin{equation}\label{fixiso}C^\infty(X_\Theta^L,\delta_\Theta)_X \simeq C^\infty(X_\Theta^L)_X\otimes\delta_\Theta.
\end{equation}
 
We will extend the morphism $\Lambda$ to $C^\infty(X_\Theta^L)_X$ by the usual method of meromorphic continuation: Let $\tilde v \in \tilde\sigma$, and consider the following distribution on $X_\Theta^L$:
$$\mathcal S(X_\Theta^L)\ni \Phi\mapsto \left<\Lambda(\Phi),\tilde v\right>.$$

Consider also the invariant-theoretic quotient $X_\Theta\sslash U_P = \spec k[X_\Theta]^{U_P}$. It can be shown (cf.\ \cite[Lemma 15.3.1]{SV}) that it contains $X_\Theta^L$ as an open orbit, whose preimage is precisely the open $P_\Theta$-orbit in $X$. For a character $\omega\in \widehat{X_\Theta^L}_\CC^\unr$, considered as a function on $X_\Theta^L$ (this requires fixing a base point), we write $\omega\gg 0$ if it vanishes sufficiently fast around the complement of the open orbit; the set of such characters contains an open subset of the whole character group.

Twisting by $\omega$ we get from $\Lambda$ and $\tilde v$ functionals:
\begin{equation}\label{distr}\mathcal S(X_\Theta^L)\ni \Phi\mapsto \left<\Lambda(\Phi\cdot \omega),\tilde v\right>
\end{equation}
factoring through $\sigma\otimes\omega^{-1}$-coinvariants.

Then, for $\omega\gg 0$ (not depending on the choice of $\tilde v$) and any $f\in C^\infty(X_\Theta^L)_X$, the distributions \eqref{distr} are in $L^1(X, f)$ --- i.e., they are represented by measures $\omega \cdot \Lambda^*(\tilde v)$  with $\int_{X_\Theta^L} |f| \cdot |\omega \cdot \Lambda^*(\tilde v)| <\infty$ (where $\Lambda^*$ denotes the adjoint of $\Lambda$ with image in the space of smooth measures on $X_\Theta^L$). 
That gives a natural way to extend them to $C^\infty(X_\Theta^L)_X$ as the integral 
\begin{equation}\label{tildeLambda}\tilde\Lambda (f) = \int_{X_\Theta^L} f \cdot \omega \cdot \Lambda^*(\tilde v).\end{equation}

\begin{lemma}\label{linearpoles1}
 For any $f\in C^\infty(X_\Theta^L)_X$ and $\tilde v\in \tilde \sigma$ the integral \eqref{tildeLambda}
 is rational in the variable $\omega\in \widehat{X_\Theta^L}^\unr$, with linear poles.
\end{lemma}

Recall that the notion of ``linear poles'' was defined in \S \ref{sslinearpoles}.

\begin{proof}
This follows from the theory of Igusa integrals and the proof of \cite[Proposition 15.3.6]{SV}, where it is shown that this integral has the form:
 $$ I(\omega) := \int  F \cdot |\Omega| \cdot \prod_{i} |f_i|^{s_i(\omega)},$$ 
where, denoting by $D$ the complement of the open $P_\Theta$-orbit in $X$: 
\begin{itemize}
\item[---]  the $f_i$'s are $P_\Theta$-eigenfunctions (hence regular and non-vanishing away from $D$), and the exponents $s_i(\omega)$ are such that the product $\prod_{i} |f_i|^{s_i(\omega)}$ has eigecharacter $\omega$;  
\item[---]  $F$ is the pull-back of a finite function (i.e., generalized eigenfunction) on $(F^\times)^r$, for some $r$, via an $r$-tuple of ``local coordinates'' $(g_1, \dots, g_r)$, which are rational functions whose divisor is contained in $D$;
\item[---]  $\Omega$ an algebraic volume form whose divisor is contained in $D$.
\end{itemize} 

As in \cite[Proposition 15.3.6]{SV}, we now refer to \cite{Igusa} and \cite[p.5]{Denef} for the fact that such an integral has rational continuation with linear poles. 
\end{proof}

\begin{remark}
 For symmetric spaces, an alternative proof of rationality and linearity of the poles was given by Blanc-Delorme in \cite[Theorem 2.8(iv) and Theorem 2.7(i)]{BD}.
\end{remark}

Composing this with the map $R_\Theta$ of \eqref{RTheta} we now get, for every $\Lambda\in \Hom_{L_\Theta} (\mathcal S(X_\Theta^L),\sigma)$, a rational family of $P_\Theta$-morphisms:
$$\tilde\Lambda_\omega: \mathcal S(X)\to \sigma \otimes \omega^{-1}\delta_\Theta,$$
whose specialization at $\omega=1$ (if regular) is the operator $\tilde\Lambda$.

If we let $\Lambda$ vary, this defines a rational family of $P_\Theta$-morphisms from $\mathcal S(X)$ to the coinvariant space $\mathcal S(X_\Theta^L)_{\sigma \otimes \delta_\Theta}$; recall that 
$$\mathcal S(X_\Theta^L)_{\sigma}= \Hom_{L_\Theta} (\mathcal S(X_\Theta^L),\sigma)^*\otimes \sigma.$$
Inducing from $P_\Theta$, and forgetting the isomorphism \eqref{fixiso}, we land in the coinvariant space:
$$\mathcal S(X_\Theta^h,\delta_\Theta)_\sigma := \left(\Hom_{L_\Theta} (\mathcal S(X_\Theta^L,\delta_\Theta), \sigma\otimes\delta_\Theta\right)^*\otimes I_{\Theta} \sigma,$$

This completes the definition of the map $R_{\Theta,\sigma}$, and the definition of $T_{\Theta,\sigma}$ is completely analogous. Notice that $T_{\Theta,\sigma}$ factors through the quotient $\mathcal S(X_\Theta)\to \mathcal S(X_\Theta)_\sigma$, and essentially coincides with the standard intertwining operator $I_{\Theta^-}(\sigma)\to I_\Theta(\sigma)$:
$$\mathcal S(X_\Theta)_{\sigma}= \Hom_{L_\Theta} (\mathcal S(X_\Theta^L),\sigma)^*\otimes I_{\Theta^-} \sigma $$
$$ \longrightarrow \mathcal S(X_\Theta^h,\delta_\Theta)_\sigma = \Hom_{L_\Theta} (\mathcal S(X_\Theta^L,\delta_\Theta), \sigma\otimes\delta_\Theta)^*\otimes I_{\Theta} \sigma.$$
Notice that $\Hom_{L_\Theta} (\mathcal S(X_\Theta^L),\sigma) = \Hom_{L_\Theta} (\mathcal S(X_\Theta^L)\otimes \delta_\Theta, \sigma\otimes\delta_\Theta) $ canonically, and the non-canonical difference between $\mathcal S(X_\Theta^L)\otimes\delta_\Theta$ and $\mathcal S(X_\Theta^L,\delta_\Theta)$ accounts for the non-canonicity of the choice of measure on $U_\Theta$ for the intertwining operator.

We have the following:

\begin{lemma}\label{invertible} 
For $\sigma$ in general position (in a family of irreducible representations of $L_\Theta$ twisted by elements of $\widehat{X_\Theta^L}^\unr$) the representation $I_{\Theta^-}(\sigma)$ is irreducible and the operator $T_{\Theta,\sigma}$ is an isomorphism. 
\end{lemma}

\begin{proof}
Indeed, this follows from the fact that $\widehat{X_\Theta^L}^\unr$ contains ``$P_\Theta$-regular'' characters --- i.e.\ characters which are non-trivial on the image of all coroots corresponding to roots in the unipotent radical of $P_\Theta$ (s.\ the proof of \cite[Corollary 15.3.3]{SV}) --- that $T_{\Theta,\sigma}$ is always non-zero (wherever defined), and that the points of reducibility are contained in divisors of the form $\{\sigma_0\otimes\omega\}_\omega$, where $\sigma_0$ is some fixed point in this family and $\omega$ varies along unramified characters of $L_\Theta$ satisfying $\omega|_{\check\alpha(F^\times)} = 1$ for some root $\alpha$ in the unipotent radical of $P_\Theta$, cf.\ \cite[Th\'eor\`eme 3.2]{Sauvageot} and also Lemma \ref{lemmainduced} later in the present paper.
\end{proof}

Hence, the inverses of the operators $T_{\Theta,\sigma}$ form a meromorphic family of operators, and we can define
the normalized constant terms:
$$E_{\Theta,\sigma}^* = T_{\Theta,\sigma}^{-1} \circ R_{\Theta,\sigma}: \mathcal S(X)\to \mathcal S(X_\Theta)_\sigma.$$
The normalized Eisenstein integrals are by definition their adjoints:
\begin{equation}
 E_{\Theta,\sigma}: \widetilde{\mathcal S(X_\Theta)_\sigma} \to C^\infty(X).
\end{equation}

\begin{lemma}\label{linearpoles}
The normalized constant terms $E_{\Theta,\sigma}^*$ (equivalently, the normalized Eisenstein integrals $E_{\Theta,\sigma}$) 
are rational in $\sigma$, with linear poles. 
\end{lemma} 

\begin{proof}
The rationality and linear poles of $R_\Theta$, $T_\Theta$ follow from standard Igusa theory, as we recalled in Lemma \ref{linearpoles1}. The fact that the inverse $T_\Theta^{-1}$ of the standard intertwining operator has linear poles is also well-known, but we recall an argument here, for the sake of completeness. By abuse of notation, we will denote by $T_\Theta$ the standard intertwining operator $I_{\Theta^-}(\sigma)\to I_\Theta(\sigma)$, depending on a choice of Haar measure on $U_\Theta$. We also denote by $T_{\Theta^-}$ the corresponding operator when the roles of $P_\Theta$ and $P_{\Theta^-}$ are reversed.

The composition $T_{\Theta^-}\circ T_\Theta: I_{\Theta^-}(\sigma)\to I_{\Theta^-}(\sigma)$ is a scalar $\gamma(\sigma)$, varying rationally as $\sigma$ is twisted by unramified characters of $P_\Theta$. If that scalar is zero, that representation is reducible. (This is \cite[Theorem 6.6.2]{Casselman-notes}, which is stated for $\sigma$ cuspidal, but the argument in this direction works for any $\sigma$.) As we saw in the proof of Lemma \ref{invertible}, the induced representation is reducible along linear divisors. Thus, the zeros of $\gamma$ and the poles of $T_\Theta^{-1}$ are linear divisors.
\end{proof}

Now we project normalized constant terms to discrete and cuspidal quotients of $\mathcal S(X_\Theta)_\sigma$, to obtain the morphisms used in the previous section:
$$ E_{\Theta, \disc}^* \in \CC\left(\widehat{X_\Theta^L}^\disc, \Hom_G(\mathcal S(X), \mathscr L_\Theta)\right),$$
$$ E_{\Theta, \cusp}^* \in \CC\left(\widehat{X_\Theta^L}^\cusp, \Hom_G(\mathcal S(X), \mathcal L_\Theta)\right).$$

Linearity of the poles and their role in the Plancherel formula imply that the poles actually do not meet the unitary set:
\begin{proposition}\label{regularEisenstein}
 Normalized Eisenstein integrals are regular on the subsets of unitary representations, i.e.:
$$ E_{\Theta,\disc}^* \in \Gamma\left(\widehat{X_\Theta}^\disc, \Hom_G(\mathcal S(X),\mathscr L_\Theta)\right),$$
and:
$$ E_{\Theta,\cusp}^* \in \Gamma\left(\widehat{X_\Theta}^\cusp, \Hom_G(\mathcal S(X),\mathcal L_\Theta)\right).$$
\end{proposition}

\begin{proof}
Theorem \ref{explicitunitary} shows that for every $\Phi\in \mathcal S(X)$, and every $L^2$-section $\sigma\mapsto v_\sigma$ of $\mathscr L_\Theta$, the inner product:
$$ \left< E_{\Theta,\disc}^* \Phi, v\right>_\sigma d\sigma$$
(see \eqref{Plmeasure-Theta} for the unitary structure on $\mathscr L_\Theta$) is integrable over $\widehat{X_\Theta^L}^\disc$. 

Dividing by a Haar measure $d\sigma$, and assuming that $v$ is actually \emph{regular}, we get a function $\sigma\mapsto \left< E_{\Theta,\disc}^* \Phi, v\right>_\sigma $ which has linear poles. By Lemma \ref{lemmalinearpoles}, these poles cannot meet the unitary spectrum $\widehat{X_\Theta^L}^\disc$.
\end{proof}

\part{Scattering}

\section{Goals} \label{sec:scatteringgoals}

This part is technically at the heart of our proof of Paley--Wiener theorems. The goal here is to prove Theorems \ref{unitaryscattering}, \ref{smoothscattering}; before we formulate them, let us explain what we will mean by saying that for associates $\Theta,\Omega\subset \Delta_X$, and $w\in W_X(\Omega,\Theta)$, a morphism:
\begin{equation}\label{ThetaOmega}\mathcal S(X_\Theta)_\cusp\to C^\infty(X_\Omega)_\cusp
\end{equation}
is ``$w$-equivariant with respect to the cuspidal center'' $\mathfrak z^\cusp(X_\Theta^L)$ (\S \ref{sscuspidalcenter}), and similarly that a morphism:
$$\mathscr C(X_\Theta)_\disc\to \mathscr C(X_\Omega)_\disc$$
is ``$w$-equivariant with respect to the discrete center'' $\mathfrak z^\disc(X_\Theta^L)$ (\S \ref{ssdiscretecenter}).

Conjugation by $w$ induces an isomorphism between the Levi $L_\Theta$ and the Levi $L_\Omega$ (unique up to conjugacy), and hence between their unitary duals. It is not a priori clear that this preserves the discrete and cuspidal subsets:
$$\widehat{X_\Theta^L}^\disc\xrightarrow\sim \widehat{X_\Omega^L}^\disc,$$
$$\widehat{X_\Theta^L}^\cusp\xrightarrow\sim \widehat{X_\Omega^L}^\cusp,$$ 
however this will be implicit (and hence will be proven) whenever we say that a morphism of the form \eqref{ThetaOmega} is ``$w$-equivariant''. Recall from \S \ref{sscuspidalcenter}, \ref{ssdiscretecenter} that the cuspidal, resp.\ discrete center of $X_\Theta$ can be identified with regular, resp.\ smooth functions on $\widehat{X_\Theta^L}^\cusp$, resp.\ $\widehat{X_\Theta^L}^\disc$. Thus, $w$ induces isomorphisms:
$$ \mathfrak z^\cusp(X_\Theta^L) \xrightarrow{\sim} \mathfrak z^\cusp(X_\Omega^L)$$
and:
$$ \mathfrak z^\disc(X_\Theta^L) \xrightarrow{\sim} \mathfrak z^\disc(X_\Omega^L),$$
and by saying that the map is ``$w$-equivariant'' we mean with respect to this isomorphism. Notice that, by duality to $\mathcal S(X_\Omega)$, the space $C^\infty(X_\Omega)$ decomposes as a direct sum:
$$ C^\infty(X_\Omega) = C^\infty(X_\Omega)_\cusp \oplus C^\infty(X_\Omega)_\noncusp,$$
and the action of $\mathfrak z^\cusp(X_\Omega^L)$ on $C^\infty(X_\Omega)_\cusp$ is defined by duality in such a way that it extends the action on $\mathcal S(X_\Omega)_\cusp$: for $Z\in \CC[\widehat{X_\Omega^L}^\cusp_\CC]$ we let $Z^\vee$ denote the dual element: $Z^\vee(\pi)=Z(\tilde\pi)$ and we define $Z\cdot f$, for each $f\in C^\infty(X_\Omega)_\cusp$ by the property:
$$ \int_{X_\Omega} \Phi \cdot (Z\cdot f) = \int_{X_\Omega} (Z^\vee\cdot \Phi) f$$
for all $\Phi\in \mathcal S(X_\Omega)_\cusp$.

Notice that $\Omega$ could be equal to $\Theta$, but $w\ne 1$, in which case the isomorphism between centers is not the identity map.

Now we state the three main theorems of scattering theory, which will be proven in the following sections.

\begin{theorem}\label{unitaryscattering} 
 Consider the composition $i_\Omega^*\circ i_\Theta$, restricted to $L^2(X_\Theta)_\disc$. It is zero unless $\Omega$ contains a $W_X$-translate of $\Theta$, and it has image in $L^2(X_\Omega)_\cont$ unless $\Omega$ is \emph{equal} to a $W_X$-translate of $\Theta$, in which case it has image in $L^2(X_\Omega)_\disc$. In the last case, it admits a decomposition:
 $$\iota_\Omega^*\circ \iota_\Theta|_{L^2(X_\Theta)_\disc} = \sum_{w\in W_X(\Omega,\Theta)} S_w,$$
 where the morphism:
 $$S_w: L^2(X_\Theta)_\disc \to L^2(X_\Omega)_\disc$$ 
 is $w$-equivariant with respect to $\mathfrak z^\disc(X_\Theta^L)$, and is an isometry.

  The operators $S_w$ restrict to continuous morphisms between Harish-Chandra Schwartz spaces:
  $$S_w: \mathscr C(X_\Theta)_\disc \to \mathscr C(X_\Omega)_\disc,$$
  and they 
  satisfy the natural associativity conditions: 
\begin{eqnarray}\label{assoc-unit}S_{w'} \circ S_w &=&  S_{w'w} \text{ for }  w\in W_X(\Omega,\Theta), w'\in W_X(\Xi,\Omega).
\end{eqnarray}
(In particular, since $S_1=1$, they are topological isomorphisms.)
 \end{theorem}

The theorem is part of the main $L^2$-scattering theorem \cite[Theorem 7.3.1]{SV}, \cite[Theorem 6]{Delorme-Plancherel}, except for two assertions: First, the condition on equivariance with respect to $\mathfrak z^\disc(X_\Theta^L)$; indeed, the condition used in \emph{loc.cit.\ }to characterize the scattering maps $S_w$ was $w$-equivariance with respect to the action of $A_{X,\Theta}'$ (via $w:A_{X,\Theta}\xrightarrow{\sim}A_{X,\Omega}$ on $L^2(X_\Omega)$), where $A_{X,\Theta}'$ denotes \emph{the image of the $F$-points of $\mathcal Z(L_\Theta)$} via the quotient map: $\mathcal Z(L_\Theta)\to A_{X,\Theta}$. This condition is slightly weaker than $\mathfrak z^\disc(X_\Theta^L)$-equivariance. Secondly and most importantly, the fact that the scattering maps (continuously) preserve Harish-Chandra Schwartz spaces. Both of these will be proven in Section \ref{sec:unitaryscattering}.

\begin{theorem}\label{smoothscattering}
 Consider the composition $e_\Omega^*\circ e_\Theta$, restricted to $\mathcal S(X_\Theta)_\cusp$. It is zero unless $\Omega$ contains a $W_X$-translate of $\Theta$, and it has image in $C^\infty(X_\Omega)_\noncusp$ unless $\Omega$ is \emph{equal} to a $W_X$-translate of $\Theta$, in which case it has image in $C^\infty(X_\Omega)_\cusp$. In the last case, it admits a decomposition:
 $$e_\Omega^*\circ e_\Theta|_{\mathcal S(X_\Theta)_\cusp} = \sum_{w\in W_X(\Omega,\Theta)} \textswab S_w,$$
 where the morphism:
 $$\textswab S_w: \mathcal S(X_\Theta)_\cusp \to C^\infty(X_\Omega)_\cusp$$ 
 is $w$-equivariant with respect to $\mathfrak z^\cusp(X_\Theta^L)$.
 
 If we denote the subspace of $C^\infty(X_\Omega)_\cusp$ spanned by the images of all operators $\textswab S_w$ by $\mathcal S^+(X_\Omega)_\cusp$, as $\Theta$ varies and $w\in W_X(\Omega,\Theta)$, then there is a unique extension of the operators $\textswab S_w$ to the spaces $\mathcal S^+$, i.e.:
 $$\textswab S_w: \mathcal S^+(X_\Theta)_\cusp \to \mathcal S^+(X_\Omega)_\cusp$$ 
 satisfying the natural associativity conditions: 
\begin{eqnarray}\label{assoc-smooth}\textswab S_{w'} \circ \textswab S_w &=&  \textswab S_{w'w} \text{ for }  w\in W_X(\Omega,\Theta), w'\in W_X(\Xi,\Omega).
\end{eqnarray}

\end{theorem}

Both types of scattering operators have spectral expressions, which should be seen as the analogs of Theorems \ref{explicitsmooth} and \ref{explicitunitary}; to formulate them, let $w\in W_X(\Omega,\Theta)$ and denote by $w^*\mathscr L_\Omega$ the pullback of the vector bundle $\mathscr L_\Omega$ to $\widehat{X_\Theta^L}^\disc$ under the isomorphisms afforded by $w$: $\widehat{X_\Theta^L}^\disc \to \widehat{X_\Omega^L}^\disc$. We will not distinguish between sections of $\Hom_G(\mathscr L_\Theta,w^*\mathscr L_\Omega )$ over $\widehat{X_\Theta^L}^\disc$ and sections of $\Hom_G((w^{-1})^*\mathscr L_\Theta,\mathscr L_\Omega )$ over $\widehat{X_\Omega^L}^\disc$; this allows us to compose such sections for a sequence of maps: 
$$\widehat{X_\Theta^L}^\disc \to \widehat{X_\Omega^L}^\disc \to \widehat{X_\Xi^L}^\disc.$$

\begin{theorem}\label{fiberwisescattering}
 For each $w\in W_X(\Omega,\Theta)$ there is a rational family of operators, with linear poles which do not meet the unitary set:
 $$\mathscr S_w \in \Gamma\left(\widehat{X_\Theta^L}^\disc, \Hom_G(\mathscr L_\Theta,w^*\mathscr L_\Omega )\right), $$
 preserving the cuspidal direct summands $\mathcal L_\Theta, \mathcal L_\Omega$, such that the scattering operators of Theorems \ref{unitaryscattering} and \ref{smoothscattering} admit the following decompositions:
 
\begin{itemize}
 \item For any $\omega\gg 0$, if $f\in \mathcal S(X_\Theta)_\cusp$ admits the decomposition (\ref{offaxis}) then: 
\begin{equation}\label{smoothSw}
 \textswab S_w f (x) = \int_{\omega^{-1}\widehat{X_\Theta^L}^\cusp} \mathscr S_{w^{-1}}^* f^{\tilde\sigma}(x) d\sigma.
\end{equation}
 \item If $f\in L^2(X_\Theta)_\disc$ admits the decomposition (\ref{onaxis}), then:
\begin{equation}\label{unitarySw} S_w f(x) = \int_{\widehat{X_\Theta^L}^\disc} \mathscr S_{w^{-1}}^* f^{\tilde\sigma} (x) d\sigma.
\end{equation}
\end{itemize}

The operators $\mathscr S_w$ satisfy the natural associativity conditions: 
\begin{eqnarray}\label{assoc-fiberwise}\mathscr S_{w'} \circ \mathscr S_w &=&  \mathscr S_{w'w} \text{ for }  w\in W_X(\Omega,\Theta), w'\in W_X(\Xi,\Omega).
\end{eqnarray}

\end{theorem}

\begin{remark}
 The operator $\mathscr S_{w^{-1}}=\mathscr S_w^{-1}$ is a rational section of morphisms: $\mathscr L_{\Omega,{^w\sigma}}\to \mathscr L_{\Theta,\sigma}$ as $\sigma$ varies in  $\widehat{X_\Theta^L}^\disc$, and hence its  adjoint $\mathscr S_{w^{-1}}^*$ is a rational section of morphisms: $\widetilde{\mathscr L_{\Theta,\sigma}} \to \widetilde{\mathscr L_{\Omega,{^w\sigma}}}$. 
 
 The spaces $\widetilde{\mathscr L_{\Theta,\sigma}} \to \widetilde{\mathscr L_{\Omega,w^{-1}\sigma}}$ are considered as subspaces of $C^\infty(X_\Theta)$ and $C^\infty(X_\Omega)$, respectively, by duality with $\mathcal S(X_\Theta)$, resp.\ $\mathcal S(X_\Omega)$ --- recall that $f^{\tilde\sigma}\in \widetilde{\mathscr L_{\Theta,\sigma}}$.
\end{remark}

The proofs of the above theorems will occupy the rest of this part.

\section{Eigenspace decomposition of Eisenstein integrals} \label{sec:eigenspace}

Recall that the (discrete part of the) normalized constant term gives the  morphisms \eqref{Fourier}:
$$E_{\Theta,\disc}^*: \mathcal S(X)\to \CC(\widehat{X_\Theta^L}^\disc, \mathscr L_\Theta).$$

If we compose with the equivariant exponential map $e_\Omega$ (for some $\Omega\subset\Delta_X$), we get morphisms which we will denote by $E_{\Theta,\disc}^{*,\Omega}$:
\begin{equation}\label{asympEisenstein}
E_{\Theta,\disc}^{*,\Omega}:= E_{\Theta,\disc}^*\circ e_\Omega: \mathcal S(X_\Omega)\to \CC(\widehat{X_\Theta^L}^\disc, \mathscr L_\Theta) 
\end{equation}
These morphisms express the asymptotics, in the $\Omega$-direction, of normalized Eisenstein integrals. Their projection to the cuspidal part (i.e.\ the projection from $\mathscr L_\Theta$ to $\mathcal L_\Theta$) will be denoted by $E_{\Theta,\cusp}^{*,\Omega}$.

We notice that we have an action of a torus $A_{X,\Omega}$ on $\mathcal S(X_\Omega)$.
If we fix $\sigma\in \widehat{X_\Theta^L}^\disc$, any map: 
\begin{equation} \label{kmapspecialization} \mathcal S(X_\Omega)\to \mathscr L_{\Theta,\sigma}
\end{equation}
(= the fiber $\mathcal S(X_\Theta)_{\sigma,\disc}$ of $\mathscr L_\Theta$ over $\sigma$) is finite under the $A_{X,\Omega}$-action, because $\mathscr L_{\Theta,\sigma}$ if of finite length, but we will prove a stronger statement which takes into account the variation of $\sigma$. This has nothing to do with the Eisenstein integrals per se, and in fact we want to apply it to their derivatives as well (in order to prove that the normalized constant term of the Harish-Chandra space gives smooth sections over the spectrum), so we will discuss it in a more general setting.

We have already defined in \S \ref{sec:asymptotics} the sheaf $\Hom_G(\mathcal S(X), \mathscr L_\Theta)$ which by Lemma \ref{coherent} is a coherent, torsion-free sheaf over $\widehat {X_\Theta^L}_\CC^\disc$; recall that its sections over an open subset $U$ are, by definition, $G$-morphisms: $\mathcal S(X)\to \CC[U, \mathscr L_\Theta]$ --- cf.\ also Remark \ref{remarkdef}. By the same argument as in Lemma \ref{coherent}, the sheaf:
$$\mathfrak M:= \Hom_G\left(\mathcal S(X_\Omega), \mathscr L_\Theta\right)$$
is also coherent and torsion-free.

Let us fix a connected component $Y\subset \widehat{X_\Theta^L}$. We let:
\begin{equation}\label{moduleM}
\mathcal M_Y:=\CC\left(Y, \Hom_G\left(\mathcal S(X_\Omega), \mathscr L_\Theta\right)\right) = \Hom_G \left(\mathcal S(X_\Omega),\CC(Y, \mathscr L_\Theta)\right)
\end{equation}
denote the \emph{rational} global sections of this sheaf over $Y$ --- they form a finite-dimensional (by coherence) vector space over the field  $\mathcal K_Y:=\CC(Y)$.

This vector space carries a smooth $A_{X,\Omega}$-action via the action of this torus on $\mathcal S(X_\Omega)$. Thus, over a finite extension of $\mathcal K_Y$, it splits into a direct sum of generalized eigenspaces. We will describe the eigencharacters.

Let $\hat T$ denote the unitary dual of the ``universal'' split torus $T$, defined as the maximal central split torus of the Levi quotient of the minimal parabolic of $G$. In what follows, characters of a Levi $L_\Theta$ are considered as characters of $T$ via the embedding of the minimal parabolic into the parabolic $P_\Theta$. Thus, we have restriction maps:
\begin{equation}\label{restriction} \widehat{X_\Theta^L}_\CC^\unr\to \widehat{L_\Theta}_\CC^\unr \hookrightarrow \hat T_\CC,
 \end{equation}
and recall that there is also a quotient map:
\begin{equation}\label{quotient} \mathcal Z(L_\Omega)^0 \to A_{X,\Omega},
\end{equation}
whose image at the level of $F$-points we are denoting by $A_{X,\Omega}'$. 

For a (not necessarily unitary) character $\chi\in \hat T_\CC$  and an element $w\in W$, it may happen that for \emph{all} $\omega\in \widehat{X_\Theta^L}_\CC^\unr$ the restrictions:
$$ \left.{^w(\chi\omega)}\right|_{\mathcal Z(L_\Omega)^0}$$
factor through the quotient map \eqref{quotient}. For example, this is the case if $\chi$ arises as the restriction of a character in \eqref{restriction}, and $w\in W_X(\Omega,\Theta)$. We will then write:
$$ \left.{^w(\chi\omega)}\right|_{A_{X,\Omega}'},$$
and whenever we use this notation we will implicitly mean that the characters do factor through $A_{X,\Omega}'$. This is also implicitly assumed for the characters that appear in the following:

\begin{proposition}\label{finite}
Choose a base point $\sigma\in Y$, and use it to construct the finite cover: $\widehat{X_\Theta^L}^\unr\ni \omega\mapsto \sigma\otimes\omega\in Y$. Let $\widetilde{\mathcal K}_Y = \CC(\widehat{X_\Theta^L}^\unr)$ be the corresponding finite field extension. Then all eigencharacters of $A_{X,\Omega}'$ on $\mathcal M_Y$ are defined over $\widetilde{\mathcal K}_Y$.

More precisely, if $\tilde{\textswab t}:L_\Theta\to \CC[\widehat{X_\Theta^L}^\unr]^\times\subset \widetilde{\mathcal K}_Y^\times$ denotes the tautological character $a \mapsto (\omega\mapsto \omega(a))$, there is a character $\chi \in \hat T_\CC$, with restriction to $\mathcal Z(L_\Theta)^0$ equal to the central character of $\sigma$, and a subset $W_1\subset W_{L_\Omega}\backslash W$ such that 
the operator:
\begin{equation}\label{minimal}
 \prod_{w\in W_1}(z-{^w(\chi\tilde{\textswab t})} (z))
\end{equation}
annihilates $\mathcal M_Y\otimes_{\mathcal K_Y} \widetilde{\mathcal K}_Y$, for every $z\in A_{X,\Omega}'$.
\end{proposition}

We used $\tilde{\textswab t}$ for the tautological character of $L_\Theta$ into $\widetilde{\mathcal K}_Y^\times$ in order to reserve the symbol $\textswab t$ for the tautological character:
$$ \textswab t: A_{X,\Theta}' \to \CC[Y]^\times \subset \mathcal K_Y^\times,$$
$a\mapsto (\sigma\mapsto \chi_\sigma(a))$, where $\chi_\sigma$ denotes the central character of $\sigma$.

\begin{proof}
Since $\mathscr L_\Theta$ is (non-canonically) a direct sum of a finite number of copies of $I_{\Theta^-}(\bullet)$ (the sheaf over $Y$ whose local sections are regular sections $\sigma'\mapsto \phi(\sigma')\in I_{\Theta^-}(\sigma')$), it is enough to prove the proposition for the module:
 $$\mathcal M_Y' := \Hom_G\left(\mathcal S(X_\Omega), \CC(Y, I_{\Theta^-}(\bullet))\right).$$

We have: 
$$ \mathcal M_Y'\otimes_{\mathcal K_Y} \widetilde{\mathcal K}_Y \subset \mathcal M_Y'':= \Hom_G \left(\mathcal S(X_\Omega), \CC(\widehat {X_\Theta^L}^\unr_\CC, I_{\Theta^-}(\sigma\otimes\bullet))\right),$$
therefore it is enough to show that $\mathcal M_Y''$ decomposes into a direct sum of generalized eigenspaces as in the statement of the proposition.

We can represent $\sigma$ as a subrepresentation of a representation parabolically induced from a supercuspidal $\tau$, and then $I_{\Theta^-}(\sigma\otimes \omega)$ becomes a subrepresentation of $I_P(\tau\otimes\omega)$, where $P$ is a suitable parabolic, and $\tau$ is a supercuspidal representation of its Levi quotient $L$.

 Notice that: 
 $$\Hom_G\left (\mathcal S(X_\Omega), I_P(\tau\otimes\omega)\right) \overset{(*)}{\simeq} \Hom_G\left(I_P(\tilde\tau\otimes\omega^{-1}), C^\infty(X_\Omega)\right) \simeq $$
 $$ \simeq \Hom_{L_\Omega} \left(I_P(\tilde\tau\otimes\omega^{-1})_{\Omega^-}, C^\infty(X_\Omega^L)\right),$$
 where we recall that the index $~_{\Omega^-}$ denotes \emph{normalized}  Jacquet module with respect to the parabolic $P_\Omega^-$. 
 
 The Jacquet module $I_P(\tilde\tau\otimes\omega^{-1})_{\Omega^-}$ is $\mathcal Z(L_\Omega)$-finite, and it is annihilated for every $z\in \mathcal Z(L_\Omega)$ by the product:
 $$\prod_{w\in (W_{L_\Omega}\backslash W(L\to L_\Omega)/W_L)} \left(z- {^{w}(\chi_\tau\omega)^{-1}} (z)\right), $$
 (corresponding to its canonical filtration in terms of $P_\Omega^-\backslash G/P$-orbits),
where $W(L\to L_\Omega)$ denotes the set of elements $w\in W$  with ${^wL}\subset L_\Omega$, and $\chi_\tau$ is the central character of $\tau$. (We have implicitly chosen here a maximal split torus and hence a class of standard Levis, for the Weyl group to act on them.) Moreover, for $\omega$ in general position, the statement will remain true if we restrict the product to the subset $(W_{L_\Omega}\backslash W(L\to L_\Omega)/W_L )^\bullet$ of those cosets for which the restriction of elements of $^{w}(\chi_\tau \widehat{X_\Theta^L})$ to $\mathcal Z(L_\Omega)^0$ factors through the quotient \eqref{quotient}; indeed, $\mathcal Z(L_\Omega)^0$ acts on $C^\infty(X_\Omega^L)$ through the quotient $\mathcal Z(L_\Omega)^0 \to A_{X,\Omega}'$, and therefore this has to be the case for any generalized $\mathcal Z(L_\Omega)^0$-eigenspace of the Jacquet module $I_P(\tilde\tau\otimes\omega^{-1})_{\Omega^-}$ on which a morphism $I_P(\tilde\tau\otimes\omega^{-1})_{\Omega^-}\to C^\infty(X_\Omega^L)$ is non-zero.

Notice that at the step $(*)$ we have used a duality which inverts characters of $A_{X,\Omega}'$, therefore the module $\Hom_G\left (\mathcal S(X_\Omega), I_P(\tau\otimes\omega)\right) $ will be annihilated by:
 $$\prod_{w\in (W_{L_\Omega}\backslash W(L\to L_\Omega)/W_L)^\bullet} \left(z- {^{w}(\chi_\tau\omega)} (z)\right). $$

(We could alternatively have used second adjointness from the beginning, to analyze the $A_{X,\Omega}'$-action in terms of the Jacquet module $I_P(\tau\otimes\omega)_\Omega$.)
 
This essentially completes the proof of the proposition, except that the proposition was formulated with $w$ an element of $W_{L_\Omega}\backslash W$ (instead of $W_{L_\Omega}\backslash W/W_L$) and $\chi$ a character of $T$ (instead of $\chi_\tau$, a character of the center of the Levi $L$ of $P$); this formulation was chosen for uniformity, since $P$ and $L$ depend on choices, and the component of the spectrum under consideration. We may arrive at the statement of the proposition by choosing representatives in $W_{L_\Omega}\backslash W$, and choosing a character $\chi$ of $T$ which restricts to the character $\chi_\tau$ of the maximal split torus in the center of $L$.
\end{proof}

\begin{remark} \label{nochoice}
As we have seen in the last sentence of the proof, there is some choice involved in the subset $W_1\subset W_{L_\Omega}\backslash W$. Notice, however, that for cosets represented by elements $w_X\in W_X(\Omega,\Theta)$ there is no choice involved, since $W_{L_\Omega} w_X W_{L_\Theta} = W_{L_\Omega} w_X$. Moreover, for those elements we have, by construction:
\begin{equation}\label{nochoiceeq}
{^{w}(\chi\omega)} = \mbox{ the $L_\Omega$-central character of } {^w(\sigma\omega)}
\end{equation}
which, automatically, factors through $A_{X,\Omega}'$.
\end{remark}

We can think of the eigencharacters described in the last proposition as \emph{correspondences} $Y_\CC \dashrightarrow \widehat{A_{X,\Omega}'}_\CC$:

\begin{equation} \label{correspondence}\xymatrix{
& \ar@{|->}[dl] \omega\in \widehat{X_\Theta^L}^\unr \ar@{|->}[dr] & \\
\sigma\otimes \omega \in Y_\CC &&  \left.{^{w_i}(\chi\omega)}\right|_{A_{X,\Omega}'} \in \widehat{A_{X,\Omega}'}_\CC.
}\end{equation}

As a corollary of the proposition:

\begin{corollary} \label{rationaldecomp-initial}
The space $\mathcal M_Y\otimes_{\mathcal K_Y} \tilde{\mathcal K}_Y$ decomposes in a direct sum of generalized eigenspaces for the action of $A_{X,\Omega}'$, with eigencharacters ${^w(\chi\tilde{\textswab t})}$ as in the statement of Proposition \ref{finite}.
\end{corollary}

\begin{proof}
Indeed, choose a finite number of elements of $A_{X,\Omega}'$ so that they distinguish the distinct characters ${^w(\chi\tilde{\textswab t})}$. The space $\mathcal M_Y\otimes_{\mathcal K_Y} \tilde{\mathcal K}_Y$ decomposes in a direct sum of joint generalized eigenspaces of those elements, and by the proposition those have to be generalized eigenspaces for $A_{X,\Omega}'$.
\end{proof}

\subsection{Derivatives}

Now recall from \ref{ssthebundles} that $\mathscr L_\Theta$ carries a flat connection, which depends (in a very mild way) on choosing a base point $x\in X$. The resulting action of $\mathcal D(Y)$ (the ring of differential operators on $Y$) on elements of the space:
$$\mathcal M_Y=\Hom_G\left(\mathcal S(X_\Omega), \CC(Y, \mathscr L_\Theta)\right)$$
does not preserve $G$-equivariance, but it does preserve eigencharacters up to multiplicity:

\begin{lemma*}\label{derivatives}
Let $E\in \mathcal M_Y$ and let $W_E\subset W_1$, in the notation of \eqref{minimal}, be any subset such that the corresponding operator:
$$P_E(z) :=  \prod_{w\in W_E}(z-{^w(\chi\tilde{\textswab t})} (z))$$
annihilates $E$, for every $z\in A_{X,\Omega}'$.

Let $D\in \mathcal D(Y)$, so $DE \in \Hom\left(\mathcal S(X_\Omega), \CC(Y, \mathscr L_\Theta)\right)$. Then a power of $P_E(z)$ (depending only on $D$) annihilates $DE$, for every $z\in A_{X,\Omega}'$.
\end{lemma*}

\begin{proof}
For every fixed $z$, the polynomial:
$$P_{E,z}(x) = \prod_{w\in W_E}(x-{^w(\chi\tilde{\textswab t})} (z))$$
is divided by the minimal polynomial of the operator $z$ acting on $E$. The ring $\mathcal D(Y)$ acts on polynomials with coefficients in $\mathcal K_Y$, simply by acting on the coefficients. If $D\in D(Y)$ is of degree $n$, then the commutator:
$$[D, P_{E,z}^{n+1}]$$
lies in the ideal generated by $P_{E,z}$; therefore, $P_{E,z}^{n+1}(z) = P_E^{n+1}(z)$ annihilates $DE$. 
\end{proof}

\subsection{Weak tangent space of a family} \label{ssweaktangent}

In order to obtain more precise information about the characters ${^w(\chi\tilde{\textswab t})}$ that appear in the annihilator of Eisenstein integrals and their constant terms, we need a way to encode represen\-tation-theoretic information on families of representations. This information will be a ``Lie algebra'' version of the usual notion of supercuspidal support.

Recall that an irreducible representation $\sigma$ of $G$ is a subquotient of a parabolically induced supercuspidal representation $\tau$ of a Levi subgroup $L$, and the pair $(\tau,L)$ is called the \emph{supercuspidal support} of $\sigma$. It is well defined modulo $G$-conjugacy (we think of $\tau$ as an isomorphism class of representations), and the set of $G$-conjugacy classes of such pairs has a natural orbifold structure. Notationally, we can also write $(\tau,P)$ when $P$ is a parabolic with Levi subgroup $L$.

Let us denote by $SP_G$ the space of supercuspidal pairs $(\tau,L)$ and by $SC_G$ the set of their equivalence classes, which we may consider either as an orbifold or (by invariant-theoretic quotient) as an affine variety. The fiber of $SP_G$ over a fixed $L$ is acted upon with finite stabilizers by the character group $\hat L^\unr_\CC$,  and therefore the tangent space of any point on the fiber can be canonically identified with the Lie algebra of $\hat L^\unr_\CC$. 

A choice of parabolic $P$ with Levi $L$ gives an embedding $\hat L^\unr_\CC \subset \hat A^\unr_\CC$, where $A$ is the universal Cartan, and hence of the Lie algebra $\mathfrak l^*_\CC$ of the former into the Lie algebra $\mathfrak a^*_\CC = \Hom(A, \Gm) \otimes \CC$ of the latter. For a pair $x=(\tau,L)\in SP_G$ we will call \emph{weak tangent space} $WT_{G,x}$ the image of $\mathfrak l^*_\CC$ in the set-theoretic quotient $\mathfrak a^*_\CC/W$. It does not depend on $\tau$, neither on the choice of parabolic $P$.

Any two points in $SP_G$ in the preimage of a point in $SC_G$ have the same weak tangent space, thus this notion descends to the orbifold $SC_G$. Moreover, if by ``tangent space'' $T_x$ of a point $x$ on an orbifold we mean the quotient of the tangent space of a preimage on the covering manifold by the finite stabilizer, there is a well defined map: $$T_x\to WT_{G,x}$$
at every $x\in SC_G$.

Now, consider a set $\mathcal I$ of finite-length representations of $G$. Let $\mathcal I'$ be the set of isomorphism classes of irreducible subquotients of elements of $\mathcal I'$, and let $SC_G(\mathcal I)\subset SC_G$ denote the set of supercuspidal supports of elements of $\mathcal I'$. 
Then at every point $x\in SC_G(\mathcal I)$ we have a well-defined subset:
$$ WT_{G,x}(\mathcal I)\subset WT_{G,x} \subset \mathfrak a^*_\CC/W,$$
defined as the union of the weak tangent spaces at $x$ of all embedded suborbifolds: $S\subset  SC_G(\mathcal I) \subset SC_G$. We use ``embedded suborbifold'' to refer to the image in $SC_G$ of a smooth embedded submanifold of a finite (smooth manifold) cover of $SC_G$. Equivalently, $WT_{G,x}(\mathcal I)$ denotes the set of images in the ``tangent space'' $T_x$ of the derivatives at zero of all smooth one-parameter families: 
$$ \gamma: (-\epsilon,\epsilon) \to SP_G$$
with $[\gamma(0)] = x$ and $[\gamma((-\epsilon,\epsilon))]\subset \mathcal I'$, where $[\bullet]$ denotes the quotient map $SP_G\to SC_G$.

The union of the spaces $WT_{G,x}(\mathcal I)$ over all $x\in SC_G(\mathcal I)$ will be called, for brevity, the ``weak tangent space of $\mathcal I$'' instead of ``weak tangent space of the supercuspidal support of $\mathcal I$'', and denoted $WT_G(\mathcal I)$.

In what follows, we will apply these notions to Levi subgroups, instead of the group $G$. Notice that for a Levi subgroup $L$, the corresponding notion of weak tangent space for $SC_L$ gives a subset of $\mathfrak a^*_\CC/W_L$, where $W_L$ is the Weyl group of $L$. If $L=L_\Omega$ for some $\Omega\subset \Delta_X$ we will be using the index $\Omega$ instead of $L_\Omega$.

Let $\Omega\subset\Delta_X$. Let $\mathcal J$ be a family of $G$-morphisms $\{ \mathcal S(X_\Omega) \to \pi\}_\pi$, where $\pi$ varies over a set $\mathcal J_0$ of $G$-representations of finite length. Each such morphism is equivalent, by second adjointness, to a morphism:
\begin{equation}\label{Nisyros} \mathcal S(X_\Omega^L) \to \pi_\Omega,
\end{equation}
where $\pi_\Omega$ denotes the Jacquet module of $\pi$ with respect to $P_\Omega$ --- also of finite length. Let $\mathcal I$ denote the union, over all $\pi\in \mathcal J_0$, of the images of the maps \eqref{Nisyros}. We define:
$$ SC_{\Omega}(\mathcal J):= SC_{\Omega}(\mathcal I)$$
and
$$ WT_\Omega(\mathcal J):= WT_\Omega(\mathcal I),$$
the latter whenever $SC_{\Omega}(\mathcal J)$ is a suborbifold of $SC_\Omega$.

These definitions can also be given without appealing to second adjointness, of course; it suffices to dualize the morphisms:
$$ \tilde\pi\to C^\infty(X_\Omega)$$
and to use Frobenius reciprocity, as in the proof of the preceding proposition.

In the notation of the previous subsection ($Y\subset \widehat{X_\Theta^L}^\disc$, etc), an element $E$ of $\Hom_G(\mathcal S(X_\Omega),\CC(Y,\mathscr L_\Theta))$ will be considered as a family $\mathcal J$ as above by considering the evaluations of its points wherever they are defined, and we will also be using $SC(E), WT(E)$ to denote the supercuspidal support, resp.\ weak tangent space, of this family.

In this language, the proof of Proposition \ref{finite} shows:
\begin{corollary} \label{finite2}
 Let $Y\subset \widehat{X_\Theta^L}^\disc$ be a connected component and let $E \in \mathcal M_Y= \Hom_G(\mathcal S(X_\Omega),\CC(Y,\mathscr L_\Theta))$.
 
 Then:
\begin{equation}\label{SC} SC_\Omega(E) \subset \bigcup_{g\in G(x,L_\Omega)} [{^g\left(x\cdot \widehat{X_\Theta^L}^\unr_\CC\right)}],
\end{equation}
\begin{equation}\label{WT} WT_\Omega(E) = \bigcup_{w\in W'\subset W_{L_\Omega}\backslash W} [{^w\left(\mathfrak a_{X,\Theta,\CC}^*\right)}],
\end{equation}
 where:
 \begin{itemize}
  \item $x$ is a supercuspidal pair for $L_\Theta$ (i.e.\ $x\in SP_\Theta$);
  \item $G(x,L_\Omega)$ denotes the set of elements in $G$ carrying the Levi $L$ of $x$ into $L_\Omega$;
  \item $\mathfrak a_{X,\Theta,\CC}^*$ denotes the Lie algebra of $\widehat{X_\Theta^L}^\unr_\CC$, which can also be identified with the Lie algebra of $\widehat{A_{X,\Theta}}_\CC$ (hence the notation), inside of $\mathfrak a^*_\CC$;
  \item $[\bullet]$ denotes classes in $SC_\Omega$, resp.\ $WT_\Omega$;
  \item $W'$ denotes some subset of the given set.
 \end{itemize}
\end{corollary}

\begin{proof}
The proof of the corollary is essentially identical to that of Proposition \ref{finite}, if we replace the action of $\mathcal Z(L_\Omega)$ (or its quotient $A_{X,\Omega}'$) by that of the Bernstein center $\mathfrak z(L_\Omega)$. \emph{The reader should notice here that, although we are going to use the structure of the Bernstein center of $L_\Omega$ as the ring of polynomial functions on the variety of supercuspidal supports, the present corollary will only be used in the proof of Proposition \ref{fibercuspidal} when $\Omega\ne \Delta_X$; thus, we are not applying a circular argument when reproving the structure of the Bernstein center in \S \ref{ssBcenter}, since one can establish it inductively on the size of the group.}

In the notation of the proof of Proposition \ref{finite}, with $\sigma$ a base point in $Y$, choosing a basis for $\Hom_{L_\Theta} (\mathcal S(X_\Theta^L), \sigma)_\disc$ we can identify the bundle $\mathscr L_{\Theta,\sigma\otimes\bullet}$ over $\widehat{X_\Theta^L}^\unr_\CC$ with a subbundle of $I_P(\tau\otimes\bullet)^r$, for some $r$, and hence $E$ with a rational section of the bundle:
\begin{equation}\label{thisbundle}\Hom_G(\mathcal S(X_\Omega),I_P(\tau\otimes\bullet)^r).\end{equation}

At each point $\omega$ where it is regular, we get a specialization $E_\omega$, whose dual, $E_\omega^*$ is an element of:
$$\Hom_G(I_P(\tilde\tau\otimes\omega^{-1})^r, C^\infty(X_\Omega)) = \Hom_{L_\Omega}(I_P(\tilde\tau\otimes\omega^{-1})_{\Omega^-}^r, C^\infty(X_\Omega^L)). $$

The $L_\Omega$-supercuspidal support of the Jacquet module $I_P(\tilde\tau\otimes\omega^{-1})_{\Omega^-}$ consists of the classes $[({^w(\tilde\tau\otimes\omega^{-1})},{^wL})]$, where $w$ ranges in $W_{L_\Omega}\backslash W(L\to L_\Omega)/W_L$.

If we recall that $\mathfrak z(L_\Omega)$ is the set of regular functions on the variety $SC_\Omega$, each such $w$ determines by pull-back a character:
$$\mathfrak t_w: \mathfrak z(L_\Omega) \to \tilde{\mathcal K}_Y,$$
and hence as in Proposition \ref{finite} and Corollary \ref{rationaldecomp-initial}, $E$ decomposes into its generalized eigen-components with respect to the $ \mathfrak z(L_\Omega)$-action, with those eigencharacters $\mathfrak t_w$:
\begin{equation}\label{Edecomp}
E = \sum_{w\in R} E^w,
\end{equation}
where by $R$ we denote a minimal subset of representatives $R\subset W$, such that the corresponding generalized eigensummands $E^w$ are non-zero.

Thus, there is a (nonempty) Zariski open subset $U\subset \widehat{X_\Theta^L}_\CC^\unr$ such that for $\omega\in U$ the supercuspidal support of $E_{\Omega,\omega}:\mathcal S(X_\Omega^L) \to I_P(\tau\otimes\omega)_\Omega^r$ is precisely equal to the set of classes $[({^{w}(\tau\otimes\omega)},{^{w}L})]$ with $w\in R$. 

It is easy to see that the supercuspidal support of $E_{\Omega,\omega}$ at \emph{all} points where it is defined is \emph{contained} in this set of classes. Indeed, if $E$ is regular at $\omega\in \widehat{X_\Theta^L}^\unr_\CC$, and we group together the summands $E^w$ of \eqref{Edecomp} for which the specializations of the corresponding eigencharacters $\mathfrak t_w$ coincide at $\omega$ (equivalently, the classes $[{^{w}(\tau\otimes\omega)},{^{w}L}]$ coincide), then the elements of this coarser decomposition of $E$ are also regular at $\omega$, and we get the same set of supercuspidal supports. (However, the summands could vanish at some points, which explains why ``precisely equal'' was replaced by ``contained''.)

The corollary now follows.
\end{proof}

\subsection{Generic injectivity} \label{ssgenericinjectivity}

We recall the notion of ``generic injectivity of the map: $\mathfrak a_X^*/W_X\to \mathfrak a^*/W$'' in the language of \cite[\S 14.2]{SV} (for brevity we will just say: ``generic injectivity''), where $\mathfrak a_X^*= \Hom(A_X,\Gm)\otimes\QQ \subset \mathfrak a^* = \Hom(B,\Gm)\otimes \QQ$. We will also introduce a stronger version of this notion, to be termed ``strong generic injectivity'', and will show that it holds for symmetric spaces. 

For each $\Theta\subset\Delta_X$ we let:
$$\mathfrak a_{X,\Theta}^*:=\Hom(A_{X,\Theta},\Gm)\otimes\QQ \simeq \Hom((X_\Theta^L)^\unr,\Gm) \otimes \QQ,$$
which is embedded into $\mathfrak a_X^*=\mathfrak a_{X,\emptyset}^*$ by the map induced from:
$$(X_\emptyset^L)^\unr \to (X_\Theta^L)^\unr.$$

We say that $X$ satisfies the condition of \emph{generic injectivity} if the following holds:
\begin{quote}
Whenever the action of an element $w$ of the \emph{full} Weyl group $W$ on $\mathfrak a^*$ restricts to an isomorphism: 
$$\mathfrak a_{X,\Theta}^* \xrightarrow{\sim} \mathfrak a_{X,\Omega}^*$$ 
(for any $\Theta,\Omega\subset \Delta_X$, obviously of the same order, and possibly equal), there is an element of the little Weyl group $W_X$ which induces the same  isomorphism.
\end{quote}

We will say that $X$ satisfies the \emph{strong generic injectivity} condition if the following holds:
\begin{quote}
Whenever an element $w$ of the \emph{full} Weyl group $W$ on $\mathfrak a^*$ restricts to an injection: 
$$\mathfrak a_{X,\Theta}^* \hookrightarrow \mathfrak a_{X,\Omega}^*$$ 
(for any $\Theta,\Omega\subset \Delta_X$, obviously with $|\Omega|\le |\Theta|$) there is an element $w_X\in W_X$ such that:
\begin{equation}\label{actsasw} w_X|_{\mathfrak a_{X,\Theta}^*} = w|_{\mathfrak a_{X,\Theta}^*}.\end{equation}
\end{quote}

Remember from \S \ref{sec:definitions} that \textbf{we assume throughout that the strong generic injectivity condition to hold for $X$}, even though we will for emphasis repeat it in the main results of this subsection.

The condition holds for all symmetric varieties, essentially by \cite[Lemma 15]{Delorme-Plancherel}. Together with the wavefront and strong factorizability assumptions (both of which hold for symmetric spaces), it guarantees the validity of the full Plancherel decomposition \cite[Theorem 7.3.1]{SV}, \cite[Theorem 6]{Delorme-Plancherel}.

\begin{lemma}\label{stronggi-symmetric}
\begin{enumerate}
 \item If $X$ is a symmetric variety, then it satisfies the strong generic injectivity assumption. 
 \item If $X$ satisfies the strong generic injectivity condition, the element $w_X^{-1}$ in the definition of this condition can be taken to map the set of simple spherical roots $\Omega$ into $\Theta$.
\end{enumerate}
\end{lemma}

\begin{proof}
 By \cite[Lemma 15(iv)]{Delorme-Plancherel}, for every $Z\in \mathfrak a_{X,\Theta}^*$ there is an element $w_Z\in W_X$ such that:
 $$ w_Z(Z) = w(Z).$$
 Since $W_X$ is finite, an element $w_X\in W_X$ will be equal to $w_Z$ for a Zariski dense set of elements of $\mathfrak a_{X,\Theta}^*$, but then $w_X$ will actually work for all elements of $\mathfrak a_{X,\Theta}^*$.
 
 The second claim follows from known root system combinatorics: Thinking of $\mathfrak a_{X,\Theta}^*$ and $\mathfrak a_{X,\Omega}^*$ as the orthogonal complements, in $\mathfrak a_X^*$, of the sets $\Theta$, resp.\ $\Omega$, for any element $w_X\in W_X$ which maps $\mathfrak a_{X,\Theta}^*$ into $\mathfrak a_{X,\Omega}^*$ we necessarily have that $w_X^{-1}\Omega$ is in the linear span of $\Theta$.
  The set of elements in $W_X$ which satisfy \eqref{actsasw} is a union of $W_{X_\Omega}\backslash W_X/W_{X_\Theta}$-cosets. (It is actually a single coset, but that doesn't matter here.) If we choose a representative $w_X$ of minimal length for one of these cosets, then:
$$w_X \Theta >0 \mbox{ and } w_X^{-1}\Omega>0.$$
The second statement implies that $w_X^{-1}\Omega$ belongs to the positive span of $\Theta$, since it is already known to belong to its linear span;  the first statement, then, implies that it actually belongs to $\Theta$.
 \end{proof}

\begin{remark}
 The strong version of the generic injectivity condition will only be used to prove that ``scattering maps preserve cuspidality'', cf.\ Proposition \ref{fibercuspidal}. This result has been proven in a different way for symmetric varieties by \cite{CD}, relying heavily on the structure of these varieties. In each specific case, the strong generic injectivity condition is easy to check once one knows the dual group of the spherical variety; of course, it would be desirable to have a proof of this property in some more general setting.
\end{remark}

For the following lemma we identify $\mathfrak a_{X,\Omega}^*$, as we did before, with a subspace of $\mathfrak a_\Omega^*:=\Hom(P_\Theta^-,\Gm)\otimes\QQ \subset \mathfrak a^*$; on the other hand, we have a restriction map from characters of the Borel $B$ to characters of the center $\mathcal Z(L_\Omega)$ of the Levi of $P_\Omega^-$; we will write:
$$\Cent_\Omega: \mathfrak a^* \to \mathfrak a_\Omega^*\simeq \Hom(\mathcal Z(L_\Omega),\Gm)\otimes\QQ$$ for the corresponding map.

\begin{lemma}\label{GIcorollary}
Assume the strong generic injectivity condition for $X$.
Then:

\begin{enumerate}
 \item For $w\in W$ we have 
 ${^w\left(\mathfrak a_{X,\Theta,\CC}^*\right)}\subset \mathfrak a_{X,\Omega,\CC}^*$ iff $w$ is equivalent in $W/W_{L_\Theta}$ to an element $w_X\in W_X$ with $w_X^{-1} \Omega\subset \Theta$.
 
 \item For $w_X\in W_X(\Omega,\Theta)$ and $w\in W$ we have:
 $$\Cent_\Omega \circ w_X|_{\mathfrak a_{X,\Theta,\CC}^*} = \Cent_\Omega\circ w|_{\mathfrak a_{X,\Theta,\CC}^*}$$
 iff $w\equiv w_X$ in $W_{L_\Omega}\backslash W$.
\end{enumerate}
\end{lemma}

We prove the lemma below. If the meaning of the lemma is not immediately obvious, the following corollary has a representation-theoretic content related to supercuspidal supports, more precisely their ``weak tangent spaces''. Recall that we denoted by $\textswab t: A_{X,\Theta}'\to \mathcal K_Y^\times$ the tautological character.

\begin{corollary}\label{simplecharacters}
\begin{enumerate}
 \item Let $Y\subset\widehat{X_\Theta^L}^\disc$ be a connected component, and define $\mathcal M_Y$ as before. Then the only components on the right hand side of \eqref{WT} which are contained in $\mathfrak a_{X,\Omega,\CC}^*$ are those indexed by classes of elements $w_X\in W_X$ with $w_X^{-1} \Omega\subset \Theta$.
 \item  The eigencharacters ${^w\textswab t}$, $w\in W_X(\Omega,\Theta)$, appear in \eqref{minimal} with multiplicity one. The same holds if we replace $A_{X,\Omega}'$ by any subgroup of finite index.
\end{enumerate}
\end{corollary}

Strictly speaking, the discussion up to this point implies that the eigencharacters corresponding to elements of $W_X(\Omega,\Theta)$ appear with multiplicity \emph{at most} one. However, scattering theory implies that they do appear --- already in the asymptotics of Eisenstein integrals. We omit the details, since we will encounter this point later. 

\begin{proof}[Proof of Lemma \ref{GIcorollary}] 
 
For the first statement, we notice that if ${^w\left(\mathfrak a_{X,\Theta,\CC}^*\right)} \subset \mathfrak a_{X,\Omega,\CC}^* $ then, by strong generic injectivity, there is a $w_X\in W_X$ such that $w_X^{-1}\cdot w$ fixes all points of $\mathfrak a_{X,\Theta}^*$.

However, it is known that $\mathfrak a_{X,\Theta}^*$ contains strictly $P_\Theta^-$-dominant elements \cite[Proof of Corollary 15.3.2]{SV}, i.e.\ elements that are positive on each coroot corresponding to the unipotent radical of $P_\Theta^-$. Therefore the only elements of $W$ which act trivially on it are the elements of $W_{L_\Theta}$. Hence, $w \in w_X W_{L_\Theta}$.  Since $w_X$ takes $\mathfrak a_{X,\Theta}^*$ into $\mathfrak a_{X,\Omega}^*$, its inverse must map $\Omega$ into $\Theta$, by properties of root systems.

For the second statement we notice that in terms of an orthogonal $W$-invariant inner product on $\mathfrak a^*$, the operator $\Cent_\Omega$ represents the orthogonal projection onto $\mathfrak a_\Omega^*$. On the other hand, $w_X|_{\mathfrak a_{X,\Theta,\CC}^*} $ already has image in $\mathfrak a_{X,\Omega,\CC}^*\subset \mathfrak a_{\Omega,\CC}^*$, therefore the only way that 
$$\Cent_\Omega \circ w_X|_{\mathfrak a_{X,\Theta,\CC}^*} = \Cent_\Omega\circ w|_{\mathfrak a_{X,\Theta,\CC}^*}$$
is that $w|_{\mathfrak a_{X,\Theta,\CC}^*}$ also has image in $\mathfrak a_{X,\Omega}^*$. By the first statement, this implies that $w$ is equivalent to $w_X$ in $W_{L_\Omega}\backslash W$.
\end{proof}

Combining Proposition \ref{finite} with Corollary \ref{simplecharacters}, and observing that the eigencharacters ${^w\textswab t}$, $w\in W_X(\Omega,\Theta)$, are already defined over $\mathcal K_Y$, we arrive at the following strengthening of Corollary \ref{rationaldecomp-initial}:

\begin{proposition}\label{rationaldecomp}
Assume the strong generic injectivity condition for $X$. The $A_{X,\Omega}'$-module $\mathcal M_Y$ admits a decomposition:
 \begin{equation}\label{rationaldecompeq}
  \mathcal M_Y = \bigoplus_{w\in W_X(\Omega,\Theta)} \mathcal M_Y^w \,\,\, \oplus \,\,\,\mathcal M_Y^\rest,
 \end{equation}
 where $\mathcal M_Y^w$ is the (honest) 	eigenspace with eigencharacter ${^w\textswab t}$, and the space $\mathcal M_Y^\rest$ contains none of these eigencharacters.
\end{proposition}

Of course, this proposition is vacuous unless $\Theta\sim\Omega$.

\subsection{Polynomial decomposition of morphisms} 

We now return to the torsion-free sheaf 
$$\mathfrak M:=\Hom_G\left(\mathcal S(X_\Omega), \mathscr L_\Theta\right),$$
whose rational sections over a connected component $Y\subset \widehat{X_\Theta^L}^\disc$ we denoted before by $\mathcal M_Y$. We also denote by $\mathfrak M_Y$ the restriction of $\mathfrak M$ to the connected component $Y_\CC$.

Let us discuss to what extent the decomposition of Proposition \ref{rationaldecomp} extends to a decomposition of this sheaf --- the goal being to determine the poles that might get introduced when decomposing an element of $\mathcal M_Y$ as in \eqref{rationaldecompeq}. Our approach is similar to \cite[Proposition 2]{DH}, based on the theory of the resultant, which in turn was inspired by the proof of Lemma VI.2.1 in Waldspurger \cite{Waldspurger-Plancherel}, except that by considering all elements of $A_{X,\Omega}'$ simultaneously we can eliminate some unnecessary poles.

Namely, let $W_1$ be as in \eqref{minimal}, and consider only those pairs $w_X\in W_X(\Omega,\Theta)$, $w\in W_1$
 for which the equality: 
\begin{equation}\label{singular} {^{w_X}(\chi\tilde{\textswab t})}|_{A_{X,\Omega}'} = {^{w}(\chi\tilde{\textswab t})}|_{A_{X,\Omega}'},
\end{equation}
represents a divisor in $\widehat{X_\Theta^L}_\CC^\unr$.  Notice that not all pairs $(w_X,w)$ as above represent a divisor. For example, if $\Theta=\Omega=\emptyset$ so that $W_X(\Omega,\Theta)=W_X$, and $w=1$, only the pairs $(w_X, 1)$ with $w_X$ a reflection in $W_X$ represent a divisor. The images of these divisors under the quotient map $\widehat{X_\Theta^L}_\CC^\unr\ni \omega\mapsto \sigma\otimes\omega\in Y_\CC$ are divisors on $Y_\CC$, and we let $U$ denote their complement.

\begin{proposition}Assume the strong generic injectivity condition for $X$.
The restriction $\mathfrak M_U$ of the sheaf $\mathfrak M=\Hom_G\left(\mathcal S(X_\Omega), \mathscr L_\Theta\right)$ over $U$ decomposes as a direct sum of subsheaves:
$$ \mathfrak M_U = \bigoplus_{w\in W_X(\Omega,\Theta)} \mathfrak M_U^w \,\,\, \oplus \,\,\, \mathfrak M^\rest_U,$$
where $\mathfrak M_U^w$ denotes the subsheaf of $A_{X,\Theta}'$-equivariant morphisms with respect to the map $w: A_{X,\Theta}'\to A_{X,\Omega}'$.
\end{proposition}

\begin{remark}
We can be more precise about the poles of the decomposition \eqref{rationaldecompeq}. Let $E$ be any rational section of $\mathfrak M_Y$ (i.e.\ $E\in \mathcal M_Y$), and for simplicity let us consider its pull-back to $\widehat{X_\Theta^L}^\unr_\CC$ (to be denoted by the same letter). Let $f$ be any function on $Y$ whose scheme-theoretic zero locus contains the divisors \eqref{singular}; for example, 
we could take $f$ to be defined by any element $z\in A'_{X,\Omega}$ as follows:
$$f = \prod \left({^{w_X}(\chi\tilde{\textswab t})}(z) -  {^{w}(\chi\tilde{\textswab t})}(z)\right),$$
where the product ranges over all pairs $(w_X,w)$ defining divisors as above.
(Both sides are functions on $\widehat{X_\Theta^L}^\unr_\CC$.) 
Then the summands in the decomposition \ref{rationaldecompeq} of $f\cdot E$ have no more poles than $E$ itself.

We will later see (Corollary \ref{donotappear}) that for the objects that we are interested in, namely the normalized constant terms, the poles of the above form where $w$ is also in $W_X(\Omega,\Theta)$ actually do not show up.  
\end{remark}

\begin{proof}

Clearly, by the previous section, the sheaf $\mathfrak M_Y$  admits a direct sum decomposition into a finite number of eigenspaces for the maximal compact subgroup of $A_{X,\Omega}'$. Each eigencharacter defines a connected component of $\widehat{A_{X,\Omega}'}$. We fix such a component $V$ and consider $\mathfrak M_Y$ as a sheaf over $V_\CC\times Y_\CC$. (Elements of $A_{X,\Omega}'$ now restrict to polynomials over $V_\CC$.)

Set $ \mathcal R:= \CC[V \times Y].$ The annihilator of $\mathfrak M_Y$ in $\mathcal R$ is the ideal $\mathcal I$ generated by the ``minimal polynomials'' \eqref{minimal}, where $z$ ranges over all elements of $A_{X,\Omega}'$. (Clearly, a finite set of elements generating $A_{X,\Omega}'$ modulo its maximal compact subgroup suffices.)

The spectrum of the ring $\overline{\mathcal R}=\mathcal R/\mathcal I$ has a finite number of irreducible components, parametrized by orbits of the distinct factors of \eqref{minimal} under the Galois group of the extension $(\widehat{X_\Theta^L}^\unr/Y)$. We denote by $Z_w$ the components corresponding to $w\in W_X(\Omega,\Theta)$, and by $Z_\rest$ the union of the rest of the components; we use $P_w, P_\rest$ for the corresponding prime ideals.

Let $Y_\CC^\sing\subset Y_\CC$ denote the union of the images of \emph{all} subvarieties given by equations of the form \eqref{singular}, whether these equations represent divisors or subvarieties of larger codimension. 
For any $f\in \overline{\mathcal R}$ which is not a zero divisor and vanishes on $Y_\CC^\sing$, consider the localization:
$$ \mathfrak M_Y[f^{-1}]$$
which is a sheaf over the spectrum of $\overline{\mathcal R}[f^{-1}]$. 

Notice that the components $Z_\bullet \subset V_\CC\times Y_\CC$ have no intersection lying over the complement of $Y_\CC^\sing$; therefore, $\overline{\mathcal R}[f^{-1}]$ is a direct sum of integral domains, and we have a corresponding decomposition of the identity element:
\begin{equation}\label{unitdecomposition} 1 = \sum_w 1_w + 1_\rest,
\end{equation}
where $1_\bullet\in \overline{\mathcal R}[f^{-1}]$. This gives a decomposition of $ \mathfrak M_Y$ over the complement of the zero set of $f$. Since the only requirement on $f$ was that it vanishes on $Y_\CC^\sing$, we get a decomposition of $\mathfrak M_{Y_\CC\smallsetminus Y_\CC^\sing}$:

$$\mathfrak M_{Y_\CC\smallsetminus Y_\CC^\sing} = \bigoplus_{w\in W_X(\Omega,\Theta)} \mathfrak M_{Y_\CC\smallsetminus Y_\CC^\sing}^w \,\,\, \oplus \,\,\, \mathfrak M^\rest_{Y_\CC\smallsetminus Y_\CC^\sing}.$$

Finally, recall from the proof of Lemma \ref{coherent} that $\mathfrak M=\Hom_G\left(\mathcal S(X_\Omega), \mathscr L_\Theta\right)$ is a subsheaf of a locally free sheaf over $Y_\CC$. A section of $\mathscr L_\Theta$ defined in a neighborhood of a subvariety of codimension $\ge 2$ extends (uniquely) to this subvariety. Therefore, the above decomposition of $\mathfrak M_{Y_\CC\smallsetminus Y_\CC^\sing}$ extends to the complement of all divisors contained in $Y_\CC^\sing$, i.e.\ to $U$.
\end{proof}

We return to the asympotics of the normalized constant terms introduced in \eqref{asympEisenstein}:
$$E_{\Theta,\disc}^{*,\Omega}= E^*_{\Theta,\disc} \circ e_\Omega:\mathcal S(X_\Omega)\to \CC(\widehat{X_\Theta^L}^\disc, \mathscr L_\Theta).$$ 

\begin{corollary}
\label{decompEisenstein}
Let $\Omega\sim\Theta$. There is a decomposition:
  \begin{equation}\label{decompEisensteineq}
   E_{\Theta,\disc}^{*,\Omega} = \sum_{w\in W_X(\Theta,\Omega)} \mathscr S_w + \mathscr S_\Subunit,
  \end{equation}
where all summands are elements of:
$$\mathcal M= \Hom_G \left(\mathcal S(X_\Omega), \CC(\widehat{X_\Theta^L}^\disc, \mathscr L_\Theta)\right)$$
with the following properties:
\begin{enumerate}
 
 \item The operator $\mathscr S_w$ is an eigenvector of $A_{X,\Omega}'$ on $\mathcal M$; more precisely, it is 
 $w$-equivariant with respect to the action of $A_{X,\Omega}'$, i.e.\ $A_{X,\Omega}'$ acts via the character ${^w\mathfrak t}$.
 \item The operator $\mathscr S_\Subunit$ has no $w$-equivariant direct summand with respect to the action of $A_{X,\Omega}'$, for any $w\in W_X(\Omega,\Theta)$.  
 \item The poles of all summands are linear; for each component $Y\subset \widehat{X_\Theta^L}^\disc$, they are contained in the union of the poles of $E_{\Theta,\disc}^*$ and the images of divisors given by equations:
 $$ {^{w_X}(\chi\tilde{\textswab t})}|_{A_{X,\Omega}'} = {^{w}(\chi\tilde{\textswab t})}|_{A_{X,\Omega}'},$$
 for those pairs $w_X\in W_X(\Omega,\Theta),w\in W_1$ in the notation of \eqref{minimal} for which this equality represents a divisor in $\widehat{X_\Theta^L}^\unr$.
 \end{enumerate}

\end{corollary}

The statements of this corollary will be strengthened in the next couple of sections, in order to arrive at the results of \S \ref{sec:scatteringgoals}. The notation $\mathscr S_\Subunit$ is due to the fact that, as we will see in the next section using $L^2$-theory, the exponents of these morphisms over the unitary subset $\widehat{X_\Theta^L}^\disc$ are ``subunitary''. 

\subsection{Explication of the fiberwise scattering maps}\label{ssexplication}

Here we would like to emphasize here that the fiberwise scattering maps $\mathscr S_w$ play the role of ``functional equations'' between the normalized Eisenstein integrals. We use the fact that these maps are equivariant with respect to discrete centers, which is a yet-unproven statement of Theorem \ref{fiberwisescattering}, because we are not going to use the following result anywhere.

\begin{proposition} \label{propBmatrices}
Let $\Theta,\Omega$ be associates, and $w\in W_X(\Omega,\Theta)$. The corresponding fiberwise scattering map $\mathscr S_w\in \CC \left(\widehat{X_\Theta^L}^\disc, \Hom_G(\mathscr L_\Theta,w^*\mathscr L_\Omega )\right)$ is the unique rational family of maps making the following diagram commute (for almost all $\sigma\in \widehat{X_\Theta^L}^\disc$):

\begin{equation}\label{Bmatrices}
\xymatrix{
& \mathcal S(X_\Theta)_{\sigma,\disc} \ar[dd]^{\mathscr S_{w,\sigma}}\\
\mathcal S(X)  \ar[ur]^{E^*_{\Theta,\sigma}} \ar[dr]_{E^*_{\Omega,\sigma}}\\
& \mathcal S(X_\Omega)_{{^w\sigma},\disc}
}
\end{equation}

\end{proposition}

\begin{proof}
We have $\iota_\Theta f = \iota_\Omega S_w f$. By Theorem \ref{explicitunitary} and \eqref{unitarySw} this becomes:
$$ \int_{\widehat{X_\Theta^L}^\disc} E_{\Theta,\sigma,\disc} f^{\tilde\sigma} d\sigma = \int_{\widehat{X_\Theta^L}^\disc} E_{\Omega, {^w\sigma}, \disc} \mathscr S_{w^{-1}}^* f^{\tilde\sigma} d\sigma,$$
and disintegrating over $\sigma$ we get that $E_{\Theta,\sigma,\disc}$ and $E_{\Omega, {^w\sigma}, \disc} \circ \mathscr S^*_{w^{-1}} $ must be equal for almost every $\sigma$, hence equal as rational functions of $\sigma$. Using \eqref{assoc-unit}, the proposition follows by passing to adjoints.
\end{proof}

This proof is actually rather indirect, to avoid the discussion of ``small Mackey restriction'' of \cite[\S 15.5]{SV}; it can be inferred directly from this discussion, when ``injectivity of small Mackey restriction'' is known (such as in the case of symmetric varieties).

This result is essentially equivalent to the description of the constant term of Eisenstein integrals in terms of ``$B$-matrices'' and intertwining integrals in \cite[Theorem 8.4]{CD}; that work can be considered as a qualitative study of these functional equations in the case of symmetric spaces, which among other things gives some results that we prove here without relying so much on the structure of symmetric varieties, such as the fact that ``scattering maps preserve cuspidal summands'' (Theorem \ref{fiberwisescattering}).

\section{Scattering: the unitary case} \label{sec:unitaryscattering} 

The unitary asympotics (adjoints of Bernstein maps)  were obtained in \cite[\S 11.4]{SV} by filtering out the unitary exponents of the Plancherel decomposition. More precisely, given a (smooth, say) function $\Phi\in L^2(X)$ with Plancherel decomposition:
$$ \Phi(x) = \int_{\hat G} \Phi^\pi(x) \mu(\pi),$$
then it is known that $e_\Omega^* \Phi^\pi$ is $A_{X,\Omega}$-finite with only \emph{unitary} and \emph{subunitary} exponents (generalized eigencharacters) for $\mu$-almost all $\pi$. We recall the notion of \emph{subunitary exponents} \label{subunitary} for a morphism from $\mathcal S(X_\Omega)$ to a smooth representation $V$: it means that the morphism is $A_{X,\Omega}$-finite, and the image of its dual: $V\to C^\infty(X_\Omega)$ has subunitary exponents under the action of $A_{X,\Omega}$, i.e.\ generalized characters which are $<1$ on $\mathring A_{X,\Omega}^+$. (For the definition of $\mathring A_{X,\Omega}^+$ see \S \ref{sec:definitions}.)

By construction 
we have:
\begin{equation}\label{defiota} \iota^*_\Omega(\Phi) = \int_{\hat G} \left( e_\Omega^*\Phi^\pi\right)^{\operatorname{unit}} \mu(\pi),\end{equation}
where $\left( e_\Omega^*\Phi^\pi\right)^{\operatorname{unit}}$ refers to isolating the part of $\left( e_\Omega^*\Phi^\pi\right)$ with unitary generalized exponents, cf.\ \cite[Proposition 11.4.2]{SV}.

Moreover, \cite{SV}, \cite{Delorme-Plancherel} have proven Theorem \ref{unitaryscattering} restricted to $L^2(X_\Theta)_\disc$ and with the modification that the condition of $w$-equivariance with respect to $\mathfrak z^\disc(X_\Theta^L)$ be replaced by the weaker condition of $w$-equivariance with respect to $A_{X,\Theta}'$. In other words, 
\begin{equation}\label{L2scat} \iota_\Omega^*\circ \iota_\Theta|_{L^2(X_\Theta)_\disc} = \sum_{w\in W_X(\Omega,\Theta)} S_w,\end{equation}
with $S_w$ being $w$-equivariant with respect to $A_{X,\Theta}'$.

Combining all the above with Theorem \ref{explicitunitary} and Corollary \ref{decompEisenstein} we obtain:

\begin{proposition}\label{propsubunitary}
Let $\Theta,\Omega\subset \Delta_X$. For every $\sigma\in \widehat{X_\Theta^L}^\disc$ (hence unitary), the $A_{X,\Omega}$-exponents of $   E_{\Theta,\disc}^{*,\Omega}$ are unitary or subunitary.

Let $\Theta\sim \Omega$. Consider the operator $\mathscr S_\Subunit$ of Corollary \ref{decompEisenstein}. For every $\sigma\in \widehat{X_\Theta^L}^\disc$ (hence unitary) where this operator is defined (regular), the resulting morphism:
$$ \mathcal S(X_\Omega)\to \mathscr L_{\Theta,\sigma}$$
has subunitary exponents.
 
We have:
\begin{equation}\label{fiberwiseiotascat} \iota_\Omega^*\iota_\Theta f (x)  = \sum_{w\in W_X(\Omega,\Theta)} \int_{\widehat{X_\Theta^L}^\disc} \mathscr S_{w^{-1}}^* f^{\tilde\sigma} (x) d\sigma.\end{equation}
\end{proposition}

\begin{proof}
In the notation of Theorem \ref{explicitunitary}:
$$e_\Omega^*\iota_\Theta f(x) = \int_{\widehat{X_\Theta^L}^\disc} e_\Omega^*E_{\Theta,\sigma,\disc}f^{\tilde\sigma} (x) d\sigma,$$
and therefore, by the above, $e_\Omega^*E_{\Theta,\sigma,\disc}$ can only have unitary or subunitary $A_{X,\Omega}$-exponents (for almost all $\sigma$); hence, the same holds for $\mathscr S_\Subunit$.

On the other hand, \eqref{defiota} together with the property of $w$-equivariance with respect to $A_{X,\Theta}'$ of the maps $S_w$ of \eqref{L2scat} implies  that all unitary $A_{X,\Omega}'$-exponents of $e_\Omega^*E_{\Theta,\sigma,\disc}$ are contained among the exponents of the $\mathscr S_w$'s.

This proves \eqref{fiberwiseiotascat}, and it shows that $\mathscr S_\Subunit$ only has subunitary exponents.
\end{proof}

This proves assertion \eqref{unitarySw} of Theorem \ref{fiberwisescattering}: the scattering operator $S_w:L^2(X_\Theta)\to L^2(X_\Omega)$, i.e.\ the $w$-equivariant part of $\iota_\Omega^*\iota_\Theta f $ with respect to the action of $A_{X,\Theta}$, is given by:
$$ S_w f (x) = \int_{\widehat{X_\Theta^L}^\disc} \mathscr S_{w^{-1}}^* f^{\tilde\sigma} (x) d\sigma.$$

The following proves the assertion on $\mathfrak z^\disc(X_\Theta^L)$-equivariance of Theorem \ref{unitaryscattering}; assertion \eqref{assoc-fiberwise} of Theorem \ref{fiberwisescattering}; and the regularity statement of Theorem \ref{fiberwisescattering}. The regularity statement, i.e.\ the fact that the operators $\mathscr S_w$, and hence also $\mathscr S_\Subunit$ by Proposition \ref{regularEisenstein}, are actually regular on the unitary set, means that the condition ``where this operator is defined'' is superfluous in Proposition \ref{propsubunitary}.

\begin{proposition}\label{scatteringtool}
Let $\Theta\sim\Omega$, $w\in W_X(\Theta,\Omega)$.  For every $\sigma\in \widehat{X_\Theta^L}_\CC^\disc$ where the operator $\mathscr S_w$ is defined, the resulting morphism:
 $$ \mathcal S(X_\Omega)\to \mathscr L_{\Theta,\sigma}$$
 factors through the discrete ${^w\sigma}$-coinvariants $\mathcal S(X_\Omega)_{{^w\sigma},\disc} = \mathscr L_{\Omega,{^w\sigma}}$ and is generically an isomorphism between $\mathscr L_{\Omega,{^w\sigma}}$ and $\mathscr L_{\Theta,\sigma}$. In particular, $w$ induces an isomorphism: $\widehat{X_\Theta^L}^\disc\xrightarrow{\sim} \widehat{X_\Omega^L}^\disc$.
 
 Thus, $\mathscr S_w$ is a rational section of the sheaf $\Hom_G( w^*\mathscr L_\Omega, \mathscr L_\Theta )$ over $\widehat{X_\Theta^L}_\CC^\disc$. Its poles do not meet the unitary set, i.e.:
\begin{equation}\label{Sweq}\mathscr S_w \in \Gamma\left(\widehat{X_\Theta^L}^\disc,\Hom_G( w^*\mathscr L_\Omega, \mathscr L_\Theta )\right).
\end{equation}

 The operators $\mathscr S_w$ satisfy the natural associativity conditions: 
\begin{eqnarray*}\mathscr S_{w'} \circ \mathscr S_w &=&  \mathscr S_{w'w} \text{ for }  w\in W_X(\Omega,\Theta), w'\in W_X(\Xi,\Omega).
\end{eqnarray*}

The scattering map $S_w$ is $w$-equivariant with respect to the discrete center $\mathfrak z^\disc(X_\Theta^L)$.
\end{proposition}

For the proof of the proposition we will need the following lemma:

\begin{lemma}\label{lemmainduced} 
 Suppose that $\tau_1,\tau_2$ are non-isomorphic, irreducible representations of the Levi quotient of a parabolic $P$, and that $\varchi$ is a subtorus of the unramified characters of $P$ containing \emph{$P$-regular} characters (i.e.\ those which are non-trivial on the image of all coroots corresponding to roots in the unipotent radical of $P$).
 
 Then, for $\omega\in \varchi$ in general position, $I_P(\tau_1\otimes\omega)$ and $I_P(\tau_2\otimes\omega)$ are irreducible and non-isomorphic.
\end{lemma}

\begin{proof}
The irreducibility statement is \cite[Th\'eor\`eme 3.2]{Sauvageot} --- notice that irreducibility of the induced representation is a Zariski open condition \cite[VI.8.4, Proposition]{Renard}. 

Suppose $I_P(\tau_1\otimes\omega)\simeq I_P(\tau_2\otimes\omega)$ for $\omega$ in some Zariski dense subset $\varchi'$ of $\varchi$, fix such an isomorphism for each such $\omega$ and denote this representation by $\pi_\omega$. Without loss of generality, we may assume that the trivial character $\omega=1$ belongs to $\varchi'$. From the two realizations of $\pi_\omega$ we deduce that its (normalized) Jacquet module $(\pi_\omega)_P$  with respect to $P$ has an irreducible quotient which is isomorphic to $\tau_1\otimes\omega$, and an irreducible quotient which is isomorphic to $\tau_2\otimes\omega$. We will show that, if $\omega$ belongs to some fixed open dense subset of $\mathcal X$, these quotients have to coincide. 

To see that, let $M$ be the Levi quotient of $P$, and consider semisimplifications, to be denoted by $[\pi]$ (or, alternatively, elements in the Grothendieck group of admissible representations). We have $[\tau_1\otimes\omega]\subset [(\pi_\omega)_P]$ and $[\tau_2\otimes\omega]\subset[(\pi_\omega)_P]$ as representations of $M$.

Now let $(Q,\rho)$ be a pair consisting of a parabolic $Q\subset P$ and a supercuspidal representation $\rho$ of its Levi quotient $L$ such that $\pi_1$ is a subquotient of $I_Q(\rho)$.  
We will compute semisimplifications of Jacquet modules of $\pi_\omega$ with respect to $Q$. As in the proof of \cite[Th\'eor\`eme 3.2]{Sauvageot}, we have, from the two realizations of $\pi_\omega$:
\begin{equation}\label{JacquetQ1}
[(\pi_\omega)_Q] = \sum_{\underset{wMw^{-1}\supset L}{w\in W/W_M}} {^w[ (\tau_1\otimes\omega)_{M\cap w^{-1} Q w}]},
\end{equation}
and
\begin{equation}\label{JacquetQ2}
[(\pi_\omega)_Q] = \sum_{\underset{wMw^{-1}\supset L}{w\in W/W_M}} {^w[ (\tau_2\otimes\omega)_{M\cap w^{-1} Q w}]}.
\end{equation}
Here,  $W_M$ denotes the Weyl group of $M$.

Notice that the term of each of the above sums corresponding to the trivial coset $1 W_M$ is the semi-simplicifaction of the Jacquet module of $\tau_i\otimes\omega$ with respect to $M\cap Q$ ($i=1,2$). If the two irreducible quotients $\pi_\omega\to \tau_1\otimes\omega$ and $\pi_\omega\to \tau_2\otimes\omega$ do not coincide (i.e.\ do not have the same kernel), since $\tau_1$ and $\tau_2$ are irreducible, that means that all the irreducible summands of $[ (\tau_1\otimes\omega)_{M\cap Q}]$ should appear among the irreducible summands of the subsum of \eqref{JacquetQ2} with $w\ne 1 W_M$.

Let $\rho_1$ be an irreducible (necessarily cuspidal) summand of ${[ (\tau_1)_{M\cap Q}]}$ --- so $\rho_1\otimes\omega$ is an irreducible summand of $[{(\tau_1\otimes\omega)_{M\cap Q}}]$. If we assume it to be isomorphic, for each $\omega\in \varchi'$, to some irreducible (cuspidal) summand of the terms of \eqref{JacquetQ2} with $w\ne 1 W_M$, and since $\varchi'$ is Zariski dense, there is a $w\in W\smallsetminus W_M$, a Zariski dense subset $\varchi''\subset\varchi'$ and an irreducible summand $\rho_2$ of ${^w [(\tau_2)_{M\cap w^{-1} Q w}]}$ with $$\rho_1\otimes\omega \simeq \rho_2 \otimes {^w\omega}$$ for all $\omega\in \varchi''$.

Let $\omega_0\in \varchi''$, so that the central characters of $\rho_1\otimes\omega_0$ and $\rho_2\otimes{^w\omega_0}$ coincide. Thus, the restrictions of $\omega/\omega_0$ and ${^w\omega}/{^w\omega_0}$ to the center of $L$ have to coincide for all $\omega\in \varchi''$. Since $\varchi''$ is Zariski dense in $\varchi$, the restrictions of all $\omega$ and ${^w\omega}$ to the center of $L$ have to coincide, for all $\omega\in \varchi$. But, having assumed that $\varchi$ contains $P$-regular characters, this is only possible if $w\in W_M$, a contradiction. 
\end{proof}

Notice that, by the wavefront and strong factorizability assumption, $\widehat{X_\Theta^L}^\unr$ contains $P_\Theta$-regular (in fact: ``strictly $P$-dominant'') elements for every $\Theta\subset\Delta_X$ --- s.\ the proof of \cite[Corollary 15.3.3]{SV}. Therefore, the lemma applies to families of irreducible representations of $L_\Theta$ twisted by $\widehat{X_\Theta^L}^\unr$.

\begin{proof}[Proof of Proposition \ref{scatteringtool}]
By $w$-equivariance with respect to $A_{X,\Theta}'$, it follows that the specialization of $\mathscr S_w$ at $\sigma$ factors through the $A_{X,\Omega}'$-coinvariant space:
$$\mathcal S(X_\Omega)_{{^w\chi_\sigma}},$$
where $\chi_\sigma$ is the central character of $\sigma$. Moreover, since $S_w$ is a morphism: $L^2(X_\Theta)_\disc\to L^2(X_\Omega)_\disc$, it follows that $\mathscr S_w$ is zero on the kernel of the map:
$$\mathcal S(X_\Omega)_{{^w\chi_\sigma}} \to L^2(X_\Omega/A_{X,\Omega}', ^w\chi_\sigma)_\disc$$
(for almost all, and hence for all $\sigma$ where it is defined), and hence factors through the discrete coinvariants $\mathcal S(X_\Omega)_{{^w\chi_\sigma}, \disc}$.
By the fact that $S_w$ is an isometry, we get that $\mathscr S_w$ is non-zero on every connected component of $\widehat{X_\Theta^L}^\disc$.

By definition, the space $\mathcal S(X_\Omega)_{{^w\chi_\sigma}, \disc}$ is equal to:
$$ \bigoplus_\tau \mathscr L_{\Omega,\tau},$$
where $\tau$ ranges over the fiber of the map: $\widehat{X_\Omega^L}_\CC^\disc\to \widehat{A'_{X,\Omega}}_\CC$ (central character) over ${^w\chi_\sigma}$.

We claim that for $\sigma$ in general position, the only such $\tau$ with the property that $I_\Theta(\sigma)$ and $I_\Omega(\tau)$ have a common subquotient is $\tau={^w\sigma}$. To see this, let $\sigma$ vary in a family of the form $\sigma_0\otimes \chi$, with $\sigma_0\in \widehat{X_\Theta^L}^\disc$ and $\chi$ varying in $\widehat{X_\Theta^L}^\unr$. Thus, ${^w\sigma} = {^w\sigma_0}\otimes\omega$, with $\omega={^w\chi}$ varying in $\widehat{X_\Omega^L}^\unr$. By Lemma \ref{lemmainduced}, for $\omega$ in general position the representation $I_\Theta(\sigma)$ is irreducible, and hence the standard intertwining operator is an isomorphism:
$$I_\Theta(\sigma) \simeq I_\Omega({^w\sigma}).$$ 
Thus, any non-zero morphism:
$$ I_\Omega(\tau)\to I_\Theta(\sigma)$$
gives, by composition, a non-zero morphism: $I_\Omega(\tau)\to I_\Omega({^w\sigma})$. 

If $\{\tau_1,\dots,\tau_k\}$ is the fiber of $\widehat{X_\Omega^L}^\disc\to \widehat{A'_{X,\Omega}}$ over ${^w\chi_{\sigma_0}}$, then the fiber over ${^w\chi_{\sigma}}$, for $\sigma$ as above, is $\{\tau_1\otimes\omega,\dots,\tau_k\otimes\omega\}$. Again by Lemma \ref{lemmainduced}, for $\omega$ in general position, we cannot have a non-zero morphism $I_\Omega(\tau_i\otimes\omega)\to I_\Omega({^w\sigma_0}\otimes\omega)$, unless $\tau_i\otimes\omega\simeq {^w\sigma_0}\otimes\omega$, or equivalently $\tau_i \simeq {^w\sigma_0}$.

Thus, the specialization of $\mathscr S_w$ at $\sigma$ factors through $\mathscr L_{\Omega,{^w\sigma}}$. Using \eqref{unitarySw}, this proves that the scattering map $S_w$ is $w$-equivariant with respect to the discrete center $\mathfrak z^\disc(X_\Theta^L)$.
The fact that $S_w$ is an isometry now proves that the resulting map: $\mathscr L_{\Omega,{^w\sigma}}\to \mathscr L_{\Theta,\sigma}$ is an isomorphism for generic $\sigma$.

For the regularity statement, we will proceed as in the proof of Proposition \ref{regularEisenstein}, where a priori knowledge of the integrability of Eisenstein integrals gave us their regularity on the unitary spectrum. Here we will use the a priori knowledge (Theorem \ref{unitaryscattering}) that the scattering operators are bounded operators between $L^2$-spaces (in fact, isometries, but we will not use that):
$$ S_w: L^2(X_\Theta)_\disc\to L^2(X_\Omega)_\disc.$$

In terms of Theorem \ref{thmdiscrete}, this can be written as a map:
$$ L^2(\widehat{X_\Theta^L}^\disc,\mathscr L_\Theta)\to L^2(\widehat{X_\Omega^L}^\disc,\mathscr L_\Omega),$$
which, we now know, is induced by some element:
$$\mathscr S_w\in  \CC(\widehat{X_\Theta^L}^\disc, \Hom_G(\mathscr L_\Theta,w^*\mathscr L_\Omega)).$$

By Corollary \ref{decompEisenstein}, $\mathscr S_w$ has linear poles. Corollary \ref{linearpolesmap} now implies that it is regular on the unitary spectrum.

Finally, the associativity conditions on $\mathscr S_w$ follow from those of the unitary scattering maps $S_w$. Indeed, the only way that the composition of the following maps:
$$ L^2(\widehat{X_\Theta^L}^\disc,\mathscr L_\Theta)\to L^2(\widehat{X_\Omega^L}^\disc,\mathscr L_\Omega)\to L^2(\widehat{X_\Xi^L}^\disc,\mathscr L_\Xi),$$
given by fiberwise application of $\mathscr S_w$ and $\mathscr S_{w'}$, is equal to the fiberwise application of $\mathscr S_{w'w}$ is that $\mathscr S_{w'w}|_{\mathscr L_{\Theta,\sigma}} = \mathscr S_{w'}\circ \mathscr S_{w}|_{\mathscr L_{\Theta,\sigma}}$ for almost all $\sigma \in \widehat{X_\Theta^L}^\disc$, and thus for all.

\end{proof}

\begin{corollary}\label{donotappear}
When we decompose $E_{\Theta,\disc}^{*\Omega}$ as in \ref{decompEisensteineq}, the poles of the form \eqref{singular} with $w\in W_X(\Omega,\Theta)$ do not appear.
\end{corollary}

\begin{proof}
Indeed, by \eqref{nochoiceeq} these poles intersect the unitary set, where we have just proven that the summands are regular.
\end{proof}

Finally, we prove the continuous preservation of Harish-Chandra Schwartz spaces under the scattering maps, thus completing the proof of Theorem \ref{unitaryscattering}:

\begin{proof}[Proof that $S_w$, $w\in W_X(\Omega,\Theta)$, restricts to a continuous map: $\mathscr C(X_\Theta)_\disc\to \mathscr C(X_\Omega)_\disc$]
 By \eqref{unitarySw} the following diagram commutes:
 $$ \xymatrix{
 L^2(X_\Theta)_\disc \ar[r]^{S_w}\ar[d] &  L^2(X_\Omega)_\disc\ar[d]\\
 L^2(\widehat{X_\Theta^L}^\disc, \mathscr L_\Theta) \ar[r]^{\mathscr S_w} & L^2(\widehat{X_\Omega^L}^\disc, \mathscr L_\Omega),}$$
 where the vertical arrows are the isomorphisms of the Plancherel formula \eqref{L2discrete}.
 
 By the regularity statement of Proposition \ref{scatteringtool}, the restriction of the bottom arrow to smooth sections gives an isomorphism:
 $$ \xymatrix{C^\infty(\widehat{X_\Theta^L}^\disc, \mathscr L_\Theta) \ar[r]^{\mathscr S_w} & C^\infty(\widehat{X_\Omega^L}^\disc, \mathscr L_\Omega),}$$
 which by Theorem \ref{thmdiscrete} corresponds to an isomorphism between discrete summands of the corresponding Harish-Chandra Schwartz spaces.
 
\end{proof}

\section{Scattering: the smooth case} \label{sec:smooth}

We now turn to the smooth case, in order to prove Theorem \ref{smoothscattering} and the remainder of Theorem \ref{fiberwisescattering}. As in the unitary case, the smooth scattering maps $\textswab S_w$ will be given by integrating the fiberwise scattering maps $\mathscr S_w$, but now over a shift of the unitary set in analogy to Theorem \ref{explicitsmooth}. However, there is an important result that needs to be proven first: that ``cuspidal scatters to cuspidal''. This is the analog of ``discrete scatters to discrete'' in the unitary case, which was proven in the course of the development of the Plancherel theorem, by an analytic argument. Similarly, here, ``cuspidal scatters to cuspidal'' will be proven using a priori knowledge about smooth asymptotics, and more precisely the support theorem \ref{support}.

We notice that both the statement ``discrete scatters to discrete'' and ``cuspidal scatters to cuspidal'' have been proven by Carmona-Delorme \cite{CD} in the symmetric case. The proofs there heavily use the structure of symmetric varieties. Here we present a different argument which applies in greater generality.

Recall again the asymptotics of normalized constant terms, defined in section \ref{sec:eigenspace}:
$$E_{\Theta,\disc}^{*,\Omega} = E_{\Theta,\disc}^*\circ e_\Omega:\mathcal S(X_\Omega)\to \CC(\widehat{X_\Theta^L}^\disc, \mathscr L_\Theta).$$

We may project those to the cuspidal quotient (and summand) $\mathcal L_\Theta$ of $\mathscr L_\Theta$, in which case we will denote them by:
$$E_{\Theta,\cusp}^{*,\Omega}:\mathcal S(X_\Omega)\to \CC(\widehat{X_\Theta^L}^\cusp, \mathcal L_\Theta).$$

\begin{proposition}\label{fibercuspidal}
If $\Omega$ does not contain a conjugate of $\Theta$, then $E_{\Theta,\cusp}^{*,\Omega}$ is zero.

If $\Omega\sim \Theta$ then $E_{\Theta,\cusp}^{*,\Omega}$ factors through $\mathcal S(X_\Omega)_\cusp$:
$$ E_{\Theta,\cusp}^{*,\Omega} \in \Hom_G \left(\mathcal S(X_\Omega)_\cusp, \CC(\widehat{X_\Theta^L}^\cusp, \mathcal L_\Theta)\right).$$
The summands $\mathscr S_w$ of \eqref{decompEisensteineq}, viewed as in \eqref{Sweq}, restrict to elements of: 
\begin{equation}\label{Sweq-cusp} \Gamma\left(\widehat{X_\Theta^L}^\cusp,\Hom_G( w^*\mathcal L_\Omega, \mathcal L_\Theta )\right)\end{equation}
(i.e.\ preserve the cuspidal summands of the bundles $\mathscr L_\bullet$), and the projection of $\mathscr S_\Subunit$ of \eqref{decompEisensteineq} to $\mathcal L_\Theta$ is zero, hence we have a decomposition:
\begin{equation}\label{cuspdecomp}
   E_{\Theta,\cusp}^{*,\Omega} = \sum_{w\in W_X(\Theta,\Omega)} \mathscr S_w|_{\mathcal L},
\end{equation}
where $~|_{\mathcal L}$ denotes the restriction of $\mathscr S_w$ to the subbundle $\mathcal L_\bullet$.
\end{proposition}

We will prove this proposition below; let us first see how it implies  Theorems \ref{smoothscattering} and \ref{fiberwisescattering}.

First of all, the decomposition \eqref{cuspdecomp}, combined with Theorem \ref{explicitsmooth}, allows us to express the composition $e_\Omega^*e_\Theta$, restricted to $\mathcal S(X_\Theta)_\cusp$ as a sum:

$$ e_\Omega^*e_\Theta|_{\mathcal S(X_\Theta)_\cusp} = \sum_{w\in W_X(\Omega,\Theta)} \textswab S_w,$$
as claimed in Theorem \ref{smoothscattering}, where $\textswab S_w$ is defined as in \eqref{smoothSw}. Notice that \eqref{smoothSw} is independent of $\omega$ as long as $\omega\gg 0$; this follows from Corollary \ref{decompEisensteineq}, according to which the $\mathscr S_w$ are rational with linear poles (hence no poles for $\omega\gg 0$). Moreover, by Proposition \ref{fibercuspidal}, 
$$\mathscr S_w(\mathcal S(X_\Theta)_\cusp)\subset C^\infty(X_\Omega)_\cusp.$$

Now define, as in Theorem \ref{smoothscattering} the space $\mathcal S^+(X_\Theta)_\cusp\subset C^\infty(X_\Theta)_\cusp$ as the span of all spaces:
$$\textswab S_w \mathcal S(X_\Omega)_\cusp$$
with $\Omega \sim \Theta$ and $w\in W_X(\Theta,\Omega)$. This includes the case $w=1$ where $\textswab S_w=\Id$, so $\mathcal S^+(X_\Theta)_\cusp \supset \mathcal S(X_\Theta)_\cusp$. Let us prove that the scattering maps (and, incidentally, the Bernstein maps) extend to the spaces $\mathcal S^+(X_\Theta)_\cusp$, and that they satisfy the associativity properties asserted in Theorem \ref{smoothscattering}:

\begin{proposition}\label{extension}
For $f \in \mathcal S^+(X_\Theta)_\cusp$ and $w\in W_X(\Omega,\Theta)$, define $\textswab S_w f\in \mathcal S^+(X_\Omega)_\cusp$ and $e_\Theta f\in C^\infty(X)$ as follows:
If 
$$f = \sum_{(\Theta',w')} \textswab S_{w'} f_{(\Theta',w')} $$
with $w'\in W_X(\Theta,\Theta')$ and $f_{(\Theta',w')} \in \mathcal S(X_{\Theta'})_\cusp$, we set: 
$$ \textswab S_w f =  \sum_{(\Theta',w')} \textswab S_{ww'} f_{(\Theta',w')}.$$
and:
$$ e_\Theta f = \sum_{(\Theta',w')} e_{\Theta'} f_{(\Theta',w')}.$$

Then the resulting maps are well-defined 
(do not depend on the decomposition of $f$ chosen),  and $\textswab S_w$  is an isomorphism:
$$\textswab S_w: \mathcal S^+(X_\Theta)_\cusp \to \mathcal S^+(X_\Omega)_\cusp.$$

Moreover, $\textswab S_w$ is $w$-equivariant with respect to the actions of $\mathfrak z^\cusp(X_\Theta^L)$ on $C^\infty(X_\Theta)_\cusp$, $C^\infty(X_\Omega)_\cusp$, and the maps $\textswab S_w$ satisfy the associativity properties of Theorem \ref{smoothscattering}. 
\end{proposition}

\begin{proof}
First of all, \eqref{smoothSw} implies that every element of $\mathcal S^+(X_\Theta)_\cusp$ is of moderate growth, as was the case for elements of $e_\Theta^*(\mathcal S(X))$, cf.\ Proposition \ref{moderategrowth}. Hence, every element of $\mathcal S^+(X_\Theta)_\cusp$ admits a \emph{unique} spectral decomposition of the form \eqref{offaxis}, with the only difference from \eqref{offaxis} being that the forms $f^{\tilde\sigma}d\sigma$ are not polynomial, but rational with linear poles, given by \eqref{smoothSw}.
We point the reader to \cite[\S 15.4.4]{SV} for details on the spectral decomposition of functions of moderate growth. 

If $f = \sum_{(\Theta',w')} \textswab S_{w'} f_{(\Theta',w')} $ as in the statement of the proposition then, using \eqref{smoothSw} and the associativity property \eqref{assoc-fiberwise} of the fiberwise scattering maps $\mathscr S_w$ we conclude that the operator $\textswab S_w$ described in the proposition also admits the expression \eqref{smoothSw}, which proves that it is well-defined. 

Similarly for $e_\Theta$: by the commutativity of \eqref{Bmatrices}, it is expressed explicitly by applying the formula of Theorem \ref{explicitsmooth} to the spectral decomposition \eqref{offaxis} of $f$.

Moreover, \eqref{assoc-smooth} now follows from \eqref{assoc-fiberwise}, and the fact that $\textswab S_1 = \Id$ shows that these maps are isomorphisms. The extension of the action of the cuspidal center with the given properties is obvious.

The associativity relations of the operators $\textswab S_w$ follow from those of the operators $\mathscr S_w$, which were proven in the previous section.
\end{proof}

This completes the proof of Theorem \ref{fiberwisescattering}, and of Theorem \ref{smoothscattering} for the case $\Omega\sim\Theta$. 

If $\Omega$ does not contain a conjugate of $\Theta$ then the same calculation and Proposition \ref{fibercuspidal} show that the projection of $e_\Theta^* e_\Omega$ to $C^\infty(X_\Theta)_\cusp$ is zero or, equivalently, $e_\Omega^* e_\Theta$, when restricted to $\mathcal S(X_\Theta)_\cusp$, is zero.

Finally, if $\Omega$ contains, but is not equal to, a conjugate of $\Theta$ then by switching the roles of $\Theta$ and $\Omega$ in the above argument, since $\Theta$ does not contain a conjugate of $\Omega$ we have: 
$$ e_{\Theta,\cusp}^* e_\Omega |_{\mathcal S(X_\Omega)_\cusp} = 0,$$
which means that $e_\Omega^* e_\Theta$, when restricted to $\mathcal S(X_\Theta)_\cusp$, has image in $C^\infty(X_\Omega)_\noncusp$.

This completes the proof of Theorem \ref{smoothscattering}, assuming Proposition \ref{fibercuspidal}.

Now let us come to the proof of Proposition \ref{fibercuspidal}.

\begin{proof}[Proof of Proposition \ref{fibercuspidal}]

The proof is based on the same result as Theorem \ref{explicitsmooth}, namely Proposition \ref{support} on the support of elements of $e_\Theta^*(\mathcal S(X))$. This proposition implies, in particular, 
that for every $f\in \mathcal S(X_\Theta)_\cusp$, the support of $e_\Omega^* e_\Theta f$ has compact closure in an affine embedding $X_\Omega^a$ of $X_\Omega$. Moreover, Proposition \ref{moderategrowth} states that this function is of moderate growth.

By Theorem \ref{explicitsmooth}, 
\begin{equation}\label{eOmega} e_\Omega^* e_\Theta f (x) = e_\Omega^* \int_{\omega^{-1}\widehat{X_\Theta^L}^\cusp}  E_{\Theta,\sigma,\cusp} f^{\tilde\sigma} (x) d\sigma = 
\int_{\omega^{-1}\widehat{X_\Theta^L}^\cusp} E^{\Omega}_{\Theta,\sigma,\cusp} f^{\tilde\sigma} (x) d\sigma,
\end{equation}
the second equality because $e_\Omega^*$ commutes with the integral (because it is equivariant and commutes after evaluating ``close to infinity'' --- cf.\ the proof of Proposition \ref{characterization}).

We first claim that if $|\Omega|<|\Theta|$, i.e.\ $\dim A_{X,\Omega}>\dim A_{X,\Theta}$, then $e_\Omega^* e_\Theta f$ has to be zero. Since we may translate $f$ by the action of $G$, it is enough to fix an $A_{X,\Omega}$-orbit $Z$ and show that 
$$\left.e_\Omega^* e_\Theta f\right|_Z\equiv 0.$$
We identify $Z$ with $A_{X,\Omega}$ by fixing a base point. Let $X_\Omega^a$ be an affine embedding of $X_\Omega$ as above, and let $\psi$ be an algebraic character of $A_{X,\Omega}$ which extends to the closure $\bar Z$ of $Z$ in $X_\Omega^a$ by zero. We remind that an affine embedding of a torus $A_{X,\Omega}$ is described by the set of characters of $A_{X,\Omega}$ which vanish on the complement of the open orbit, and this set (monoid) of characters has to generate the character group; in particular, such a character $\psi$ exists. The function $e_\Omega^* e_\Theta f$ is of moderate growth; since $\bar Z\smallsetminus Z$ is a divisor, this is equivalent to saying that there is an open cover $\bar Z=\cup_i U_i$ and for every $i$ a function $F_i$ which is \emph{regular} on $U_i\cap Z$ such that $|e_\Omega^* e_\Theta f| \le |F_i| $ on $U_i\cap Z$. Multiplied by a high enough power of $\psi$, $F_i$ becomes regular on the whole $U_i$. The support of $\left.e_\Omega^* e_\Theta f\right|_Z$ has compact closure in $\bar Z$, and the Haar measure on $Z\simeq A_{X,\Omega}$, after multiplied by a high enough power of $\psi$, also extends to a finite measure on the closure of the support of $\left.e_\Omega^* e_\Theta f\right|_Z$.  Therefore, for a large enough $n$, the function $\left.|\psi|^n\cdot e_\Omega^* e_\Theta f\right|_Z$ belongs to $L^2(Z) = L^2(A_{X,\Omega})$, and its abelian Fourier/Mellin transform is in $L^2(\widehat A_{X,\Omega})$.

On the other hand, let us revisit Proposition \ref{finite}, fixing again a connected component $Y$ of $\widehat{X_\Theta^L}^\cusp$: As $\omega$ varies in $\widehat {X_\Theta^L}_\CC^\unr$, each factor of \eqref{minimal} varies over a set of characters of $A_{X,\Omega}$ of \emph{positive codimension} in $\widehat{A_{X,\Omega}}_\CC$. More precisely, the support of the $\CC[\widehat{A_{X,\Omega}}]$-module generated by $E_{\Theta,\cusp}^{*,\Omega}$ (restricted to the connected component $Y$) is contained in a subscheme $S$ of $\widehat{A_{X,\Omega}}_\CC$ whose reduction is the union: 
$$\bigcup_{w\in W_1} \left.{^w(\chi\widehat{A_{X,\Theta,\CC}})}\right|_{A_{X,\Omega}} \subset \widehat{A_{X,\Omega}}_\CC,$$
in the notation of \eqref{minimal}. 

Hence, on one hand we have that the $\CC[\widehat{A_{X,\Omega}}]$-module generated by $E_{\Theta,\cusp}^{*,\Omega}$ is supported on a subscheme of $\widehat{A_{X,\Omega}}_\CC$ which does not contain any connected component, and, on the other hand, the restriction of any $e_\Omega^* e_\Theta f$ to any $A_{X,\Omega}$-orbit (which decomposes as in \eqref{eOmega}), when multiplied by a high enough power $|\psi|^n$ of the absolute value of a character $\psi$, belongs to $L^2(A_{X,\Omega})$. Elementary Fourier analysis will now prove that this function is zero. 

Indeed, the $\CC[\widehat{A_{X,\Omega}}]$-action on $E_{\Theta,\cusp}^{*,\Omega}$ corresponds to an action of the completed Hecke algebra
$$ \hat{\mathcal H}(A_{X,\Omega}) := \lim_{\underset{J}\leftarrow} \mathcal H(A_{X,\Omega}, J),$$ 
with $J$ ranging over a basis of compact open subgroups. More specifically, given a polynomial $P\in \CC[\widehat{A_{X,\Omega}}]$, there is a compatible system of measures $(h_{P,J})_J\in (\mathcal H(A_{X,\Omega}, J))_J$ such for $\chi \in \widehat{A_{X,\Omega}}$ a $J$-invariant character, the Mellin transform satisfies
$$\widecheck{h_{P,J}}(\chi) := \int_{A_{X,\Omega}} h_{P,J}(a) \chi^{-1}(a) = P(\chi).$$

Since the dimension of the support subscheme $S\subset \widehat{A_{X,\Omega}}_\CC$ is smaller than that of $\widehat{A_{X,\Omega}}_\CC$, there is $P\in \CC[\widehat{A_{X,\Omega}}]$, non-zero on every connected component, which vanishes on $S$  (i.e.\ vanishes with the appropriate multiplicity, since $S$ is not necessarily reduced).  
Explicitly, if for every factor of \eqref{minimal} of the form $(z-{^w(\chi\tilde{\textswab t})} (z))$ we choose an element $z_{\chi,w}\in A_{X,\Omega}$ on which ${^w(\chi\omega)} (z_{\chi,w})$ is equal to some constant $a_{\chi,w}$ for \emph{all} $\omega\in \widehat{X_\Theta^L}^\unr$ (such an element exists for dimension reasons), then the restriction of $P$ to a connected component of $\widehat{A_{X,\Omega}}$ can be taken to be the product, over all pairs $(\chi,w)$ such that ${^w\chi}$ belongs to that component, of the terms $(z_{\chi,w} - a_{\chi,w})$ (where $z_{\chi,w}$ is by evaluation a polynomial on this component and $a_{\chi,w}$, we repeat, is a constant complex number). For every open compact subgroup $J$, the measure $h_{P,J}$ annihilates the restriction of $E_{\Theta,\cusp}^{*,\Omega}$ to the chosen component $Y\subset \widehat{X_\Theta^L}^\cusp$.

Now take $f\in \mathcal S(X_\Theta)_\cusp$, spectrally supported on the chosen component $Y\subset \widehat{X_\Theta^L}^\cusp$, and an open compact subgroup $J\subset A_{X,\Omega}$ such that $e_\Omega^* e_\Theta f$ is $J$-invariant. By \eqref{eOmega}, we will have $h_{P, J}\star e_\Omega^* e_\Theta f = 0$, and therefore $$(|\psi|^n h_{P,J}) \star (|\psi|^n e_\Omega^* e_\Theta f) = |\psi|^n \cdot (h_{P,J} \star e_\Omega^* e_\Theta f) = 0.$$

But the function $|\psi|^n e_\Omega^* e_\Theta f$, restricted to any $A_{X,\Omega}$-orbit, belongs to $L^2(A_{X,\Omega})$ (where we have chosen a base point to identify this restriction with a function $\Phi$ on $A_{X,\Omega}$). Thus, its Mellin transform $\check \Phi(\chi) = \int_{A_{X,\Omega}} \Phi(a\chi) \chi^{-1}(a) da$ is in $L^2(\widehat{A_{X,\Omega}})$. On the other hand, the Mellin transform of $\Phi':= (|\psi|^n h_{P,J}) \star \Phi$ is equal to $\check \Phi' (\chi) = P(\chi |\psi|^{-n}) \check\Phi(\chi |\psi|^{-n})$. Since $\Phi'=0$, and $P$ is non-vanishing outside of a set of measure zero in $\widehat{A_{X,\Omega}}$, it follows that $\check\Phi = 0$, and hence $\Phi=0$. 

This proves that $e_\Omega^* e_\Theta f \equiv 0$ when $|\Omega|<|\Theta|$.

Now assume that $\Omega$ does not contain a conjugate of $\Theta$. By induction on $|\Omega|$, we may assume that $e_{\Omega'}^* e_\Theta f = 0$ for every $\Omega'\subsetneq \Omega$ and hence $e_\Omega^* e_\Theta f \in C^\infty(X_\Omega)_\cusp$, for all $f\in \mathcal S(X_\Theta)_\cusp$. That means that $E_{\Theta,\cusp}^{*,\Omega}$ factors through $\mathcal S(X_\Omega)_\cusp$, so by evaluating at points of regularity we get a family $\mathcal I$ of morphisms:
$$ \mathcal S(X_\Omega)_\cusp \to \mathcal L_{\Theta,\sigma}.$$
In the language of \S \ref{ssweaktangent}, we will examine the weak tangent space of this family which, we recall, has to do with the set of morphisms obtained by second adjointness:
\begin{equation}\label{red} \mathcal S(X_\Omega^L)_\cusp \to (\mathcal L_{\Theta,\sigma})_\Omega.
\end{equation}

By Corollary \ref{finite2} we have that $WT_\Omega(\mathcal I)$, if nonempty, is a union of sets of the form: 
$$[{^w\left(\mathfrak a_{X,\Theta,\CC}^*\right)}],$$
where $[\,\bullet \,]$ denotes image in $\mathfrak a^*/W_{L_\Omega}$. On the other hand, we can twist the morphisms \eqref{red} by elements of $\widehat{X_\Omega^L}^\unr_\CC$, thus obtaining a possibly larger family $\mathcal J$ whose weak tangent space will be a union of components of the form:
$$[{^w\left(\mathfrak a_{X,\Theta,\CC}^*\right)}]+ \mathfrak a^*_{X,\Omega,\CC}.$$

By the first statement of Lemma \ref{GIcorollary} (with $\Theta$ and $\Omega$ interchanged), since $\Omega$ does not contain a conjugate of $\Theta$, the dimension of this is \emph{strictly larger} than the dimension of $\mathfrak a^*_{X,\Omega,\CC}$. This is a contradiction: the supercuspidal supports of all finite-length quotients of $\mathcal S(X_\Omega^L)_\cusp $ belong to the equivalence classes of a \emph{countable} union of families of the form $(\tau\otimes\omega,L)$, with $(\tau,L)$ a supercuspidal pair in $L_\Omega$ and $\omega$ varying in  $\widehat{X_\Omega^L}_\CC^\unr$; hence, we cannot have a family of finite-length quotients of $\mathcal S(X_\Omega^L)_\cusp $ whose weak tangent space has dimension larger than that of $\widehat{X_\Omega^L}_\CC^\unr$, i.e.\ larger than that of $A_{X,\Omega}$. This shows that $E_{\Theta,\cusp}^{*,\Omega}$ is zero.

Finally, consider the case $\Omega\sim\Theta$. First of all:

\begin{lemma}\label{cuspidalsummands}
Let $Y\subset \widehat{X_\Theta^L}^\cusp$ be a connected component, and $E\in \mathcal M_Y^\cusp:=\Hom_G(\mathcal S(X_\Omega), \CC(Y,\mathcal L_\Theta))$. Let $\mathcal K_Y = \CC(Y)$, and let $E= \sum_i E_i$ be the decomposition of $E$ into elements of (distinct) generalized eigenspaces\footnote{``Generalized eigenspaces'' is used here in the generality of non-algebraically closed fields, i.e.\ a generalized eigenspace does not necessarily correspond to an eigenvalue, but to an irreducible monic polynomial.} for the action of $A_{X,\Omega}$ on the finite-dimensional $\mathcal K_Y$-vector space $\mathcal M_Y^\cusp$.
If $E$ factors through $\mathcal S(X_\Omega)_\cusp$, so does each of the $E_i$'s.
\end{lemma}

The validity of the lemma is obvious, since $\Hom_G(\mathcal S(X_\Omega)_\cusp, \CC(Y,\mathcal L_\Theta))$ is an $A_{X,\Omega}$-stable subspace of $\mathcal M_Y^\cusp$. 

Because of the lemma, 
 the projections to $\mathcal L_\Theta$ of all summands $\mathscr S_\bullet$ of \eqref{decompEisensteineq} all factor  through $\mathcal S(X_\Omega)_\cusp$, in other words by \eqref{Sweq} they restrict to elements of \eqref{Sweq-cusp}.

Finally, we claim that the projection of $\mathscr S_\Subunit$ to $\mathcal L_\Theta$ is zero. The argument here is identical to the one one we used for the case that $\Omega$ does not contain a conjugate of $\Theta$, using the weak tangent space of the corresponding family of maps: 
$$\mathcal S(X_\Omega^L)_\cusp \to (\mathcal L_{\Theta,\sigma})_\Omega$$ which arises from $\mathscr S_\Subunit$.
 If this family were non-zero, based on Lemma \ref{GIcorollary} it would give rise again to a family of morphisms of the form \eqref{red} whose weak tangent space has dimension larger than that of $\mathfrak a^*_{X,\Omega,\CC}$, a contradiction.

\end{proof}

\begin{remark}
 Regarding the last step of the proof: in the discrete case, there is no contradiction to the existence of subunitary exponents. The reason is that the ``subunitary parts'' of Eisenstein integrals do not need to be ``discrete modulo center'' (while the cuspidal parts were necessarily cuspidal modulo center by Lemma \ref{cuspidalsummands}). Indeed, they could be non-discrete, but with a central character that makes them decay ``towards infinity''.
\end{remark}

\part{Paley--Wiener theorems}

\section{The Harish-Chandra Schwartz space}

We start by proving the following two results:

\begin{proposition}\label{Theta to X}
 $\iota_\Theta$ takes $\mathscr C(X_\Theta)_\disc$ continuously into $\mathscr C(X)$.
\end{proposition}

And in the other direction:
\begin{proposition}\label{X to Theta}
 $\iota_{\Theta,\disc}^*$  takes $\mathscr C(X)$ continuously into $\mathscr C(X_\Theta)_\disc$.
\end{proposition}

We first reduce both statements to the case $\mathcal Z(X)=1$. This is achieved by using \eqref{Cfactorizable} and \eqref{Lfactorizable}, which identify $\mathscr C(X)$ and $L^2(X)$ as closed subspaces of a direct sum of spaces of the form:
$$ \mathscr C(\mathcal Z(X)\times Y) \simeq \mathscr C(\mathcal Z(X))\hat\otimes\mathscr C(Y),$$
respectively:
$$ L^2(\mathcal Z(X)\times Y) \simeq L^2(\mathcal Z(X))\hat\otimes L^2(Y),$$
where $Y$ is a spherical $[G,G]$-variety with $\mathcal Z(Y)=1$.

It is obvious from the definitions that the Bernstein maps on those spaces are induced by Bernstein maps for the second factor:
$$\iota_\Theta: L^2(Y_\Theta)\to L^2(Y),$$
which reduces both problems to the case $\mathcal Z(X)=1$. We will assume this for the two proofs.

The proof of Proposition \ref{Theta to X} will require a lemma: 
Fix an open compact subgroup $J\subset G$ and a collection $(N_\Theta)_\Theta$ of $J$-good neighborhoods of infinity. We may, and will, assume that this collection is determined by the neighborhoods $ N_{\hat \alpha}$, where $\alpha$ runs over all simple spherical roots and $\hat\alpha:= \Delta_X\smallsetminus\{\alpha\}$, in the following sense:
\begin{equation}\label{nbhdcompat}
N_\Theta = \bigcap_{\alpha\notin\Theta} N_{\hat \alpha}.
\end{equation}
We will also be denoting:
$$ N'_\Theta := N_{\Theta} \smallsetminus \bigcup_{\Omega\subsetneq\Theta} N_{\Omega},$$
remembering that the image of $N'_\Theta$ in $X_\Theta/A_{X,\Theta}$ is compact.

By a \emph{decaying function} on $X$ we will mean a positive, smooth function whose restriction to \emph{each} $N_\Theta$ is $A_{X,\Theta}$-finite function with \emph{subunitary} exponents. Notice that, by our definition, a subunitary exponent on $A_{X,\Theta}$ is allowed to be unitary on a ``wall'' of $A_{X,\Theta}^+$ (it only has to be $<1$ on $\mathring A_{X,\Theta}^+$). However, by demanding that the exponents of our function are $A_{X,\Theta}^+$-subunitary on \emph{every} $N_\Theta$, this possibility is ruled out: no exponent can be unitary on a wall of $A_{X,\Theta}^+$. We will call such exponents \emph{strictly subunitary}. Together with our assumption that $\mathcal Z(X)=1$ a decaying function is automatically in $\mathscr C(N_\Theta)$.

This definition is essentially compatible with the way the notion of ``decaying function on $A_{X,\Theta}^+$'' that was introduced in \cite{SV} (and will be used in the proof below): Indeed, for each $\Theta\subset\Delta_X$, a decaying function on $A_{X,\Theta}^+$, according to \cite{SV}, is any function bounded by the restriction of a positive, $A_{X,\Theta}$-finite function with strictly subunitary $A_{X,\Theta}^+$-exponents. Hence, a decaying ($J$-invariant) function $f$ on $X$, in our present sense, is precisely a $J$-invariant function with the property that for every $\Theta\subset \Delta_X$ and every $x\in N'_\Theta$ (or equivalently: $x\in N_\Theta$) the function:
\begin{equation}\label{deconA} A_{X,\Theta}^+ \ni a\mapsto |\left< f, a^{-1}\cdot 1_{xJ}\right>|,\end{equation}
where $1_{xJ}$ denotes the characteristic function of $xJ$, is bounded by a decaying function on $ A_{X,\Theta}^+$. (Notice that this is stronger than saying that the restriction of the function to the $A_{X,\Theta}^+$-orbit is a decaying function; our definition of the action of $A_{X,\Theta}^+$ is normalized by the square root of the volume, s.\ \S \ref{sseigenmeasures}, so the above bound is equivalent to a bound of $f(ax)$  by $\Vol(axJ)^{-\frac{1}{2}} $ times a decaying function on $A_{X,\Theta}^+$.)

Let $\tilde\tau_\Theta$ denote, for each $\Theta$, the map of restriction to $N_\Theta$.

\begin{lemma}\label{decaying} 
When $\mathcal Z(X)=1$, for any $\Phi\in L^2(X)^J$ the alternating sum:
 \begin{equation}\label{alternating}
\operatorname{Alt}(\Phi):= \sum_{\Theta\subset\Delta_X} (-1)^{|\Theta|} \tilde\tau_\Theta \iota_\Theta^* \Phi
 \end{equation}
 is bounded in absolute value by $\Vert \Phi\Vert_{L^2(X)}$ times a decaying function which depends only on $J$.
\end{lemma}

\begin{proof}

We claim that for every $\alpha\in \Delta_X$, and every $x\in N_{\hat\alpha}$, the restriction of \eqref{alternating} to $A_{X,\hat\alpha}^+ \cdot x$ satisfies the bound:
\begin{equation}\label{bd} |\operatorname{Alt}(\Phi)(a\cdot x)| \le \Vert\Phi\Vert_{L^2(X)} \Vol(axJ)^{-\frac{1}{2}} R_\alpha(a),\end{equation}
where $R_\alpha$ is a decaying function on $A_{X,\hat\alpha}^+ $ that only depends on $J$. 

This will prove the Lemma: Indeed, for an arbitrary $\Theta$, $x\in N'_\Theta$ and $a\in A_{X,\Theta}^+$ we get a bound:
$$|\operatorname{Alt}(\Phi)(a\cdot x)| = |\operatorname{Alt}(\Phi)(b\cdot a'x)|\le  \Vert\Phi\Vert_{L^2(X)}   \Vol(axJ)^{-\frac{1}{2}}  R_\Theta(a), $$
for every $\alpha \in \Delta_X\smallsetminus \Theta$ and decomposition $a = a' b$ with $a'\in A_{X,\Theta}^+$ and $b \in A_{X,\hat\alpha}^+$, 
where $R_\Theta$ is the decaying function on $A_{X,\Theta}^+$ defined as:
$$R_\Theta(a)= \min_{\alpha\in\Delta_X\smallsetminus \Theta}  \left(\min_{a = a' b,  a'\in A_{X,\Theta}^+, b \in A_{X,\hat\alpha}^+} R_\alpha(b)\right).$$

(Notice that $A_{X,\Theta}$ is generated by the one-dimensional tori $A_{X,\hat\alpha}$, $\alpha\in \Delta_X\smallsetminus \Theta$, up to finite index. Checking that $R_\Theta$, as defined above, is a decaying function essentially reduces the problem to the monoid $\mathbb N^{\Delta_X\smallsetminus \Theta}$, decaying functions $R_{\hat\alpha}$ for each of the coordinates, which can also be assumed to be equal to the same function $R$, and $R_\Theta((n_\alpha)_{\alpha\in \Delta_X\smallsetminus \Theta}) = R(\max_{\alpha} n_\alpha)$.) 

To prove \eqref{bd}, we notice that $\operatorname{Alt}(\Phi)$ is equal to the sum over all $\Theta$ containing $\alpha$ of the terms:
$$(-1)^{|\Theta|} (\tilde\tau_\Theta \iota_\Theta^* \Phi - \tilde\tau_{\Theta\smallsetminus\{\alpha\}} \iota_{\Theta\smallsetminus\{\alpha\}}^* \Phi).$$

By the transitivity property of Bernstein maps:
$$ \iota_\Theta^* \Phi  = \iota_\Theta^{\Theta', *} \circ \iota_{\Theta'}^* \Phi \mbox{ for } \Theta\subset \Theta',$$
where $\iota_\Theta^{\Theta', *} $ is the corresponding adjoint Bernstein map for the variety $X_{\Theta'}$,
and by the fact that the norm of $\iota^*_{\Theta'}  \Phi$ for some $\Theta'\subset\Delta_X$ with $\alpha\in \Theta'$, is bounded by a fixed multiple of $\Vert\Phi\Vert_{L^2(X)}$, it is enough to prove the statement when \eqref{alternating} is replaced by $\Phi - \iota_{\hat \alpha}^*\Phi$ (the rest of the terms being similar, with $\Phi$ replaced by $\iota_\Theta^*\Phi$, whose norm is bounded by a constant times $\Vert \Phi\Vert_{L^2(X)}$). Thus, we need to prove that for every $x\in N_{\hat\alpha}$ the restriction of $\Phi - \iota_{\hat \alpha}^*\Phi$ to $A_{X,\hat\alpha}^+ \cdot x$ is bounded by $\Vert \Phi \Vert_{L^2(X)} \Vol(axJ)^{-\frac{1}{2}} \cdot R_\alpha$, where $R_\alpha$ is a decaying function on $A_{X,\hat\alpha}^+ $ that only depends on $J$.  

This follows from \cite[Lemma 11.5.1]{SV} (and its proof): Indeed, if $\Psi = 1_{xJ}$, the characteristic function of some $J$-orbit on $ N_{\hat\alpha}$, and $a\in A_{X,\hat\alpha}^+$, then we have:
 $$ \left|\left< \Phi - \iota_{\hat\alpha}^* \Phi, a^{-1}\cdot \Psi\right> \right| = \left|\left< \Phi, (e_{\hat\alpha}-\iota_{\hat\alpha})a^{-1}\cdot \Psi\right> \right| \le$$ 
 $$ \le \Vert  \Phi \Vert_{L^2(X)} \cdot \Vert  (e_{\hat\alpha}-\iota_{\hat\alpha})a^{-1}\cdot \Psi \Vert_{L^2(X)} \le \Vert  \Phi\Vert_{L^2(X)} C_\Psi Q^J(a)$$
in the notation of \emph{loc.cit.}  so we can set $R_\alpha(a)=  C_\Psi Q^J(a)$, where $Q^J(a)$ is a decaying function on $A_{X,\hat\alpha}^+$ which is independent of $\Phi,\Psi$. In the proof of \cite[Lemma 11.5.1]{SV} it is seen, actually, that the constant $C_\Psi$ can be bounded by a fixed multiple of:
$$\Vert \Psi\Vert_{L^2(X_{\hat\alpha})} + \sum_i \Vert e_{\hat\alpha} a_i \cdot \Psi\Vert_{L^2(X)},$$
where the $a_i$'s range in a fixed finite set of elements of $A_{X,\hat\alpha}^+$. When $\Psi$ is the characteristic function of a $J$-orbit $xJ$ on $N_{\hat\alpha}$, this sum will simply be bounded by a fixed (the number of $a_i$'s $+1$) multiple of $\Vol(xJ)^{\frac{1}{2}}$, the $L^2$ norm of $\Psi$ --- recall that ``close to infinity'' the Bernstein maps $e_\Theta$ are induced by measure-preserving identifications of $J$-orbits, cf.\ \cite[Proposition 4.3.3]{SV} and \cite[Theorem 2]{Delorme-Plancherel}. Therefore, for every $x\in N_{\hat\alpha}$ and $a\in A_{X,\hat\alpha}^+$ we get:
$$\left|\left< \Phi - \iota_{\hat\alpha}^* \Phi, a^{-1}\cdot 1_{xJ}\right> \right| \ll  \Vert  \Phi\Vert_{L^2(X)} \cdot \Vol(xJ)^{\frac{1}{2}} Q^J(a) $$ 
where the implicit constant only depends on $J$. Taking into account that the measure on $N_{\hat\alpha}$ is an $A_{X,\hat\alpha}$-eigenmeasure, say with character $\delta_{\hat\alpha}$, we have by definition:
$$\left< \Phi - \iota_{\hat\alpha}^* \Phi,  a^{-1}\cdot 1_{xJ} \right>= \left< \Phi - \iota_{\hat\alpha}^* \Phi, \delta_{\hat\alpha}^{-\frac{1}{2}}(a) 1_{axJ}\right>=$$
$$= (\Phi - \iota_{\hat\alpha}^* \Phi)(ax) \delta_{\hat\alpha}^{-\frac{1}{2}}(a) \Vol(axJ) = (\Phi - \iota_{\hat\alpha}^* \Phi)(ax) \delta_{\hat\alpha}^{\frac{1}{2}}(a) \Vol(xJ),$$
and hence the above inequality becomes:
$$ |\Phi - \iota_{\hat\alpha}^* \Phi|(ax) \ll \Vert  \Phi\Vert_{L^2(X)} \cdot \Vol(a\cdot xJ)^{-\frac{1}{2}} Q^J(a) .$$

This proves the lemma.

\end{proof}

\begin{proof}[Proof of Proposition \ref{Theta to X}]
Fix a $\Theta$, as in a statement of the proposition. We will prove the proposition inductively on $|\Delta_X|-|\Theta|$, the base case $\Theta=\Delta_X$ being trivial. Assume that it has been proven for all orders of $|\Delta_X|-|\Theta|$ smaller than the given one.

We have already reduced the proposition to the case $\mathcal Z(X)=1$, which we will henceforth assume. (Notice, however, that by the inductive assumption we are free to assume the proposition for any smaller value of $|\Delta_X|-|\Theta|$ without this assumption.) We fix an open compact subgroup $J$ and $J$-good neighborhoods $N_\Omega$ of $\Omega$-infinity as in the setup of Lemma \ref{decaying}, and use the notation $N'_\Omega$ as before.

By Lemma \ref{decaying}, it is enough to prove:

\begin{quote}\emph{ For all $\Omega\subsetneq \Delta_X$, the composition of $\iota^*_\Omega \iota_\Theta: L^2(X_\Theta)_\disc\to L^2(X_\Omega)$ with restriction to $N_\Omega$ takes $\mathscr C(X_\Theta)_\disc$ continuously into $\mathscr C(N_\Omega)$.}\end{quote}
\begin{equation}\label{claimforprop}\end{equation}
Indeed, by Lemma \ref{decaying}, for every $f\in \mathscr C(X_\Theta)^J$ the difference:
$$ \iota_\Theta f - \sum_{\Omega\subsetneq\Delta_X} (-1)^{|\Omega|} \tilde\tau_\Omega \iota_\Omega^*\iota_\Theta f$$ is bounded by a fixed decaying function times $\Vert \iota_\Theta f\Vert \ll \Vert f \Vert$, and by the claim the subtrahend above is (continuously) in the Harish-Chandra Schwartz space.

First of all, if $\Omega$ does not contain a conjugate of $\Theta$ then $\left.\iota_\Omega^*\iota_\Theta\right|_{L^2(X_\Theta)_\disc} =0$ and there is nothing to prove. Now let $(\Omega_i)_i$ denote representatives for $W_{X_\Omega}$-conjugacy classes of subsets of $\Omega$ which are $W_X$-conjugate to $\Theta$. (Here, $W_{X_\Omega}\subset W_X$ denotes the little Weyl group of $X_\Omega$, which is generated by the simple reflections corresponding to elements of $\Omega$.) Denote by $\iota_{\Omega_i}^\Omega: L^2(X_{\Omega_i})\to L^2(X_\Omega)$ the analogous Bernstein maps for the variety $X_\Omega$. We claim that there are non-zero integers $d_\Omega(\Omega_i)$ such that:
\begin{equation}\label{previous}
 \iota_\Omega^* \circ \iota_\Theta|_{L^2(X_\Theta)_\disc} = \sum_i \sqrt{d_\Omega(\Omega_i)^{-1}} \iota_{\Omega_i}^\Omega \circ \iota_{\Omega_i}^* \circ \iota_\Theta.
\end{equation}

Indeed, the image of $L^2(X_\Theta)_\disc$ under $\iota_\Omega^* \circ \iota_\Theta$ lies in the direct sum (over all $i$) of the spaces $L^2(X_\Omega)_{[\Omega_i]}$, where $L^2(X_\Omega)_{[\Omega_i]}$ denotes the image of $L^2(X_{\Omega_i})_\disc$ under $\iota_{\Omega_i}^\Omega$. This follows from the transitivity property $\iota_{\Xi}^* = \iota_{\Xi}^{\Omega,*} \circ \iota_\Omega^*$ of the Bernstein maps, and the fact that $\iota_{\Xi,\disc}^*\iota_\Theta=0$ unless $\Xi$ is a $W_X$-conjugate of $\Theta$. 

Let $d_\Omega(\Omega_i) = \# W_{X_\Omega}(\Omega_i,\Omega_i)$.
It follows from Theorem \ref{L2theorem} that the map: $$\sqrt{d_\Omega(\Omega_i)^{-1}} \iota_{\Omega_i}^\Omega  \circ \iota_{\Omega_i}^{\Omega,*}: L^2(X_\Omega)\to L^2(X_\Omega)$$
is the identity on $L^2(X_\Omega)_{[\Omega_i]}$, and zero on the summands $L^2(X_\Omega)_{[\Omega_j]}$ with $j\ne i$. Hence,
$$\sum_i \sqrt{d_\Omega(\Omega_i)^{-1}} \iota_{\Omega_i}^\Omega  \circ \iota_{\Omega_i}^{\Omega,*}$$
is the identity on the image of $L^2(X_\Theta)_\disc$ under $\iota_\Omega^* \circ \iota_\Theta$, and \eqref{previous} follows.

The map $\iota_{\Omega_i}^* \circ \iota_\Theta$ is a continuous map: 
$$\mathscr C(X_\Theta)_\disc \to \mathscr C(X_{\Omega_i})_\disc$$
by Theorem \ref{unitaryscattering},
so by replacing $\Theta$ by $\Omega_i$ we have reduced the claim \eqref{claimforprop} to the statement of the proposition when $\Theta\subset\Omega$ and $X$ is replaced by $X_\Omega$. It now follows by the induction hypothesis.
\end{proof}

Now we come to proving the other direction. We keep assuming that $\mathcal Z(X)=1$, having reduced the problem to this case.

\begin{proof}[Proof of Proposition \ref{X to Theta}]
 Let $f\in \mathscr C(X)^J$, and let $X=\bigsqcup_\Theta N'_\Theta$ be a decomposition as above; then $f_\Theta:=f|_{N'_\Theta} \in \mathscr C(N'_\Theta)\subset \mathscr C(X_\Theta)$. By Theorem \ref{thmdiscrete}, we need to show that the image of $\iota_{\Theta,\disc}^* f$ in $L^2(X_\Theta)_\disc = L^2(\widehat{X_\Theta^L}^\disc, \mathscr L_\Theta)$ actually lies in $C^\infty(\widehat{X_\Theta^L}^\disc,\mathscr L_\Theta)$ (smooth sections), and that the resulting map: $\mathscr C(X)^J\to C^\infty(\widehat{X_\Theta^L}^\disc,\mathscr L_\Theta)$ is continuous.

 First of all, for each $\Theta$ and $\Omega$ consider the composition of maps:
\begin{equation}\label{kmaps} \mathcal S(X_\Omega)\xrightarrow{e_\Omega} \mathcal S(X) \xrightarrow{\iota_{\Theta,\disc}^*} L^2(X_\Theta)_\disc \xrightarrow{\sim} L^2(\widehat{X_\Theta^L}^\disc, \mathscr L_\Theta).
\end{equation}
By \eqref{explicitunitary}, this composition is given by the restriction of the maps $E_{\Theta,\disc}^{*,\Omega}$ of \eqref{asympEisenstein} ($\Omega$-asymptotics of normalized constant terms). Recall that by Proposition \ref{regularEisenstein}, the image of $E_{\Theta,\disc}^{*,\Omega}$ lies in $\Gamma(\widehat{X_\Theta^L}^\disc, \mathscr L_\Theta)$ (i.e.\ rational sections whose poles do not meet the unitary set).

Our goal is show that the maps $E_{\Theta,\disc}^{*,\Omega}$ extend continuously to operators represented by the bottom horizontal row of the following diagram, where the vertical arrows are the natural inclusions:
$$\xymatrix{
 \mathcal S(N'_\Omega) \ar[r] \ar[d] & \Gamma(\widehat{X_\Theta^L}^\disc, \mathscr L_\Theta) \ar[d]\\
 \mathscr C(N'_\Omega) \ar@{-->}[r]	     & C^\infty(\widehat{X_\Theta^L}^\disc, \mathscr L_\Theta).
}$$

This will prove the proposition, once we know it for all $\Omega$.

Fix a connected component $Y$ of $\widehat{X_\Theta^L}^\disc$, and recall that $\Gamma(Y, \mathscr L_\Theta)$ is actually a $\mathcal D(Y)$-module (module for the ring of polynomial differential operators on $Y$). Fix any $D\in \mathcal D(Y)$ and apply it to the operator $E_{\Theta,\disc}^{*,\Omega}$. As we have seen in Lemma \ref{derivatives}, the resulting element:
$$DE_{\Theta,\disc}^{*,\Omega} \in \Hom\left(\mathcal S(X_\Omega), \Gamma(Y, \mathscr L_\Theta)\right)$$ 
(not a $G$-equivariant homomorphism) has the same exponents, possibly with higher multiplicity, as $E_{\Theta,\disc}^{*,\Omega}$, and by Proposition \ref{propsubunitary}, these are either unitary or subunitary with respect to $\mathring A_{X,\Theta}^+$.

We will use the following lemma of linear algebra:
\begin{lemma}
Suppose $S$ is a finitely generated abelian group 
together with a finitely generated submonoid $S^+ \subset S$ that generates $S$. If $S$ has a locally finite action on a complex vector space $V$, with the degrees of all vectors uniformly bounded by an integer $m$, and generalized eigencharacters which are unitary or subunitary with respect to $S^+$, and if $\Vert \bullet \Vert$ is any norm on $V$, there exist a tempered function $T$ on $S$ and a finite subset $S_0\subset S$, depending only on $S^+$ and $m$, with the property that:
$$\Vert s\cdot v\Vert \le T(s) \max_{s'\in S_0} \Vert s'\cdot v\Vert$$
for all $s\in S^+$, $v\in V$.
\end{lemma}

We remind that ``locally finite'' means that the span of the $S$-translates of each vector is finite dimensional, and the degree of a vector is the dimension of this $S$-span. Compare this lemma with \cite[Lemma 10.2.5]{SV}.

\begin{proof}
We may replace $S^+$ by the free monoid on a set of generators, and $S$ by the free group on this set. We can then reduce to the case $S=\mathbb Z$, $S^+=\mathbb N$, because, if $T', S_0'$ work for $(\mathbb Z,\mathbb N)$, then $T(n_1, \dots, n_r) = T'(n_1) \cdots T'(n_r)$, $S_0 = S_0'\times \cdots \times S_0'$ will work for $(\mathbb Z^r,\mathbb N^r)$.

For $(\mathbb Z,\mathbb N)$ we apply induction on the degree: Writing the minimal polynomial of the generator $M:={'1'}\in \mathbb N$ on an element $v\in V$ as $P(x) = (x-\zeta) Q(x)$, and assuming by induction that the lemma holds for the vector $v'=(M-\zeta)v$ (with some tempered function $T'$ and some set $S_0'$ depending only on the degree of $Q$), we get an estimate:
$$ \Vert M^n v \Vert \le \Vert M^{n-1} v'\Vert + |\zeta| \Vert M^{n-1}v \Vert  \le T'(n-1) \cdot \max_{s'\in S_0'} \Vert s' \cdot v'\Vert + \Vert M^{n-1}v\Vert,$$
where we have used the induction hypothesis and the assumption that $|\zeta|\le 1$. Repeating this estimate for $M^{n-1} v$ and so forth, in the end we get:
$$ \Vert M^n v\Vert \le (T'(n-1)+T'(n-2)+\dots + T'(1)) \max_{s'\in S_0'} \Vert s' \cdot v'\Vert + \Vert v\Vert \le T(n) \max_{s'\in S_0} \Vert s' \cdot v\Vert$$
(again using $|\zeta|\le 1$) where $T(n) = 2(T'(n-1)+T'(n-2)+\dots + T'(1))+1$, $S_0 = \{0\} \cup S_0' \cup (S_0'+1)$.
\end{proof}

We now fix a Haar measure $d\sigma$ on $\widehat{X_\Theta^L}^\disc$, which determines norms $\Vert \bullet\Vert_\sigma$ on the fibers of $\mathscr L_\Theta$ over the unitary set, cf.\ \eqref{Plmeasure-Theta}. By regularity, for any $F\in \mathcal S(X_\Omega)$ the numbers 
$$\left\Vert DE_{\Theta,\disc,\sigma}^{*,\Omega}(F)\right\Vert_\sigma,$$
as $\sigma$ varies in $Y$, are uniformly bounded in $\sigma$. If we now fix a set of $J$-orbits on $N'_\Omega$ whose $A_{X,\Omega}^+$-translates cover $N'_\Omega$, and denote by $F_i$ their characteristic functions, we claim that there is a finite set $S_0$ of elements of $A_{X,\Omega}$ and a tempered function $T$ on $A_{X,\Omega}$ such that:
\begin{equation}\label{Eest}\left\Vert DE_{\Theta,\disc,\sigma}^{*,\Omega}(aF_i)\right\Vert_\sigma \le T(a) \max_{s\in S_0} \left\Vert DE_{\Theta,\disc,\sigma}^{*,\Omega}(sF_i)\right\Vert_\sigma
\end{equation}
for all $a\in A_{X,\Omega}^+$. Indeed, this follows from the above lemma, using the fact that $A_{X,\Omega}$ acts on $\mathcal S(X_\Omega)^J$ through a finitely generated quotient, and that the $DE_{\Theta,\disc,\sigma}^{*,\Omega}$ are all $A_{X,\Omega}$-finite with uniformly bounded degree, by Lemma \ref{derivatives}.

For an arbitrary element $\Phi\in L^2(N'_\Omega)^J$, writing it as a series in $A_{X,\Omega}^+$-translates of the $F_i$s:
$$ \Phi = \sum_{i,j} c_{ij} a_j\cdot F_i,$$
its image in $L^2(\widehat{X_\Theta^L}^\disc, \mathscr L_\Theta)$ is given by the corresponding series:
$$ \sum_{i,j} c_{ij} E_{\Theta,\disc}^{*,\Omega}(a_j\cdot F_i).$$

If, in particular, $\Phi\in \mathscr C(N'_\Omega)$, by \eqref{Eest} we deduce that the corresponding series for $DE_{\Theta,\disc}^{*,\Omega}(\Phi)$ converges in $L^2(\widehat{X_\Theta^L}^\disc, \mathscr L_\Theta)$, and is bounded by continuous seminorms on $\mathscr C(N'_\Omega)$. 

Since the seminorms:
$$ f\mapsto \left\Vert Df \right\Vert_{L^2(\widehat{X_\Theta^L}^\disc, \mathscr L_\Theta)}, \,\,\, D\in \mathcal D(Y),$$
form a complete system of seminorms for $C^\infty(\widehat{X_\Theta^L}^\disc, \mathscr L_\Theta)$, we deduce that the maps $E_{\Theta,\disc}^{*,\Omega}$, restricted to $\mathcal S(N'_\Omega)^J$, extend continuously to:
$$ \mathscr C(N'_\Omega)^J \to C^\infty(\widehat{X_\Theta^L}^\disc, \mathscr L_\Theta).$$

This proves Proposition \ref{X to Theta}.

\end{proof}

We are now ready to complete the proof of our main result on the Harish-Chandra Schwartz space:

\begin{theorem}\label{HCtheorem}
For each $\Theta$, orthogonal projection to $L^2(X_\Theta)_\disc$ gives a topological direct sum decomposition:
$$\mathscr C(X_\Theta) = \mathscr C(X_\Theta)_\disc \oplus \mathscr C(X_\Theta)_\cont.$$

 For each $w\in W_X(\Omega,\Theta)$ the scattering map $S_w$ restricts to a topological isomorphism: 
$$\mathscr C(X_\Theta)_\disc\xrightarrow{\sim} \mathscr C(X_\Omega)_\disc.$$ 

The map $\iota^*$ of \eqref{L2} restricts to a topological isomorphism:
\begin{equation}\label{HCtheoremeq}  \mathscr C(X)\xrightarrow{\sim} \left(\bigoplus_{\Theta\subset\Delta_X} \mathscr C(X_\Theta)_\disc\right)^\inv, \end{equation}
where the exponent ${~^\inv}$ denotes invariants of the scattering maps $S_w$.
\end{theorem}

\begin{proof}
The first two statements have been proven in Proposition \ref{propdecomp} and Theorem \ref{unitaryscattering}.

Because of the second statement, the space $\left(\bigoplus_{\Theta\subset\Delta_X} \mathscr C(X_\Theta)_\disc\right)^\inv$ makes sense. By Proposition \ref{X to Theta} and Theorem \ref{L2theorem} the space $\mathscr C(X)$ injects continuously into it. Finally, since $\sum_\Theta (i_\Theta^*\circ i_\Theta)$ is a multiple of the identity on $\left(\bigoplus_{\Theta\subset\Delta_X} L^2(X_\Theta)_\disc\right)^\inv$, it follows from Proposition \ref{Theta to X} that the map from $\mathscr C(X)$ to $\left(\bigoplus_{\Theta\subset\Delta_X} \mathscr C(X_\Theta)_\disc\right)^\inv$ is onto.
\end{proof}

The combination of Theorems \ref{explicitunitary} and \ref{HCtheorem} gives Theorem \ref{HCPW2}, which we repeat for convenience of the reader:

\begin{theorem}[cf.\ Theorem \ref{HCtheorem2}] \label{HCtheorem2}
 The normalized constant terms $E_{\Theta,\disc}^*$ extend to an isomorphism of LF-spaces:
\begin{equation} \label{HCtheorem2eq}\mathscr C(X) \xrightarrow\sim \left(\bigoplus_\Theta C^\infty(\widehat{X_\Theta^L}^\disc, \mathscr L_\Theta)\right)^\inv,\end{equation}
where $~^\inv$ here denotes $\mathscr S_w$-invariants, i.e.\ collections of sections $(f_\Theta)_\Theta$ such that for all triples $(\Theta,\Omega, w\in W_X(\Omega,\Theta))$ we have: $\mathscr S_w f_\Theta = f_\Omega$. 
\end{theorem}

In particular, the existence of a ring $\mathfrak z^\temp(X)$ of multipliers on $\mathscr C(X)$, as described in Corollary \ref{mult-HC}, immediately follows from either of the above two versions of our Paley--Wiener theorem for the Harish-Chandra Schwartz space:

\begin{corollary}\label{mult-HC2} Let $$\mathfrak z^\temp(X) = \left(\bigoplus_\Theta \mathfrak z^\disc(X_\Theta^L)\right)^\inv,$$
where the exponent $~\inv$ denotes invariants of all the isomorphisms induced by triples $(\Theta,\Omega, w\in W_X(\Omega,\Theta))$.

 There is a canonical action of $\mathfrak z^\temp(X)$ by continuous $G$-endomorphisms on $\mathscr C(X)$, characterized by the property that for every $\Theta$, considering the map:
 $$ \iota_{\Theta,\disc}^*: \mathscr C(X) \to \mathscr C(X_\Theta)_\disc$$
we have:
$$ \iota_{\Theta,\disc}^*(z\cdot f)  = z_\Theta (\iota_{\Theta,\disc}^* f)$$
for all $z\in \mathfrak z^\temp(X)$, where $z_\Theta$ denotes the $\Theta$-coordinate of $z$.
\end{corollary}

\begin{proof}
Indeed, $\mathfrak z^\temp(X)$ acts by continuous $G$-automorphisms on the right hand side of \eqref{HCtheoremeq} or \eqref{HCtheorem2eq}, and the action is characterized by the stated property.
\end{proof}

We complete this section by formulating an extension of the properties of Bernstein and scattering maps from the discrete components of Harish-Chandra Schwartz spaces to the whole space. For the proof, we point the reader to the proof of Theorem \ref{smoothextension}, which will be completely analogous.

\begin{theorem}\label{unitaryextension} 
For every triple $(\Theta,\Omega,w\in W_X(\Omega,\Theta))$ the scattering map:
$$S_w: L^2(X_\Theta)\to L^2(X_\Omega)$$
restricts to a topological isomorphism:
$$\mathscr C(X_\Theta) \to \mathscr C(X_\Omega)$$
which is $\mathfrak z^\temp(X_\Theta^L)$-equivariant with respect to the obvious isomorphism:
$$ \mathfrak z^\temp(X_\Theta^L)\xrightarrow\sim \mathfrak z^\temp(X_\Omega^L)$$
induced by $w$.

The Bernstein maps $\iota_\Theta$ and their adjoints $\iota_\Theta^*$ map $\mathscr C(X_\Theta)$ continuously into $\mathscr C(X)$ and vice versa.
\end{theorem}

\section{The Schwartz space} 

We now come to our Paley--Wiener theorem for the Schwartz space of compactly supported, smooth functions on $X$. Besides the properties of the scattering operators $\textswab S_w$ of \S \ref{sec:scatteringgoals}, we will use the following basic result:

\begin{theorem}\label{propbasicdecomp} 
Let $[\Theta]$ run over all associate classes of subsets of $\Delta_X$, and for each such class let $\mathcal S(X)_{[\Theta]}$ denote the space generated by all $e_\Omega \mathcal S(X_\Omega)_\cusp$, $\Omega\in [\Theta]$. Then:

\begin{equation} \label{basicdecomp} \mathcal S(X) = \bigoplus_{[\Theta]} \mathcal S(X)_{[\Theta]}.
\end{equation}

\end{theorem}

\begin{proof}
The sum is direct by Theorem \ref{smoothscattering}. We need to show that the map:
$$ \sum_\Theta e_\Theta: \bigoplus_\Theta \mathcal S(X_\Theta)_\cusp \to \mathcal S(X)$$
is surjective.

We will use induction on the size of $\Delta_X$, the case $\Delta_X=\emptyset$ being tautologically satisfied (because then $\mathcal S(X)=\mathcal S(X)_\cusp$). Assume that the proposition has been proven when $X$ is replaced by $X_\Omega$, for all $\Omega\subsetneq \Delta_X$.  Let us denote by $e_\Theta^\Omega: \mathcal S(X_\Theta) \to \mathcal S(X_\Omega)$ ($\Theta\subset\Omega$) the corresponding maps for the variety $X_\Omega$. Recall the transitivity property:
$$ e_\Omega\circ e_\Theta^\Omega = e_\Theta.$$
Therefore, 
$$ \sum_\Theta e_\Theta\left(\mathcal S(X_\Theta)_\cusp\right) = \mathcal S(X)_\cusp + \sum_{\Theta \ne\Delta_X} e_\Theta\left(\mathcal S(X_\Theta)\right).$$

Assume that $\mathcal S(X)\ne \bigoplus_{[\Theta]} \mathcal S(X)_{[\Theta]}$, then there would be a non-zero subspace $V$ of the smooth dual (i.e.\ $C^\infty(X)$) which would vanish on all the spaces on the right hand side of the last equation. In particular, $e_\Theta^*V= 0$ for all $\Theta\ne \Delta_X$, hence the elements of $V$ are compactly supported modulo the center of $X$. But then they cannot be orthogonal to the cuspidal part $\mathcal S(X)_\cusp=\mathcal S(X)_{[\Delta_X]}$.

\end{proof}

\begin{remark}\label{falseifnotfactorizable}
 This theorem is \emph{false}, in general, in the non-factorizable case, if we define the cuspidal subspace $\mathcal S(X)_\cusp$, as in section \ref{sec:cuspidal}, by requiring that the image under the Plancherel decomposition \eqref{L2discrete} of the (smooth) function is relatively cuspidal. For example, if $X=\PGL_2$ under the $G=\Gm\times \PGL_2$ action (with $\Gm$ acting as a split subtorus by multiplication on the left, and $\PGL_2$ acting by multiplication on the right) then it is known that the tensor product of the trivial character of $F^\times$ by the Steinberg representation $\rm{St}$ of $\PGL_2$ is relatively cuspidal on $X$, while this is not the case for non-trivial characters of $F^\times$. What this means is that the image of an embedding $\rm{St}\hookrightarrow \mathcal S(F^\times\backslash \PGL_2)$ will be orthogonal to $e_\emptyset(\mathcal S(X_\emptyset))$ (this is the property of $1 \otimes \St$ being relatively cuspidal), but also orthogonal to $\mathcal S(X)_\cusp$ (which has no Steinberg-equivariant part, since for generic characters of $F^\times$ the Steinberg representation is not relatively cuspidal).
\end{remark}

Recall that for each $\Theta$ we have defined $\mathcal S^+(X_\Theta)_\cusp$ as the subspace of $C^\infty(X_\Theta)$ generated by all spaces of the form:
$$\textswab S_w \mathcal S(X_\Omega)_\cusp$$
where $\Omega$ is an associate of $\Theta$ and $w\in W_X(\Omega,\Theta)$, and in Theorem \ref{smoothscattering} (and Proposition \ref{extension}) we extended the scattering operators $\textswab S_w$ to isomorphisms between these spaces. 

We are now ready to prove a Paley--Wiener theorem, reminding first that the exponent $~^\inv$ in:
\begin{equation}\label{spaceofinvariants}\left( \oplus_\Theta \mathcal S^+(X_\Theta)_\cusp\right)^\inv
\end{equation}
denotes invariants of these maps. Notice that, as follows easily from the definitions, any element of \eqref{spaceofinvariants} can be obtained by averaging elements of the spaces $\mathcal S(X_\Theta)_\cusp$ via the operators $\textswab S_w$, i.e.:

\begin{lemma}\label{averaging} For any element $f = (f_\Theta)_\Theta$ of \eqref{spaceofinvariants} there is a (non-unique) element
$$(f'_\Theta)_\Theta \in \oplus_\Theta \mathcal S(X_\Theta)_\cusp$$
such that:
\begin{equation}\label{ftheta} f_\Theta = \sum_{\Omega; w\in W_X(\Theta, \Omega)} \textswab S_w f'_\Omega.\end{equation}
\end{lemma} 

\begin{proof}
Let $\Theta_i$ vary in a set of representatives for associate classes of subsets of $\Delta_X$. For each $i$ there is, by definition of the spaces $\mathcal S^+(X_\Theta)$, a collection 
$$(f'_{\Omega,w})_{\Omega\sim \Theta_i; w\in W_X(\Theta_i, \Omega)} \in \oplus_{\Omega\sim \Theta_i; w\in W_X(\Theta_i, \Omega)} \mathcal S(X_\Omega)_\cusp$$ with $f_{\Theta_i} = \sum_{\Omega; w\in W_X(\Theta_i, \Omega)} \textswab S_w f'_{\Omega,w}$.

Setting then $$f'_\Omega = \frac{1}{|W_X(\Omega,\Omega)|}\sum_{w\in W_X(\Theta_i, \Omega)} f'_{\Omega,w}$$ (where $\Theta_i$ is the representative for the associate class of $\Omega$), we easily get from the $W_X(\Theta_i,\Theta_i)$-invariance of $f_{\Theta_i}$ and the associativity properties of the scattering operators that \eqref{ftheta} holds for every $\Theta$.

\end{proof}

\begin{theorem}\label{Schwartztheorem} 
 The sum of the morphisms $e_{\Theta,\cusp}^*$ defines an isomorphism:
\begin{equation}\label{Schwartztheoremeq} \mathcal S(X) \simeq \left( \oplus_\Theta \mathcal S^+(X_\Theta)_\cusp\right)^\inv.\end{equation}
\end{theorem}

\begin{proof}
It is an immediate corollary of Theorem \ref{propbasicdecomp} and Theorem \ref{smoothscattering} that the image of $\bigoplus_\Theta e_{\Theta,\cusp}^*$ lies in $\left( \oplus_\Theta \mathcal S^+(X_\Theta)_\cusp\right)^\inv$. 

Lemma \ref{averaging} shows that the map is surjective, and injectivity follows from Proposition \ref{characterization}.
\end{proof}

It is easy from this to deduce the fiberwise version in terms of normalized constant terms. First of all, for every $\Theta\subset\Delta_X$ let:
$$ \CC^+\left[\widehat{X_\Theta^L}^\cusp, \mathcal L_\Theta\right] \subset  \CC\left(\widehat{X_\Theta^L}^\cusp, \mathcal L_\Theta\right)$$
be the subspace generated by the images of all fiberwise scattering maps $\mathscr S_w$, for $\Omega$ and associate of $\Theta$ and $w\in W_X(\Theta,\Omega)$. Notice that by the regularity of scattering maps on the unitary spectrum (Theorem \ref{fiberwisescattering}), we might as well have written $\Gamma( \,\,\,)$ instead of $\CC(\,\,\,)$. Then it is clear that such an $\mathscr S_w$ induces an isomorphism:
$$\CC^+\left[\widehat{X_\Omega^L}^\cusp, \mathcal L_\Omega\right] \xrightarrow\sim \CC^+\left[\widehat{X_\Theta^L}^\cusp, \mathcal L_\Theta\right],$$
and the combination of Theorems \ref{explicitsmooth} and \ref{Schwartztheorem} gives Theorem \ref{PW2}, which we repeat for convenience of the reader:

\begin{theorem}\label{Schwartztheorem2}
 The normalized cuspidal constant terms $E_{\Theta,\cusp}^*$ give rise to an isomorphism:
\begin{equation}\label{Schwartztheorem2eq}\mathcal S(X) \xrightarrow\sim \left(\bigoplus_\Theta \CC^+[\widehat{X_\Theta^L}^\cusp, \mathcal L_\Theta]\right)^\inv,\end{equation}
where $~^\inv$ here denotes $\mathscr S_w$-invariants.
\end{theorem}

In particular, the existence of a ring $\mathfrak z^\sm(X)$ of multipliers on $\mathcal S(X)$, as described in Corollary \ref{mult-sm}, immediately follows from either of the above two versions of our Paley--Wiener theorem for the Schwartz space:

\begin{corollary}\label{mult-sm2} Let $$\mathfrak z^\sm(X) = \left(\bigoplus_\Theta \mathfrak z^\cusp(X_\Theta^L)\right)^\inv,$$
where the exponent $~\inv$ denotes invariants of all the isomorphisms induced by triples $(\Theta,\Omega, w\in W_X(\Omega,\Theta))$.

 There is a canonical action of $\mathfrak z^\sm(X)$ by continuous $G$-endomorphisms on $\mathcal S(X)$, characterized by the property that for every $\Theta$, considering the map:
 $$ e_{\Theta,\cusp}^*: \mathcal S(X) \to \mathcal S^+(X_\Theta)_\cusp$$
we have:
$$ e_{\Theta,\cusp}^*(z\cdot f)  = z_\Theta (e_{\Theta,\cusp}^* f)$$
for all $z\in \mathfrak z^\sm(X)$, where $z_\Theta$ denotes the $\Theta$-coordinate of $z$.
\end{corollary}

\begin{proof}
Indeed, $\mathfrak z^\sm(X)$ acts by continuous $G$-automorphisms on the right hand side of \eqref{Schwartztheoremeq} or \eqref{Schwartztheorem2eq}, and the action is characterized by the stated property.
\end{proof}

We complete this section by extending the smooth scattering maps from the cuspidal components of Schwartz spaces to the whole space.

\begin{theorem}\label{smoothextension}
There are unique extensions of the smooth scattering maps, for all triples $(\Theta,\Omega,w\in W_X(\Omega,\Theta))$:
$$\textswab S_w: \mathcal S(X_\Theta)\to C^\infty(X_\Omega),$$
such that for all $\Theta'\subset\Theta$, setting $\Omega'=w\Theta'$:
\begin{equation}\label{smoothextensioneq} e_{\Omega'}^\Omega \circ \textswab S_w|_{\mathcal S(X_{\Theta'})_\cusp} =
\textswab S_w \circ e_{\Theta'}^\Theta|_{\mathcal S(X_{\Theta'})_\cusp},
\end{equation}
where as usual we denote by $e_{\Theta'}^\Theta$, $e_{\Omega'}^\Omega$ the analogous equivariant exponential maps for the varieties $X_\Theta$, $X_\Omega$, respectively.

These maps satisfy the same associativity relations as their restrictions to cuspidal spectra (s.\ Theorem \ref{smoothscattering}), and $\textswab S_w$ is $\mathfrak z^\sm(X_\Theta^L)$-equivariant with respect to the obvious isomorphism:
$$ \mathfrak z^\sm(X_\Theta^L)\xrightarrow\sim \mathfrak z^\sm(X_\Omega^L)$$
induced by $w$.
\end{theorem}

\begin{proof}
Given Proposition \ref{basicdecomp} (applied to $X_\Theta$), property \eqref{smoothextensioneq} characterizes the extension of $\textswab S_w$, provided it is unambiguous. For this, we need to show that if $F= e_{\Theta_1'}^\Theta f_1 = e_{\Theta_2'}^\Theta f_2 \in \mathcal S(X_\Theta)$, with $f_1 \in \mathcal S(X_{\Theta_1'})_\cusp$ and $f_2\in \mathcal S(X_{\Theta_2'})_\cusp$, then the left hand side of \eqref{smoothextensioneq}, applied to either $f_1$ or $f_2$, gives the same result.

Theorem \ref{Schwartztheorem}, applied to $X_\Theta$, implies that the kernel of the map:
$$ \sum_{i=1,2} e_{\Theta_i'}^\Theta: \bigoplus_{i=1,2} \mathcal S(X_{\Theta'_i})_\cusp \to \mathcal S(X_\Theta)$$
is generated by elements of the form: 
$ (f_1,-f_2)$ with $f_i \in \mathcal S(X_{\Theta'_i})_\cusp$ and $f_2 = \textswab S_{w'} f_1$ for some $w'\in W_X(\Theta'_2,\Theta'_1)$; thus, we can assume our pair $(f_1,f_2)$ to be of this form.

Then we have: 
$$\textswab S_{w} f_2 = \textswab S_{w}\textswab S_{w'} f_1 = \textswab S_{w w'} f_1 = \textswab S_{ww'w^{-1}} \textswab S_{w} f_1$$
(by the associativity properties of scattering maps),
hence the element $(\textswab S_{w} f_1, - \textswab S_{w} f_2)$ belongs to the kernel of the map:
$$ \sum_{i=1,2} e_{\Omega_i'}^\Omega: \bigoplus_{i=1,2} \mathcal S(X_{\Omega'_i})_\cusp \to \mathcal S(X_\Omega).$$

Thus, the operator $\textswab S_w$ is well-defined on $\mathcal S(X_\Theta)$. The associativity and $\mathfrak z^\sm(X_\Theta^L)$-equivariance properties follow easily from the construction. 
\end{proof}

\section{Examples of scattering operators} \label{sec:examples}

\subsection{Scattering operators in the group case} \label{ssgroupscattering}

Let us consider the case of the group, $X=H$, $G=H\times H$. We consider it not just as a homogeneous space, but as a pointed space, with a distinguished element $1\in H$, which will help us fix isomorphisms for its boundary degenerations. Its boundary degenerations are parametrized by conjugacy classes of parabolics in $H$, where a given class of parabolics $[P]$ corresponding to $\Theta\subset\Delta_X$ we have:
$$ X_{[P]}:= X_\Theta= L_P\times^{P\times P^-} (H\times H) \simeq X_\Theta\simeq L_P^\diag\backslash \left(U_P\backslash H\times U_P^-\backslash H\right).$$
Here we have chosen representatives $P$ for $[P]$ and $P^-$ for the opposite class, $L_P=P\cap P^-$ a Levi subgroup and $U_P, U_P^-$ the corresponding unipotent radicals. The space $X_{[P]}$ lives over an open subset $Y$ in the product of Grassmannians of $H$ corresponding to $[P]\times [P^-]$, and its fiber over $(P, P^-)$ is isomorphic to $L_P$ \emph{canonically up to inner automorphism}. In particular, each fiber has a canonical point $1\in L_P$; the subgroup $H^\diag$ acts transitively on $Y$, preserving those points.

Let us identify the space $\widehat{X_\Theta^L}^\disc_\CC$ with the space $\widehat{L_P}_\CC^\disc$ of isomorphism classes of discrete series representations of $L_P$ by identifying $\sigma\in \widehat{L_P}_\CC^\disc$ with the representation $\tilde\sigma \otimes \sigma$ of the Levi quotient $L_P$ of $P\times P^-$. Fixing a Haar measure on $H$, which induces invariant measures on all $X_{[P]}$, we can identify the smooth dual of the bundle $\mathscr L_{[P]}$ of discrete coinvariants for $X_{[P]}$ with the bundle whose fiber $\widetilde{\mathscr L_{[P],\sigma}}$ over $\sigma\in\widehat{L_P}_\CC^\disc$ is the induced representation $I_{P\times P^-}^{H\times H}(C_{\tilde\sigma})$, where $C_{\tilde\sigma}\subset C^\infty(L_P)$ is the space of matrix coefficients of $\tilde\sigma$. The matrix coefficient map $$\sigma\otimes\tilde\sigma \ni v\otimes \tilde v \mapsto \left< v, \tilde\sigma(\bullet) \tilde v\right>\in C_{\tilde\sigma}$$ allows us to \emph{canonically} identify $I_{P\times P^-}^{H\times H}(C_{\tilde\sigma})$ with $I_P(\sigma)\otimes I_{P^-}(\tilde\sigma)$. 

Now consider a pair $([P],[Q])$ of associate classes of parabolics of $H$, and an element $w\in W_H(Q,P)$. If we fix a representative $P$ for $[P]$ and a Levi subgroup $L$, the pair $([Q],w)$ gives rise to a representative $Q$ of $[Q]$ which shares the Levi subgroup $L$ with $P$; the relation is that ${^wQ}$ shares a minimal parabolic with $P$. We will say that ``the relative position of $P$ and $Q$ is determined by $w$''. 

We let $T_{P|Q}$ denote the rational family of standard intertwining operators, as $\sigma$ varies: 
\begin{equation} T_{P|Q}: I_Q^H(\sigma)\to I_P^H(\sigma).\end{equation}
This family depends on the choice of a Haar measure on $U_P/U_P\cap U_Q$, and is generically invertible with a rational family $ T_{P|Q}^{-1}$ of inverses. 

We denote by $P^-, Q^-$ the opposite parabolics with respect to the chosen Levi. The product:
$$T_{P|Q} \otimes T^{-1}_{Q^-|P^-}: I_Q^H(\sigma)\otimes I_{Q^-}^H(\tilde\sigma) \simeq \widetilde{\mathscr L_{[Q],{^w\sigma}}} \to \widetilde{\mathscr L_{[P],\sigma}} \simeq I_P^H(\sigma)\otimes I_{P^-}^H(\tilde\sigma)$$
does \emph{not} depend on choices of Haar measures, because of the isomorphisms:
$$U_P/U_Q\cap U_P \simeq U_P\cap U_{Q^-} \simeq U_{Q^-}/U_{P^-}\cap U_{Q^-}.$$
Notice also that there is no lack of symmetry here, because $T_{P|Q} \otimes T_{Q^-|P^-}^{-1} = T_{Q|P}^{-1} \otimes T_{P^-|Q^-}$.

\begin{proposition} \label{groupscatteringprop}
For any pair $([P],[Q])$ of associate classes of parabolics of $H$, and $w\in W_H(Q,P)$, the adjoints of the corresponding fiberwise scattering operators $\mathscr S_w$ are the family of operators:
$$ T_{P|Q} \otimes T^{-1}_{Q^-|P^-}: \widetilde{\mathscr L_{[Q],{^w\sigma}}} \to \widetilde{\mathscr L_{[P],\sigma}}.$$
\end{proposition}

\begin{proof}
The result will follow from Proposition \ref{propBmatrices}.

Let us for simplicity denote by $T_0$ the standard intertwining operator between induction from a given parabolic and its opposite; it will be clear from the context which parabolic we are referring to.

By \cite[Lemma 15.7.1]{SV}, the normalized Eisenstein integrals can be written as the composition of matrix coefficients with $T_0^{-1}$ as follows:
$$E_{Q} : I_Q^H(\sigma)\otimes I_{Q^-}^H(\tilde\sigma) \xrightarrow{\Id \otimes T_0^{-1}} I_Q^H(\sigma)\otimes I_{Q}^H(\tilde\sigma)  \xrightarrow{M} C^\infty(H).$$

Here $M$ denotes the matrix coefficient map, which depends on the choice of a Haar measure on $U_{Q^-}$, which brings $ I_Q^H(\tilde\sigma)$ and $I_{Q}^H(\sigma) $ in duality. On the other hand, $T_0: I_{Q}^H(\tilde\sigma) \to I_{Q^-}^H(\tilde\sigma)$ is also proportional to the choice of a measure on $U_{Q^-}$, so the composition of $M$ with $\Id\otimes T_0^{-1}$ does not depend on choices.

Now, we have a commutative diagram:
\begin{equation}\label{cd-group1}\xymatrix{ & C^\infty(H)   & \\   I_Q^H(\sigma)\otimes I_{Q}^H(\tilde\sigma) \ar[ur]^{M}  \ar[rr]^{T_{P|Q} \otimes T_{Q|P}^{-1} } && I_P^H(\sigma)\otimes I_P^H(\tilde\sigma) \ar[ul]_{M}
\\  I_Q^H(\sigma)\otimes I_{Q^-}^H(\tilde\sigma) \ar[u]^{1\otimes T_0^{-1}} \ar[rr]^{T_{P|Q} \otimes T_{Q^-|P^-}^{-1}} 
&& I_P^H(\sigma)\otimes I_{P^-}^H(\tilde\sigma) \ar[u]_{1\otimes T_0^{-1}}
}\end{equation}
of operators varying rationally with $\widehat{L_P}_\CC^\disc$. The fact that the operator $T_{P|Q} \otimes T_{Q|P}^{-1} $ commutes with matrix coefficients follows from the fact that $T_{P|Q}$ is adjoint to $T_{Q|P}$. The fact that the operator on the last horizontal arrow making the diagram commute is $T_{P|Q} \otimes T_{Q^-|P^-}^{-1}$ follows from the fact that $T_0 T_{Q|P} = T_{Q^-|P^-} T_0 = T_{Q^-|P^-} T_{P^-|Q} T_{Q|P}$.

From the commutative diagram \eqref{Bmatrices} (dualized), we now infer that the adjoint of $\mathscr S_w$ is the operator $T_{P|Q} \otimes T^{-1}_{Q^-|P^-}$.
\end{proof}

\subsection{Unramified scattering operators}

We now assume that $G$ is split, and $G, X$ are endowed with compatible models over the ring of integers $\mathfrak o$. We will discuss results of \cite{SaSph} in the light of our current framework. For this, we assume that the conditions of \cite[\S 1.7 and 2.4]{SaSph} hold. In particular, there is a way to identify the universal torus $A_X$ as a torus orbit $A_X\subset X$ (over $\mathfrak o$), so that its ``anti-dominant'' elements $A_X^+\subset A_X(F)$ represent all $K=G(\mathfrak o)$-orbits on $X$.

We will consider only the most degenerate boundary degeneration $X_\emptyset$, which carries an action of $A_X$. Scattering operators for that degeneration are parametrized by elements of the little Weyl group $W_X$ of $X$. We have
\begin{equation}\label{Iwahori}
X_\emptyset(F)/K \simeq A_X(F)/A(\mathfrak o),
\end{equation} where $A$ is the universal Cartan of $G$ (whose quotient as an algebraic variety is $A_X$). The isomorphism \eqref{Iwahori} is fixed so that for ``very antidominant'' elements $a\in A_X^+$, the $K$-orbit represented by $a$ on $X_\emptyset$ corresponds, under the exponential map, to the $K$-orbit represented by $a$ on $X$, cf.\ \cite[Theorem 4.2]{SaSatake}.

A technical comment is in order: For the purpose of interpreting expressions of the form $e^{\check\alpha}(\chi)$, where $\check\alpha$ is a coroot of the universal Cartan of $G$ and $\chi$ an unramified character of $A_X$ (or $A_X'=$the image of $A(F)$ in $A_X(F)$), we identify $A_X$ as a quotient of the universal Cartan $A$ in such a way that the action of $A_X$ on the open Borel orbit $\mathring X_\emptyset\subset X_\emptyset$ is compatible with the action of $A=B/N$ on  $\mathring X_\emptyset/N$. Then, $e^{\check\alpha}(\chi)$ just means the value of $\chi$ on $\check\alpha(\varpi)$, where $\varpi$ is a uniformizer in the field. (We use exponential notation, because we use additive notation for the coroots.) This convention is compatible with \cite{SaSph}.

We may interpret the functional equations established in \emph{loc.cit}.\ in terms of normalized Eisenstein integrals and scattering operators as we did above for the group, but we do not actually need to worry about normalization: Indeed, Theorem 4.2.2 in \emph{loc.cit}.\ implies the following, which we state before defining the terms used:

\begin{theorem}\label{withBw}
For every $w\in W_X$ there is a rational family of $\mathcal H(G,K)$-equivariant operators 
$$\underline B_w(\chi) := \left(\prod_{\check\alpha>0, w\check\alpha<0} (-e^{\check\alpha}(\chi))\right) \underline{b}_w(\delta_{(X)}^\frac{1}{2}\chi): C^\infty(X_\emptyset)^{\chi,K} \to C^\infty(X_\emptyset)^{{^w\chi},K},$$
satisfying the cocycle conditions $\underline B_{w'}({^w\chi}) \circ \underline B_{w'w}(\chi) = \underline B_w(\chi)$,
such that, for a Zariski dense subset of $X$-distinguished Satake parameters, the space of $e_\emptyset^*$-asymptotics of $\mathcal H(G,K)$-eigenfunctions with those Satake parameters is precisely the space of all $\underline B_w$-invariants.
\end{theorem}

The theorem itself does not talk about asymptotics, actually, but about the evaluation of $\mathcal H(G,K)$-eigenfunctions (where $\mathcal H(G,K)$ denotes the unramified Hecke algebra $G$ with respect ot $K$) on $A_X^+\subset X(F)$. It was explained in \cite{SaSatake} that can also be seen as a formula for $e_\emptyset^*$-asymptotics on $X_\emptyset$.

We explain the notation: The space $C^\infty(X_\emptyset)^{\chi}$ is the subspace of $C^\infty(X_\emptyset)$ where $A(F)$ acts with unramified character $\chi$; this notation is compatible with the notation $C^\infty(X_\Theta)_\cusp^{\tilde\sigma}$ that we have been using for the dual of the fiber $\mathcal L_{\Theta,\sigma}$ (the index $~_\cusp$ here is superfluous).

The character $\chi$ lives in the space of unramified characters of $A(F)$ which are trivial on the kernel of $A\to A_X$. The normalized action of $A$ explained in \ref{sseigenmeasures} implies that $C^\infty(X_\emptyset)^{\chi}$ is a direct sum of copies of the normalized induced representation $I_{P(X)^-}(\chi)$; thus, its Satake parameter is the $W$-conjugacy class of $\delta_{(X)}^\frac{1}{2}\chi$, where $\delta_{(X)}$ denotes the modular character of $L(X)$. Those are the ``$X$-distinguished Satake parameters'' of the theorem. 

The notation $\underline{b}_w$ refers precisely to the operators (matrices there, because a basis has been chosen) denoted by the same symbol in Theorem 4.2.2, while the notation $\underline B_w$ is adapted from (6.1) of \emph{loc.cit}.\ (s.\ also Theorem 1.2.1 there), which refers to a slightly special case. The coroots $\check\alpha$ appearing in the relation between $\underline B_w$ and $\underline b_w$ are the coroots of $G$.

Finally, the notion of ``$\underline B_w$-invariants'' is completely analogous to the ``$\mathscr S_w$-invariants'' of our main theorems: a vector 
\begin{equation}\label{Wspaces} (f_w)_w\in \bigoplus_{w\in W_X} C^\infty(X_\emptyset)^{^w\chi}
\end{equation} is in the space of ``$\underline B_w$-invariants'' if for every $w, w'\in W_X$ we have:
$$ f_{w'w} = \underline B_{w'} f_w.$$
Notice that the map from $W_X$-conjugacy classes of $\chi$'s to $X$-distinguished Satake parameters is not necessarily injective (as happens, for example, when $X=N\backslash G$, with $N$ maximal unipotent, where $W_X=1$). Therefore, to obtain the $e_\emptyset^*$-asymptotics of all $\mathcal H(G,K)$-eigenfunctions with a given Satake parameter, as in the theorem, one might need to take the direct sum of $\underline{B}_w$-invariants of several of the spaces \eqref{Wspaces}.

From Theorem \ref{withBw} we can now deduce:

\begin{proposition}
In the notation of Theorem \ref{withBw}, the adjoints of the fiberwise scattering operators $$\mathscr S_w(\chi): \mathcal L_{\emptyset,{^w\chi}}\to \mathcal L_{\emptyset,{\chi}},$$ restricted to $K$-invariants, are given by:
\begin{equation}
\mathscr S_{w^{-1}}(\chi)^* = \underline{B}_w(\chi^{-1}): C^\infty(X_\emptyset)^{\chi^{-1},K} \to C^\infty(X_\emptyset)^{{^w\chi^{-1}},K}.
\end{equation}
\end{proposition}

\begin{proof}
From the fact that the images of normalized Eisenstein integrals:
$$E_{\emptyset,\chi}: C^\infty(X_\emptyset)^{\chi^{-1}}\to C^\infty(X)$$ 
 also span the space of $\mathcal H(G,K)$-eigenfunctions on $X$ with given Satake parameter, for a Zariski-dense set of $X$-distinguished Satake parameters, we deduce:
$$e_\emptyset^* E_{\emptyset,\chi} = \sum_{w\in W_X} \underline B_w (\chi^{-1}) : C^\infty(X_\emptyset)^{\chi^{-1}} \to C^\infty(X_\emptyset) ,$$
and by the definition of the fiberwise scattering maps in \eqref{decompEisensteineq}, the claim follows.
\end{proof}

Up to this point we have presented nothing more than a new symbol for the scattering maps; however, the results of \cite{SaSph} now give rise to a many examples of scattering operators, restricted to unramified vectors. We will only discuss the two most characteristic examples, that of course are much older than \emph{loc.cit.}

\begin{example} (Whittaker model.) \label{exWhit}
Consider the case of $X=N^-\backslash G$, where $N^-$ is a maximal unipotent subgroup over $\mathfrak o$, and $N^-(F)$ is equipped with a non-degenerate character $\Psi$. We keep assuming that $G$ is split, for simplicity. The character is chosen to be trivial on all $\check\alpha(\mathfrak o)$ for all simple coroots $\check\alpha$, but non-trivial on $\check\alpha(\varpi^{-1} \mathfrak o)$, where $\varpi$ denotes a uniformizer. In this case, $X_\emptyset = X$ as varieties, but with trivial character on $N^-(F)$. For every unramified character $\chi$ of the universal Cartan $A(F)$, the space $C^\infty(X_\emptyset)^K$ is $1$-dimensional (isomorphic to the unramified vectors of the normalized principal series $I_{B^-}^G(\chi)$), with a canonical basis element $\varphi^-_{K,\chi}$ which is equal to $1$ on $N^-1K$. (The normalization of the character $\Psi$ on $N^-(F)$ makes the double coset $N^-1K$ unambiguous; the exponent $^-$ on $\varphi$ is to remind that we are using the opposite Borel than that containing $N^-$ to identify the character $\chi$ with a character of the universal Cartan.)

The Shintani-Casselman-Shalika formula states that, in this case, $\underline b_w(\chi) = 1$ for all $w\in W_X=W$ in terms of the canonical basis elements, i.e.\ $\underline b_w(\chi) \varphi^-_{K,\chi} = \varphi^-_{K,{^w\chi}}$, cf.\ \S 5.5 in \emph{loc.cit}. Hence, 
$$\mathscr S_w^*(\chi)=\left(\prod_{\check\alpha>0, w\check\alpha<0} (-e^{-\check\alpha}(\chi))\right): C^\infty(X_\emptyset)^{\chi^{-1},K} \to C^\infty(X_\emptyset)^{{^w\chi^{-1}},K}.$$

In particular, the scattering operators have no poles and we get:

\begin{corollary}
In the Whittaker case, we have $\mathcal S^+(X_\emptyset)^K = \mathcal S(X_\emptyset)^K$.
\end{corollary}

This is the only example that we know where the extended Schwartz space of the boundary degeneration is equal to the original Schwartz space, as far as $K$-invariants go. As we will see in the next subsection, this is not true for Iwahori-invariants.
\end{example}

\begin{example} (Group case.) \label{exGp}
We discussed scattering operators for the group case $X=H$ in the previous subsection, but for unramified vectors we can also describe them explicitly using Macdonald's formula for spherical functions. We keep assuming that $G$ (hence $H$) is split, for simplicity. Using the notation of \S \ref{ssgroupscattering}, for $[P]=[B]$ (where $B$ here denotes the class of Borel subgroups of $H$, not of $G$) we may identify the smooth dual $C^\infty(X_\emptyset)^{\chi^{-1}}$ of $\mathcal L_{[B],\chi}$ with  $I_B(\chi)\otimes I_{B^-}(\chi^{-1})$ (since the fiber of $X_{[B]}$ over $B$ is trivialized, cf.\ \S \ref{ssgroupscattering}). Here $\chi$ is an unramified character of $B$, and we keep the convention from \S \ref{ssgroupscattering} of using the parameter $\chi$ to denote the representation $\chi\otimes \chi^{-1}$ of $B\times B^-$.

Using the canonical basis vector of $(I_B(\chi)\otimes I_{B^-}(\chi^{-1}))^K$, we may again express scattering operators on $K=H(\mathfrak o)\times H(\mathfrak o)$-invariants as scalars.
In this case we have $W_X = W_H$ (the Weyl group of $H$), and Macdonald's formula implies:
$$\mathscr S_w^*(\chi): \left(\prod_{\check\alpha>0, w\check\alpha<0} (-e^{-\check\alpha}) \frac{1-q^{-1} e^{\check\alpha}}{1-q^{-1} e^{-\check\alpha}}\right)(\chi): C^\infty(X_\emptyset)^{\chi^{-1},K} \to C^\infty(X_\emptyset)^{{^w\chi^{-1}},K}$$
where the coroots $\check\alpha>0$ in the product above are the positive (with respect to $B$) coroots of $H$, \emph{not} all positive coroots of $G$.
\end{example}

\subsection{Examples with Iwahori-fixed vectors}

Let us consider the Whittaker and the group case in rank one.

\begin{example} (Whittaker model)
Consider the Whittaker model of $G=\PGL_2$, with conventions (about integral models and characters) as above. Let $w$ be the non-trivial element of the Weyl group. We can reinterpret Example \ref{exWhit} as saying that 
\begin{equation}\label{STWhit}\mathscr S_w^*(\chi)= \frac{1-e^{-\check\alpha}}{1-q^{-1} e^{\check\alpha}}(\chi) T_0: \end{equation}
$$ C^\infty(X_\emptyset)^{\chi^{-1}} = I_{B^-}^G(\chi^{-1}) \to I_{B^-}^G({^w\chi^{-1}}) = C^\infty(X_\emptyset)^{{^w\chi^{-1}}}$$

Indeed, this holds for $K$-fixed vectors by the formula 
$$T_0 \varphi^-_{K,\chi^{-1}} = \frac{1-q^{-1} e^{\check\alpha}}{1-e^{\check\alpha}}(\chi) \varphi^-_{K,{^w\chi^{-1}}},$$ but for $\chi$ in general position these generate the whole representation.

The rational family of operators \eqref{STWhit} may be regular on $K$-fixed vectors, but this is not true for all vectors in the representation. More precisely, for a holomorphic family of functions $\chi\mapsto \varphi_{\chi^{-1}}\in C^\infty(X_\emptyset)^{\chi^{-1}}$ with the property that $\varphi_{\delta_B^{\frac{1}{2}}}$ does not belong to the trivial subrepresentation $\CC \subset I_{B^-}^G(\delta_B^\frac{1}{2})$, the section $\mathscr S_w^*(\chi)(\varphi_{\chi^{-1}})$ has a simple pole at $\chi = \delta_B^{-\frac{1}{2}}$. It follows from this that we have a short exact sequence:
$$ 0 \to \mathcal S(X_\emptyset)^J \to \mathcal S^+(X_\emptyset)^J \to \St^J\to 0$$ 
of $\mathcal H(G,J)$-modules, where $J$ is the Iwahori subgroup and $\St$ denotes the Steinberg representation. (Restricting to Iwahori-invariants is just a way to isolate the spectral contribution of unramified principal series.) The quotient $\St^J$ lives over the character $\chi=\delta^{-\frac{1}{2}}$ of $A$, which is where the trivial representation is a quotient of $\mathcal S(X_\emptyset)_\chi = \mathcal L_{\emptyset,\chi}$. The trivial representation is no longer a quotient of $\mathcal S^+(X_\emptyset)$, as should be the case for the Whittaker model.
\end{example}

\begin{example}
(The group case)
Consider the case $X=H=\PGL_2$ under the $G=\PGL_2\times \PGL_2$-action. We have seen that for the non-trivial element $w\in W_X=W_H$ the scattering map is given by:
$$ \mathscr S_w^*(\chi) = T_0 \otimes T_0^{-1}: C^\infty(X_\emptyset)^{\chi^{-1}} \simeq I_{B}(\chi) \otimes I_{B^-}(\chi^{-1})\to C^\infty(X_\emptyset)^{{^w\chi}^{-1}},$$
where again we denote here by $B$ a Borel subgroup of $H$ (not $G$).

It can easily be seen that, for $\chi$ unramified, $T_0 \otimes T_0^{-1}$ has poles precisely at $\chi=\delta^{\pm\frac{1}{2}}$; more precisely, for a holomorphic family of functions $\chi\mapsto \varphi_{\chi^{-1}}\in C^\infty(X_\emptyset)^{\chi^{-1}}$, the section $ \mathscr S_w^*(\chi)(\varphi_{\chi^{-1}})$ has a (simple) pole at $\chi = \delta^{\pm\frac{1}{2}}$ if and only if the specialization of $\varphi_{\chi^{-1}}$ at that point does not belong to any proper subrepresentation of $$C^\infty(X_\emptyset)^{\delta^{\mp \frac{1}{2}}} \simeq I_{B}(\delta^{\pm\frac{1}{2}}) \otimes I_{B^-}(\delta^{\mp \frac{1}{2}}).$$

From this it can be inferred that we have a short exact sequence:
$$ 0 \to \mathcal S(X_\emptyset)^J \to \mathcal S^+(X_\emptyset)^J \to \St^J\otimes \St^J \oplus \CC \otimes \CC\to 0,$$ 
with the quotient $\St^J\otimes \St^J$ living over $\chi = \delta^{-\frac{1}{2}}$ and the quotient $\CC \otimes \CC$ living over $\chi=\delta^{\frac{1}{2}}$. The fiber $V$ of $\mathcal S^+(X_\emptyset)$ over either of $\chi = \delta^{\pm\frac{1}{2}}$ admits a short exact sequence:
$$0 \to \St\otimes \CC \oplus \CC \otimes\St \to V \to \St\otimes \St \oplus \CC \otimes \CC \to 0$$
and, of course, as a result both $\St\otimes \St$ and $\CC \otimes \CC$ are quotients of $\mathcal S(H)$.
\end{example}

\section{The Bernstein center and the group Paley--Wiener theorem} \label{sec:group}

\subsection{The Bernstein center}\label{ssBcenter} 
We will now see how our Paley--Wiener theorem, and in particular the description of multipliers (Corollary \ref{mult-sm}), implies the well-known theorem on the structure of the Bernstein center in the case of the group, $X=H$, $G=H\times H$.  The argument is inductive in the size of $H$; in particular, we have used the structure of the Bernstein center for its proper Levi subgroups in Corollary \ref{finite2} and hence Proposition \ref{fibercuspidal} in order to deduce our Paley--Wiener theorem and the existence of the multiplier ring $\mathfrak z^\sm(H)$ on $\mathcal S(H)$.

Recall that the Bernstein center $\mathfrak z(H)$ is, by definition, the center of the category $\mathcal M(H)$ of smooth representations of $H$, i.e.\ the algebra of natural transformations of the idendity functor of $\mathcal M(H)$. When $X=H$ the boundary degenerations $X_\Theta$, $\Theta\subset \Delta_X$ are parametrized by classes of parabolics in $H$, where for a given parabolic $P$ corresponding to $\Theta\subset\Delta_X$ we have:
$$ X_P:= X_\Theta\simeq L_P^\diag\backslash \left(U_P\backslash H\times U_P^-\backslash H\right) \simeq L_P\times^{P\times P^-} (H\times H).$$
Here $P^-$ is an opposite parabolic, $L_P=P\cap P^-$ a Levi subgroup and $U_P, U_P^-$ the corresponding unipotent radicals. 

For all $H\times H$-representations that appear below, if not specified otherwise, we let the Bernstein center of $H$ act via the embedding $H\xrightarrow{\Id \times 1} H\times H$.

\begin{theorem}
 \begin{enumerate}
  \item The canonical morphism: 
  \begin{equation}\label{centertoS}
\mathfrak z(H) \to \End_{H\times H} (\mathcal S(H))   
  \end{equation}
is an isomorphism.
  \item For every class of parabolics $P$ in $H$ (corresponding to $\Theta\subset\Delta_X$) the Bernstein center acts fiberwise, i.e.\ $\mathfrak z^\cusp(X_\Theta^L) = \CC[\hat L^\cusp]$-equivariantly, on $\mathcal S(X_\Theta)_\cusp \simeq \CC[\hat L_P^\cusp, \mathcal L_\Theta]$.
  \item The action of any element of $\mathfrak z(H)$ on each fiber of $\mathcal L_\Theta$ is scalar; this scalar varies polynomially on $\hat L^\cusp_{P,\CC}$, i.e.\ we get a canonical morphism:
  \begin{equation}\label{Bmorphism}
    \mathfrak z(H) \to \bigoplus_P \CC[\widehat L_P^\cusp].  
  \end{equation}
  \item The above map gives rise to an isomorphism:
  \begin{equation}
    \mathfrak z(H) \xrightarrow\sim \left(\bigoplus_P \CC[\widehat L_P^\cusp],\right)^\inv = \mathfrak z^\sm(H),
  \end{equation}
  where the exponent $~^\inv$ denotes invariants with respect to the isomorphisms: 
  $$\widehat L_P^\cusp \simeq \widehat L_Q^\cusp$$
  induced by all $w\in W_H(P,Q)$.
 \end{enumerate}
\end{theorem}

\begin{proof}
 \begin{enumerate}
  \item Choose a Haar measure $dh$ on $H$, and let $z\mapsto \alpha(z)$ denote the morphism \eqref{centertoS}. We can construct an inverse to $\alpha$ as follows: Let $(\pi,V)$ be a smooth representation of $H$ and let $J$ be an open compact subgroup. For $Z\in \End_{H\times H} (\mathcal S(H))   $ we define an endomorphism $\beta(Z)$ of $V^J$ by:
  $$\beta(Z) (v) = \pi(Z(1_J/\Vol(J)) dh) (v),$$
  where $1_J$ is the characteristic function of $J$. It is easy to see that this defines an endomorphism $\beta(Z)$ of $V$, and that the collection of these endomorphisms is an element of the Bernstein center (also to be denoted by $\beta(Z)$). Finally, the fact that $\beta$ is inverse to $\alpha$ follows from applying any $z\in \mathfrak z(H)$ to the morphism of smooth $H$-representations:
  $$ \mathcal S(H) \otimes \pi \ni f\otimes v \mapsto \pi(fdh) (v) \in  \pi,$$
  where the left hand side is considered as an $H$-module only via the action on $\mathcal S(H)$ by left multiplication.
  
  \item This is obvious from the definition of the Bernstein center and the fact that the $\mathfrak z^\cusp (X_\Theta^L)$-action commutes with the $G=H\times H$-action.
  
  \item The action is generically scalar because for $\sigma\in \widehat L^\cusp_\CC$ in general position the representations $I_{P^-}^H(\sigma)$ and $I_P^H(\sigma)$ are irreducible, s.\ Lemma \ref{lemmainduced}. On the other hand, it has to preserve the space $\mathcal S(X_\Theta)_\cusp \simeq \CC[\hat L_P^\cusp, \mathcal L_\Theta]$ of regular sections of $\mathcal L_\Theta$, so it has to be polynomial in $\sigma$.
  
  \item From our Paley--Wiener theorem (e.g.\ in the form of Theorem \ref{Schwartztheorem2}) and the $\mathfrak z^\cusp(X_\Theta^L)$-equivariance properties of the scattering maps, it follows that the image of \eqref{Bmorphism} has to lie in the invariants. On the other hand, by the inverse of \eqref{centertoS} and the fact that $\mathfrak z^\sm(H)\subset\End_{H\times H}(\mathcal S(H))$, every invariant induces an $H\times H$-equivariant endomorphism of $\mathcal S(H)$, thus by the first assertion of this proposition we get the desired isomorphism.
 \end{enumerate}

\end{proof}

\subsection{Paley--Wiener theorem}
In the case of the group, $X=H$, $G=H\times H$, we would like to explain the relation of our theorem to the well-known Paley--Wiener theorem of Bernstein \cite{BePadic} and Heiermann \cite{Heiermann}. We clarify that our theorem goes only half-way towards their result; for the other half, one needs to appeal to Proposition (0.2) of \cite{Heiermann}, which is probably also the hardest part of that paper. This is because the Paley--Wiener theorem of Bernstein and Heiermann for the group does not generalize (as a statement) to spherical varieties; and there is a non-trivial distance to cover in order to obtain one from the other, accomplished through the aforementioned proposition of Heiermann. In fact, the steps taken in part A of \cite{Heiermann} can be recast in the setting of our general proof; thus, our work provides a weak generalization, but not a new proof of the Paley--Wiener theorem for reductive groups. We find it important, nevertheless, to explain the connection.

To state the Paley--Wiener theorem of Bernstein and Heiermann we will use the language of bundles, as in \S \ref{sec:generalities}, \ref{sec:coinvariants}; we will not explicitly detail the algebraic structure of the bundles that we will encounter, since the process is identical to the one we have used thus far.  

\begin{theorem}[{Bernstein \cite{BePadic}, Heiermann \cite{Heiermann}}]\label{PWgroup}
For every parabolic $P$ of $H$, denoting its Levi quotient by $L$, consider the bundle $\sigma\mapsto \End\left(I_P^H(\tilde\sigma)\right)$ over $\hat L^\cusp_\CC$. 

Fixing a Haar measure $dh$, for every smooth representation $\pi$ we have the canonical map:
$$ \mathcal S(H) \ni f \mapsto \pi(f dh) \in \End(\pi).$$

Then this map gives rise to an isomorphism:
$$ \mathcal S(H) \xrightarrow\sim \left( \bigoplus_P \CC\left[\sigma\in \hat L^\cusp, \End\left(I_P^H(\tilde\sigma)\right)\right] \right)^\inv,$$
where:
\begin{itemize}
 \item $P$ ranges over all conjugacy classes of parabolics;
 \item the exponent $~^\inv$ refers to sections of the bundle of endomorphisms which commute with all standard intertwining operators. 
\end{itemize}
\end{theorem}

We will use the notation of \S \ref{ssgroupscattering}. In particular, $T_{Q|P}$ is the standard intertwining operator between representations induced from parabolics $P, Q$ which share a common Levi subgroup (depending on a choice of Haar measure on $U_Q/U_Q\cap U_P$), $X_{[P]}$ denotes the boundary degeneration corresponding to a class $[P]$ of parabolics in $H$, and the space $\widehat{X_{[P]}^L}^\cusp_\CC$ is identified with $\widehat{L}_\CC^\cusp$ as explained there. 

By setting $I_P^H(\sigma)$ in duality with $I_{P}^H(\tilde\sigma)$ (that depends on the choice of a Haar measure on $U_{P^-}$), the bundle with fibers $\sigma\mapsto \End\left(I_P^H(\tilde\sigma)\right)$ of Theorem \ref{PWgroup} is identified with the bundle whose fiber over $\sigma \in \widehat L_\CC^\cusp$ is $I_P^H(\tilde\sigma)\otimes I_{P}^H(\sigma)$.

Thus, the morphism $f\mapsto \pi(f dh)$ can be understood as a morphism:
\begin{equation}\label{MCdual}M^*:\mathcal S(X)\to \CC\left[ \sigma\in \hat L^\cusp, I_P^H(\tilde\sigma)\otimes I_{P}^H(\sigma)\right]\end{equation}
where the notation $M^*$ is due to the fact that this is dual to the operation of taking matrix coefficients. On the other hand, the condition of invariance under standard intertwining operators in Theorem \ref{PWgroup} can be tranlated to the condition of invariance under the operators:
$$\CC\left(\sigma\in \hat L^\cusp, I_P^H(\tilde\sigma)\otimes I_{P}^H(\sigma)\right) \xrightarrow {T_{Q|P} \otimes T_{P|Q}^{-1} } \CC\left(\sigma\in\hat L^\cusp, I_Q^H(\tilde\sigma)\otimes I_{Q}^H(\sigma)\right).$$

By \cite[\S 15.7]{SV}, one obtains the normalized cuspidal constant terms $E_{\Theta,\cusp}^*$ out of this by composing with the inverse of the standard intertwining operator $T_0: I_{P^-}^H(\sigma) \to I_{P}^H(\sigma)$ in the second variable:
$$\mathcal S(X)\xrightarrow{M^*} \CC\left[ I_P^H(\tilde\sigma)\otimes I_{P}^H(\sigma)\right]\xrightarrow{1\otimes T_0^{-1}} \CC\left( I_P^H(\tilde\sigma)\otimes I_{P^-}^H(\sigma)\right),$$
where we have for brevity omitted $\hat L^\cusp$ from the notation.

Thus, we have a commutative diagram:
\begin{equation}\label{cd-group2}\xymatrix{ & \mathcal S(H) \ar[dl]_{M^*} \ar[dr]^{M^*} & \\  \CC\left[ I_P^H(\tilde\sigma)\otimes I_{P}^H(\sigma)\right] \ar[d]_{\simeq}^{1\otimes T_0^{-1}} \ar[rr]^{T_{Q|P} \otimes T_{P|Q}^{-1} } && \CC\left[ I_Q^H(\tilde\sigma)\otimes I_Q^H(\sigma)\right] \ar[d]_{\simeq}^{1\otimes T_0^{-1}} 
\\  \CC\left( I_P^H(\tilde\sigma)\otimes I_{P^-}^H(\sigma)\right) \ar[rr]^{T_{Q|P} \otimes T_{P^-|Q^-}^{-1}} 
&& \CC\left( I_Q^H(\tilde\sigma)\otimes I_{Q^-}^H(\sigma)\right)
}\end{equation}
which is dual to \eqref{cd-group1}, so the compositions of slanted and vertical arrows are the normalized constant terms. Notice that up to this point we have made choices of Haar measures on $H$ and $U_{P^-}$ (so that the left slanted arrow $M^*$ is proportional to the measure $dh$ and inversely proportional to $U_{P^-}$), and of a Haar measure on $U_P$ (to which $T_0^{-1}$ is inversely proportional; and similarly when $P$ is replaced by $Q$. The measure $dh$ also induces measures on the boundary degenerations $X_{[P]}$, $X_{[Q]}$, and we leave to the reader to check that the choices of measures cancel each other out when we identify the bundles in the bottom row with the bundles of cuspidal coinvariants $\mathcal L_{[P]}$, resp.\ $\mathcal L_{[Q]}$.

Our Theorem \ref{Schwartztheorem}, together with Proposition \ref{groupscatteringprop}, states that the sum of normalized constant terms induces an isomorphism:
\begin{equation}\label{fd} \mathcal S(H) \xrightarrow{\sim} \left(\bigoplus_{P}  \CC^+\left[ I_P^H(\tilde\sigma)\otimes I_{P^-}^H(\sigma)\right] \right)^\inv,\end{equation}
where $\inv$ denotes invariants of the fiberwise scattering maps $\mathscr S_w = T_{Q|P} \otimes T_{P^-|Q^-}^{-1}$. Recall that the space $\CC^+\left[ I_P^H(\tilde\sigma)\otimes I_{P^-}^H(\sigma)\right]$ is generated by applying these scattering maps to regular sections.

To see that this implies Theorem \ref{PWgroup}, the only non-trivial statement to prove is that \emph{every} element of:
\begin{equation}\label{fd2}\left( \bigoplus_Q \CC\left[I_Q^H(\tilde\sigma)\otimes I_Q^H(\sigma)\right] \right)^\inv
\end{equation}
corresponds to an element of the right hand side of \eqref{fd} under the diagram \eqref{cd-group2}, but this is \cite[Proposition 0.2]{Heiermann} which, in our language, states:

\begin{proposition}[Heiermann \cite{Heiermann}]
 For every element $\varphi = (\varphi_Q)_Q$ of \eqref{fd2} there is an element $\xi = (\xi_P)_P \in \bigoplus_{P}  \CC\left[ I_P^H(\tilde\sigma)\otimes I_{P^-}^H(\sigma)\right]$ such that:
 $$ \varphi = \left( \sum_{P\sim Q, w\in W_H(P,Q)} T_{Q|P} \otimes T_{Q|P^-} \xi_P \right)_Q.$$
\end{proposition}

Notice that $T_{Q|P} \otimes T_{Q|P^-} = T_{Q|P} \otimes (T_0\circ T_{P^-|Q^-}^{-1})$. Thus, under the vertical arrows of diagram \eqref{cd-group2}, the element $\varphi$ corresponds to the element:
$$\left( \sum_{P\sim Q, w\in W_H(P,Q)} T_{Q|P} \otimes T_{P^-|Q^-}^{-1} \xi_P \right)_Q$$
of $ \left(\bigoplus_{Q}  \CC^+\left[ I_Q^H(\tilde\sigma)\otimes I_{Q^-}^H(\sigma)\right] \right)^\inv$. This recovers Theorem \ref{PWgroup} on the basis of \cite[Proposition 0.2]{Heiermann}.

\appendix

\section{Characterization of strongly factorizable spherical varieties} \label{app:factorizable}

In this appendix we assume that $G$ is split.

Recall that a homogeneous spherical variety $X$ is called \emph{factorizable} if the rank of $X^\ab$ is equal to the rank of $\mathcal Z(X)$, and that a wavefront spherical variety is called \emph{strongly factorizable} if all its Levi varieties are factorizable.

We will characterize factorizable and strongly factorizable spherical varieties in terms of combinatorial invariants attached to $X$. We refer the reader to \cite{Lu} for more details on the definitions and properties of these invariants.

Recall that the group $\varchi(X)=\Hom(A_X,\Gm)$ defined previously denotes the subgroup of characters of a Borel subgroup $B$ which are trivial on generic stabilizers or, equivalently, the group of eigencharacters of the Borel subgroup on the space $F(X)^{(B)}$ of non-zero rational $B$-eigenfunctions on $X$. By definition, the group $\varchi(X)$ is a subgroup of $\varchi(A)$, the character group of the universal Cartan $A=B/N$.  

The little Weyl group $W_X$ acts on $\varchi(X)$. The character group of $\mathcal Z(X)$ can be identified with the quotient of $\varchi(X)$ by the characters in the subspace of $\varchi(X)_\QQ = \varchi(X)\otimes\QQ$ spanned by the set $\Delta_X$ of spherical roots. 

Let $\mathcal D$ be the set of \emph{colors}, i.e.\ prime $B$-stable geometric divisors. Each of them induces a valuation on the function field over the algebraic closure $\bar F(X)$ and, by restriction to $B$-eigenfunctions, a map $\rho_X: \mathcal D\to \Hom(\varchi(X),\Z)$. Indeed, there is a short exact sequence
\begin{equation}\label{char} 1\to \bar F^\times \to \bar F(X)^{(B)} \to \varchi(X)\to 1,\end{equation}
and the valuations are trivial on $\bar F^\times$. Since, in the case of $G$ being split, the Galois group acts trivially on $\varchi(X)$, these valuations are Galois stable --- in particular, Galois-conjugate colors give rise to the same valuation.

\begin{proposition}\label{propositionstrict}
A homogeneous spherical $G$-variety $X$ is factorizable if and only if the following two conditions are satisfied:
\begin{enumerate}
\item $\varchi(X)_\QQ^{W_X} \subset \varchi(A)_\QQ^W$;
\item the set $\rho_X(\mathcal D)$ of valuations induced by colors lies in the subspace of $\varchi(X)_\QQ^*$ spanned by the images of coroots of $G$ under the quotient map $\varchi(A)_\QQ^*\to \varchi(X)_\QQ^*$.
\end{enumerate}

A wavefront homogeneous spherical variety is strongly factorizable if and only if the following two conditions are satisfied:
\begin{enumerate}
\item For every subset $\Theta$ of the set $\Delta_X$ of spherical roots, 
$$\varchi(X)_\QQ^{W_{X_\Theta}} \subset \varchi(A)_\QQ^{W_{L_\Theta}}.$$
(Recall that the little Weyl group $W_{X_\Theta}$ of $X_\Theta$ is generated by the simple reflections associated to elements of $\Theta$; by $W_{L_\Theta}$ we denote the Weyl group of the Levi subgroup $L_\Theta$.)
\item For every color $D$, $\rho_X(D)$ is a multiple of the image of a simple coroot $\check\alpha$ of $G$.
\end{enumerate}
\end{proposition} 

Before we prove the proposition, we make some remarks, prove some lemmas and give some examples that clarify its use.

\begin{remarks}
\begin{enumerate}
\item In terms of the dual groups, the first condition means that the center of the dual group of $X$ is, up to finite indices, contained in the center of the dual group of $G$, and similarly for all Levi varieties (for strong factorizability). Notice that the dual group\footnote{We are using here the Gaitsgory-Nadler dual group, denoted by $\check G_{X,GN}$ in \cite{SV}, assuming that its root datum is the one corresponding to the spherical roots --- see \emph{loc.cit.}, Theorem 2.2.3.} $\check G_X$ of $X$ is a canonical subgroup of the dual group $\check G$ of $G$, if the Tannakian construction of Gaitsgory-Nadler is assumed, or a canonical subgroup up to conjugacy by the canonical maximal torus, if a combinatorial definition based on the set $\Delta_X$ of spherical roots is used. The dual $\check L_\Theta$ of the Levi $L_\Theta$ is determined by the set of simple roots of $G$ in the support of elements of $\Theta$, and the simple roots in the Levi of $P(X)$. Hence, this is a condition that can be easily checked once the spherical roots of $X$ and the parabolic $P(X)$ are known.

\item The second condition (in both cases) cannot be read off from dual groups. It requires more specific knowledge on colors, which can be obtained from the Luna diagram \cite{Lu} of the spherical variety. This condition, for strong factorizability, eventually boils down to a determination of valuations for ``type $T$'' colors, i.e.\ colors attached\footnote{A color $D$ is attached to the simple root $\alpha$ of $G$ if it belongs to $\mathring X\cdot P_\alpha$, where $P_\alpha$ is the parabolic of semisimple rank $1$ associated to $\alpha$.} to simple roots of $G$ belonging to (the set of unnormalized spherical roots)\footnote{Those are the spherical roots as used by Luna. They are multiples of elements of $\Delta_X$.} $\Sigma_X$. More precisely: 
\end{enumerate}
\end{remarks}

\begin{lemma}
The second condition for strong factorizability is satisfied if and only if for every $\alpha \in\Sigma_X$ which is also a simple root for $G$, the two colors contained in $\mathring X\cdot P_\alpha$ induce valuations equal to $\frac{\check\alpha}{2}$.
\end{lemma}

Such colors are called \emph{undetermined} in \cite[\S 5]{KnAu}.

\begin{proof}
Indeed, for every color $D$ there exists at least one $\alpha$ such that $D\subset \mathring X P_\alpha$, and then $\rho_X(D) = $ the image of $\check\alpha$ (hence the second condition of strong factorizability is satisfied) \emph{except} when $\alpha\in \Sigma_X$, cf.\ \cite{Lu}. In this last case, consider the boundary degeneration of rank one associated to $\Theta = \{\alpha\}$. The Levi $L_\Theta$ has simple roots $\{\alpha\}\cup S^p_X$, where, using Luna's notation, $S^p_X$ denotes the simple roots in the Levi of $P(X)$. But it is known that if $\beta\in S^p_X$ then the image of the coroot $\check\beta$  in $\varchi(X)_\QQ^*$ is zero. Hence, $\rho_X(D)$ has to be a multiple of (the image of) $\check\alpha$, and then it has to be equal to $\frac{\check\alpha}{2}$.
(Cf.\ \emph{loc.cit.} for all the facts we are using.) 
\end{proof}

Moreover, it has been proven by Losev \cite[Theorem 2]{Lo} that for $\alpha\in \Sigma_X$ as in the previous lemma, the spherical variety $X_N = N(H)\backslash G = X/\Aut_G(X)$ (where $X=H\backslash G$ and $N(H)$ denotes its normalizer), the spherical root $\alpha$ gets replaced by $2\alpha$. That means that, up to dividing by the $G$-automorphism group or a suitable subgroup thereof, $X$ is a spherical variety whose spherical system has no simple roots, and those have been classified, along with ``strict spherical varieties'' by \cite{BCF}. Most of those are symmetric, and among the non-symmetric ones some are not of wavefront type or do not satisfy the first condition of Proposition \ref{propositionstrict}, but there are some examples that do:

\begin{example} \label{exA1}
Let $X=G_2\backslash \SO_7$. Its spherical system is denoted $b'''(3)$ in \cite{BCF}. With simple roots labelled consecutively (on the Dynkin diagram) as $\alpha_1,\alpha_2,\alpha_3$, where $\alpha_3$ is the short root, we have $S^p_X = \{ \alpha_1,\alpha_2\}$ and $\Sigma_X = \{\alpha_1+2\alpha_2+3\alpha_3\}$.

The variety $X$ is factorizable for trivial reasons: $\mathcal Z(X)$ is trivial. From its spherical system it can be deduced that its only boundary degeneration is the unique horospherical homogeneous variety $X_\emptyset$ with $S^p_{X_\emptyset}=S^p_X$ and character group  spanned by $\alpha_1+2\alpha_2+3\alpha_3$. Thus, $X_\emptyset = \SL_3 U\backslash G$, where $U$ is the unipotent radical of the parabolic with Levi of type $\GL_3$, and  $\SL_3$ is the derived group of this Levi. Hence, the Levi variety is $\SL_3\backslash \GL_3$, which is factorizable.
\end{example}

\begin{example} \label{exA2}
Let $X=H\backslash G$, where $G$ is the exceptional group $G_2$ and $H=\SL_3$; its spherical system is denoted $g(2)$ in \emph{loc.cit.} Here $S^p_X=\{\alpha_2\}$ (the long root) and $\Sigma_X =\{ 2\alpha_1+\alpha_2\}$. Again, $\mathcal Z(X)=1$, and from the spherical system it can be deduced that the only boundary degeneration  $X_\emptyset$  is isomorphic to the quotient of $G$ by the subgroup $\SL_2\cdot U$, where $U$ is the unipotent radical of the parabolic whose Levi has root $\alpha_2$, and $\SL_2$ belongs to that Levi. Thus, $X_\emptyset^L = \SL_2\backslash \GL_2$, which is factorizable.
\end{example}

Now we come to the proof of Proposition \ref{propositionstrict}.

\begin{proof}[Proof of Proposition \ref{propositionstrict}]

Consider the diagram of natural morphisms of tori:
$$ \xymatrix{ 
\mathcal Z(X) \ar[r] &  A_X \ar[r] & X^\ab \\
\mathcal Z(G)^0 \ar[r]\ar[u] & A \ar[r]\ar[u] & G^\ab \ar[u].}$$

In terms of character groups we have a dual diagram:
$$ \xymatrix{ 
\varchi(X)\otimes \QQ/\left<\Sigma_X\right>_{\QQ} \ar[d]& \ar[l] \varchi(X)\otimes \QQ \ar[d]  & \ar[l]\varchi(X^\ab)\otimes \QQ \ar[d]\\
\varchi(A)\otimes \QQ/\left<\Phi\right>_{\QQ} & \ar[l]\varchi(A)\otimes \QQ & \ar[l] \varchi(G^\ab) \otimes\QQ .}$$

The vertical arrows are injective, and so are the horizontal arrows on the right.  Thus, for an element $\chi\in \varchi(X)\otimes \QQ$ to come from $\varchi(X^\ab)\otimes \QQ$, first of all it has to come from $\varchi(G^\ab)\otimes\QQ$; thus, it has to be orthogonal to the subspace of $\Hom(\varchi(A),\QQ)$ spanned by coroots of $G$. Granted that, and assuming without loss of generality that $\chi\in \varchi(X)$, we need to make sure that such a character is trivial on the subgroup $H$ stabilizing a point on the homogeneous variety $X$. This is \emph{equivalent} to saying that $\chi$, thought of as a function on the open Borel orbit (uniquely determined up to scalar by \eqref{char}), extends to a \emph{nonvanishing, regular function} on the spherical variety.  (The equivalence is established by pulling it back to $G$ and using the fact that  a regular, non-vanishing function on $G$ is necessarily a scalar multiple of a character.)  Which, in turn, is equivalent to saying that its valuations on all colors are trivial. Thus:

\begin{equation} \varchi(X^\ab)\otimes \QQ = \varchi(X)\otimes\QQ \cap \overline{\check \Phi}^\perp \cap \rho_X(\mathcal D)^\perp,\end{equation}
where $\overline{\check\Phi}$ denotes the image of coroots of $G$.
Thus, $X$ is factorizable if and only if the span of colors and coroots of $G$ in $\Hom(\varchi(X),\QQ)$ has the same dimension as the span of $\Sigma_X$ in $\varchi(X)_\QQ$. 

This condition can be reformulated, taking into account that there is a Weyl group action on $\varchi(X)$, under which the quotient $\varchi(X)\otimes \QQ/\left<\Sigma_X\right>_{\QQ}$ can be identified with the subspace of $W_X$-fixed vectors. This, in turn, contains the subspace $\varchi(X)\otimes\QQ \cap \overline{\check \Phi}^\perp$. (Notice that, in terms of dual groups, this containment corresponds the embedding into the center of the dual group of $X$ of its intersection with the center of the dual group of $G$.) Thus, a variety is factorizable if and only if the $W_X$-fixed subspace of $\varchi(X)\otimes\QQ $ belongs to the $W$-fixed subspace of $\varchi(A)\otimes \QQ$ \emph{and}, moreover, the valuations induced by colors, i.e.\ the set $\rho_X(\mathcal D)$, are in the subspace spanned by the images of coroots of $G$.

For a Levi variety, the lattice $\varchi(X)$ does not change, the set of spherical roots is a subset of $\Sigma_X$, and the set of colors can be identified with the subset of those $D\in \mathcal D$ with $\left<\gamma,\rho_X(D)\right> >0$ for some $\gamma$ in that subset of $\Sigma_X$ (cf.\ \cite{Lu}). The first condition of strong factorizability follows directly from the first condition of factorizability, and the second follows from considering boundary degenerations with a unique spherical root, and clearly suffices for all other boundary degenerations. 

\end{proof}

Finally, we check that the two examples \ref{exA1}, \ref{exA2} of non-symmetric, strongly factorizable varieties that we saw satisfy the rest of the assumptions of this paper (s.\ \S \ref{ssassumptions}).

In both cases, the character group $\varchi(X)$ is generated by the spherical root, hence is of rank one, and the spherical root is a root of the group. Hence, we have $\mathfrak a_{X,\Theta}^* = $ either the line generated by the spherical root (when $\Theta =\emptyset$) or the trivial space $\{0\}$ (when $\Theta = \Delta_X$). In the non-trivial case, since $\mathfrak a_{X,\Theta}^*$ is one-dimensional, an element of the Weyl group that leaves it invariant either acts trivially on it or acts by $(-1)$; hence, both actions are represented by elements of the little Weyl group $W_X$. Thus, the strong generic injectivity assumption is fulfilled.

We now sketch the argument for the validity of the explicit Plancherel formula of \cite[Theorem 15.6.2]{SV}, by checking the ``generic injectivity of small Mackey restriction'' (we point the reader to \emph{loc.cit}.\ for the definitions). This has to do with representations appearing in the continuous spectrum of $X$, i.e.\ in the spectrum of (the unique boundary degeneration) $X_\emptyset$. In this case, since $X_\emptyset$ is horospherical, the representations will have the form $I_{P(X)}(\chi)$, where $\chi$ is a character that is trivial on the intersection of $P(X)$ with the stabilizer of a point on $X_\emptyset$ belonging to the opposite parabolic. 

We need to show that, for generic such $\chi$, any morphism:
$$ \mathcal S(X)\to I_{P(X)}(\chi)$$
is obtained by the analytic continuation of the functional ``integration over the open $P(X)$-orbit'' (against the character $\chi^{-1}\delta_{P(X)}^{-\frac{1}{2}}$). To show this, we argue that other $P(X)$-orbits (or, for that matter, other Borel orbits) cannot support such a $P(X)$-equivariant (or Borel-equivariant) distribution. For unramified characters, this has been done in \cite{SaSpc}, but the same argument works in general. It is enough to show that no $B$-orbit other than the open one has character group (= the group of characters of the Borel subgroup which are trivial on stabilizers of points on this orbit) different from the open orbit (whose character group coincides with $\varchi(X)$ and hence, in our examples, is generated by the unique spherical root). By \cite{KnOrbits}, the rank of the character group of each Borel orbit is  at most equal to that rank for the open orbit, which in our case is $1$; and all orbits of maximal rank (in our case, rank $1$) are conjugate under an action of the full Weyl group $W$ defined by F.\ Knop, which is compatible with the action of $W$ on character groups (considered as subgroups of $\varchi(A)$, the character group of the Borel). The stabilizer of the open orbit under this action is the product $W_X\ltimes W_{L(X)}$, where $L(X)$ is the Levi of $P(X)$, so the whole problem boils down to checking that the stabilizer of $\varchi(X)\subset \varchi(A)$ in $W$ is equal to $W_X\ltimes W_{L(X)}$ in our examples. This is indeed the case: Let $w\in W$ stabilize $\varchi(X)$. Without loss of generality (multiplying, if necessary, by the non-trivial element of $W_X$), $w$ acts trivially on $\varchi(X)$. But this places $w$ in the centralizer of the dual torus $\check A_X$ (the torus with cocharacter group $\varchi(X)$) in the dual group $\check G$ of $G$, a Levi of $\check G$ which by \cite[Lemma 3.1]{KnAs} is the Levi dual to $P(X)$; thus, $w\in W_{L(X)}$.

\bibliographystyle{alphaurl}
\bibliography{biblio}

\end{document}